\documentclass[final]{elsarticle}
\usepackage{amssymb}
\usepackage{algorithm} 
\usepackage{algpseudocode} 
\usepackage{multicol, multirow}
\usepackage{ltablex}

\usepackage{amsmath, array, bm}
\usepackage{graphicx}
\usepackage{subcaption}
\usepackage{float}
\usepackage{xcolor}
\usepackage{hyperref}
\usepackage[title]{appendix}

\usepackage{mathtools, amsthm, graphics, array}

\newtheorem{theorem}{Theorem}

\newtheorem{lemma}{Lemma}
\newtheorem{remark}{Remark}
\newtheorem{assumption}{Assumption}
\newtheorem{corollary}{Corollary}

\journal{}

\usepackage{amsmath,amssymb,amsfonts}%
\usepackage{mathrsfs}%
\usepackage{xcolor}%
\usepackage{textcomp}%
\usepackage{manyfoot}%
\usepackage{booktabs}%
\usepackage{algorithm}%
\usepackage{algorithmicx}%
\usepackage{algpseudocode}%
\usepackage{listings}%
\usepackage{stmaryrd}

\numberwithin{equation}{section}  

\usepackage{chngcntr}
\counterwithin{figure}{section}
\counterwithin{table}{section}

\def\vavg#1{\{\!\!\{#1\}\!\!\}}   
\def\norm#1{\hspace{0.1em}|\!\hspace{0.2em}\!|\!\hspace{0.2em}\!|\hspace{0.05em}#1\hspace{0.05em}|\!\hspace{0.2em}\!|\!\hspace{0.2em}\!|\hspace{0.1em}_h}  
\def\nnorm#1{\hspace{0.1em}|\!\hspace{0.2em}\!|\!\hspace{0.2em}\!|\hspace{0.05em}#1\hspace{0.05em}|\!\hspace{0.2em}\!|\!\hspace{0.2em}\!|\hspace{0.1em}}
\def\norma#1{\hspace{0.1em}|\!\hspace{0.2em}\!|\!\hspace{0.2em}\!|\hspace{0.05em}#1\hspace{0.05em}|\!\hspace{0.2em}\!|\!\hspace{0.2em}\!|\hspace{0.1em}_a}

\begin{document}

\begin{frontmatter}

\title{A Streamline Upwind/Petrov-Galerkin Method for the Magnetic Advection-Diffusion Problem
}

\author[1]{Haochen Li} 
\ead{2401110047@stu.pku.edu.cn}
\author[1]{Yangfan Luo}
\ead{luoyf@stu.pku.edu.cn}
\author[2]{Jindong Wang}
\ead{jindong.wang@kaust.edu.sa}
\author[1]{Shuonan Wu\corref{cor1}}
\cortext[cor1]{Corresponding author. Email: snwu@math.pku.edu.cn}

\affiliation[1]{organization={School of Mathematical Sciences, Peking University},
            city={Beijing},
            postcode={100871}, 
            country={China}}

\affiliation[2]{organization={Computer, Electrical and Mathematical Science and Engineering Division,King Abdullah University of Science and Technology},
            city={Thuwal},
            postcode={23955},
            country={Saudi Arabia}}

\begin{abstract}
{
This paper presents the development and analysis of a streamline upwind/Petrov-Galerkin (SUPG) method for the magnetic advection-diffusion problem. A key feature of the method is an SUPG-type stabilization term based on the residuals and weighted advection terms of the test function. By introducing a lifting operator to characterize the jumps of finite element functions across element interfaces, we define a discrete magnetic advection operator, which subsequently enables the formulation of the desired SUPG method. Under mild assumptions, we establish the stability of the scheme and derive optimal error estimates. 
Numerical examples in both two and three dimensions are provided to demonstrate the theoretical convergence and stabilization properties of the proposed method. }
\end{abstract}


\end{frontmatter}

\section{Introduction}\label{sect:intro}

Magnetohydrodynamics (MHD) \citep{gerbeau2006mathematical} provides a fundamental framework for modeling the behavior of electrically conducting fluids under the influence of magnetic fields. It plays a critical role in a wide range of scientific and engineering applications, including nuclear fusion devices, astrophysical simulations, and laboratory plasma experiments.

The full MHD system couples the Navier–Stokes equations of fluid dynamics with Maxwell’s equations of electromagnetism, resulting in a nonlinear and computationally challenging model, particularly when both the hydrodynamic and magnetic Reynolds numbers are high. Due to this complexity, it is often useful to study simplified subsystems that capture essential physical mechanisms while remaining tractable to analysis and simulation. While significant progress has been made in understanding the behavior under high hydrodynamic Reynolds numbers \citep{john2017divergence}, the regime of high magnetic Reynolds numbers remains comparatively less explored. Therefore, one simplification strategy is to assume a known fluid velocity $\bm{v}$, which decouples the Navier-Stokes equations and yields the following reduced system \citep{gerbeau2006mathematical}:
$$
\left\{\begin{aligned}
    \bm{j} - R_m^{-1}\nabla\times(\mu_r^{-1}\bm{B}) &= 0,   \\
    \bm{B}_t + \nabla\times\bm{E} &= 0, \\
    \bm{j} - \sigma_r(\bm{E}+\bm{v}\times\bm{B}) &= 0,\\
    \nabla\cdot\bm{B} &= 0.
\end{aligned}\right.
$$
Physically, the dimensionless quantities $\bm{E}$ and $\bm{B}$ correspond to the electric and magnetic fields within a conductor moving with the given velocity $\bm{v}$. Due to the divergence-free nature of the magnetic field $\bm{B}$, it is possible to introduce a magnetic vector potential $\bm{A}$, such that $\bm{B}=\nabla\times\bm{A}$. Using the identity $\nabla\times(\nabla(\cdot))\equiv \bm{0}$, we have
$$
\nabla\times\bm{E} = -\bm{B}_t = \nabla\times(-\bm{A}_t-\nabla(\bm{v}\cdot\bm{A})),
$$
thus $\bm{E} = -\bm{A}_t-\nabla(\bm{v}\cdot\bm{A})-\nabla\phi$, where $\nabla\phi$ represents the undetermined gauge freedom. The elimination of $\bm{j}$ and $\bm{B}$ leads to
$$
R_m^{-1}\sigma_r^{-1}\nabla\times(\mu_r^{-1}\nabla\times\bm{A}) = \bm{E} + \bm{v}\times(\nabla\times\bm{A}).
$$
After discretizing the temporal derivative $\bm{A}_t$ and setting $\phi = 0$, we obtain the following magnetic advection-diffusion equation for the magnetic vector potential $\bm{A}$:
\begin{equation}\label{magnetic_advection_diffusion_equation_for_magnetic_vector_potential}
    R_m^{-1}\sigma_r^{-1}\nabla\times(\mu_r^{-1}\nabla\times\bm{A}) - \bm{v}\times(\nabla\times\bm{A}) + \nabla(\bm{v}\cdot\bm{A}) + \gamma\bm{A} = \bm{f}.
\end{equation}

Despite being a simplification, the magnetic advection-diffusion equation \eqref{magnetic_advection_diffusion_equation_for_magnetic_vector_potential} still presents significant challenges for accurate discretization due to the advection-dominated nature of the problem in high magnetic Reynolds number regimes, which are characteristic of many astrophysical and fusion applications. Consequently, several studies have focused on developing robust numerical schemes for this fundamental equation, as exemplified by \citep{heumann2013stabilized, wang2023exponentially, wang2024discontinuous}.


To align with common notational conventions, we now abstract the physical equation \eqref{magnetic_advection_diffusion_equation_for_magnetic_vector_potential} by introducing $\bm{u} := \bm{A}$ and $\bm{\beta} := \bm{v}$. We therefore focus on the following boundary value problem:
\begin{equation}\label{curl_advection_diffusion_equation}
    \left\{
    \begin{aligned}
        \nabla\times(\varepsilon\nabla\times\bm{u}) - \bm{\beta}\times(\nabla\times\bm{u}) + \nabla(\bm{\beta}\cdot\bm{u}) + \gamma\bm{u} &= \bm{f} ~~~~\text{in}~~\Omega,\\
        \bm{n}\times\bm{u} + \chi_{\Gamma^-}(\bm{u}\cdot\bm{n})\bm{n} &= \bm{g}~~~~\text{on}~~\Gamma,
    \end{aligned}
    \right.
\end{equation}
where $\bm{f}\in(L^2(\Omega))^3$, $\bm{g}\in(H^{1/2}(\Gamma))^3$, $\Omega\subset\mathbb{R}^3$ is a bounded, Lipschitz, polyhedral domain with boundary $\Gamma=\partial\Omega$. $\bm{\beta}\in (W^{1,\infty}(\Omega))^3$ is the velocity field and $\gamma\in L^{\infty}(\Omega)$ is the reaction coefficient. The diffusion coefficient $\varepsilon$ is a positive constant. The term $-\bm{\beta}\times(\nabla\times\bm{u}) + \nabla(\bm{\beta}\cdot\bm{u})$ in \eqref{curl_advection_diffusion_equation} is the magnetic advection term, also known as the Lie derivative
\begin{equation}
    L_{\bm{\beta}}\bm{u} := - \bm{\beta}\times(\nabla\times\bm{u}) + \nabla(\bm{\beta}\cdot\bm{u}),
\end{equation}
which is a vector counterpart of the scalar advection $\bm{\beta}\cdot\nabla u$. The inflow and outflow parts of $\Gamma$ are defined in the usual fashion
$$
\begin{aligned}
    \Gamma^+ &:= \{{x}\in\Gamma:\bm{\beta}({x})\cdot\bm{n}({x})\ge0\},\\
    \Gamma^- &:= \{{x}\in\Gamma:\bm{\beta}({x})\cdot\bm{n}({x})<0\},
\end{aligned}
$$
where $\bm{n}(x)$ is the unit outward normal vector to $\Gamma$ at $x\in\Gamma$.

One major numerical challenge associated with \eqref{curl_advection_diffusion_equation} is maintaining stability when $\varepsilon$ is small, i.e., in advection-dominated regimes where $\varepsilon \ll \vert\bm{\beta}\vert$. Under such conditions, standard numerical methods often exhibit spurious oscillations or a loss of accuracy. This issue is especially pronounced and has been extensively studied for the closely related scalar advection-diffusion problem
\begin{equation}\label{eq:scalar_advection_diffusion}
-\nabla \cdot (\varepsilon\nabla u) + \bm{\beta} \cdot \nabla u + \gamma u = f.
\end{equation}
 
This challenge has been a central focus in the scalar advection-diffusion setting, where it has given rise to the development of numerous numerical methods. A prominent class of approaches for the scalar advection-diffusion equation \eqref{eq:scalar_advection_diffusion} employs upwinding or streamline-based stabilization to enhance the numerical stability, including: 
the streamline upwind/Petrov-Galerkin (SUPG) method \citep{mizukami1985petrov, franca1992stabilized, chen2005optimal, burman2010consistent}, 
the Galerkin/least-squares finite element method \citep{hughes1989new}, 
bubble function stabilization \citep{brezzi1994choosing, brezzi1998applications, brezzi1998further, brezzi1999priori, franca2002stability}, 
the stabilized discontinuous Galerkin (DG) method \citep{houston2002discontinuous, brezzi2004discontinuous, ayuso2009discontinuous}, 
the discontinuous Petrov-Galerkin (DPG) method \citep{demkowicz2010class,demkowicz2011class}, the local projection stabilization (LPS) method \citep{braack2006local, matthies2007unified, knobloch2010generalization},
the edge stabilization \citep{burman2004edge},
and the continuous interior penalty (CIP) method \citep{burman2005unified, burman2007continuous, burman2009weighted}.
Another class of methods for \eqref{eq:scalar_advection_diffusion} is based on exponential fitting, where exponential functions are incorporated either in the construction of basis functions \citep{dorfler1999uniform, dorfler1999uniform_2, o1991analysis, o1991globally, wang1997novel} or in the assembly of the stiffness matrix \citep{xu1999monotone}.

In recent years, the magnetic advection-diffusion problem \eqref{curl_advection_diffusion_equation} has attracted increasing attention within the numerical analysis community. Exponentially-fitted methods have been proposed in \citep{wu2020simplex, wang2023exponentially}. For upwinding-based stabilization, Heumann and Hiptmair developed specialized upwind techniques for vector-valued problems \citep{heumann2010eulerian, heumann2013stabilized}. A local projection stabilization method was introduced in \citep{Luo2025LocalPS}. In addition, both primal and hybridizable discontinuous Galerkin methods for \eqref{curl_advection_diffusion_equation} have been presented in \citep{wang2024discontinuous, wang2024hybridizable}. Although many existing stabilization techniques have been extended to the vector-valued case, to the best of our knowledge, the streamline upwind/Petrov-Galerkin (SUPG) method, a conforming approach, has not previously been formulated for this magnetic advection-diffusion problem.

The streamline upwind/Petrov-Galerkin (SUPG) method, originally introduced as the streamline-diffusion finite element method (SDFEM) by Hughes and Brooks \citep{hughes1979multidimensional}, achieves stabilization by adding weighted residuals to the bilinear form. A more comprehensive and widely cited theoretical analysis, including the derivation of the design criteria for stabilization parameters, was presented in \citep{brooks1982streamline} and firmly established SUPG as a primary tool for computational fluid dynamics due to its theoretical clarity, computational efficiency, and ease of implementation. Owing to these advantages, the SUPG method has been successfully applied to a wide range of problems in computational fluid dynamics, including incompressible Navier-Stokes equations \citep{TEZDUYAR1992221, brooks1982streamline}, compressible flows \citep{LEBEAU1993397}, and transport equations \citep{burman2010consistent}, among others.

In this paper, we extend the streamline upwind/Petrov-Galerkin (SUPG) method to the magnetic advection-diffusion problem, expanding on the pioneering work of Hughes et al. \citep{hughes1979multidimensional}. Our approach first identifies the discrete magnetic advection operator through a lifting operator that captures finite element jumps across interfaces. This operator is then used to add a corresponding stabilization term defined as the inner product of the residual and a weighted version of it, thus establishing the SUPG method. Under mild assumptions, we prove the stability of the method and derive optimal a priori error estimates. Numerical experiments confirm the theoretical analysis and demonstrate the effectiveness of the method in suppressing spurious oscillations compared to existing approaches.


The remainder of this paper is organized as follows. In Section \ref{sect:pset}, we introduce the necessary preliminaries. Section \ref{sect:SUPGscheme} details the SUPG discretization scheme and its derivation. The stability analysis of the proposed method is established in Section \ref{sect:stability_analysis}. Section \ref{sect:priori} presents the derivation of optimal a priori error estimates. 
Numerical experiments validating the theoretical findings are conducted in Section \ref{sect:numerical_experiment}. Finally, concluding remarks are given in Section \ref{sec:conclusion}.

\section{Preliminaries}\label{sect:pset}
Given a bounded domain $D \subset \mathbb{R}^3$, a positive integer $s$, and $p \in [1, \infty]$, we denote by $W^{s,p}(D)$ the Sobolev space equipped with the standard norm $\Vert\cdot\Vert_{s,p,D}$ and semi-norm $\vert\cdot\vert_{s,p,D}$. For simplicity, we omit the index $p$ when $p = 2$ and the domain $D$ when $D = \Omega$ in the subscript. Specifically,
\[
\|\cdot\|_{s,D} := \|\cdot\|_{s,2,D}, \quad
\|\cdot\|_{s,p} := \|\cdot\|_{s,p,\Omega}, \quad
\|\cdot\|_s := \|\cdot\|_{s,2,\Omega}.
\]
The same convention is adopted for semi-norms. The $L^2$-inner products over $D$ and $\partial D$ are denoted by $(\cdot,\cdot)_D$ and $\langle\cdot,\cdot\rangle_{\partial D}$, respectively. We define the space
$$
\bm{H}(\mathrm{curl};\Omega) :=\{ \bm{v}\in \left(L^2(\Omega)\right)^3 : { \nabla \times} \bm{v}\in \left(L^2(\Omega)\right)^3 \},
$$
along with the $\bm{H}(\mathrm{curl};\Omega)$-trace operator
$$
\gamma_t(\bm{v}) := \bm{v}|_{\partial\Omega}\times\bm{n}.
$$
{For later use, we also introduce the subspace with vanishing tangential trace
$$
\bm{H}_0(\mathrm{curl};\Omega) :=\{ \bm{v}\in\bm{H}(\mathrm{curl};\Omega) : \gamma_t(\bm{v})|_{\partial\Omega}\equiv 0\}.
$$
}
\par
Let $\mathcal{T}_h$ be a shape-regular family of triangulations of $\Omega$ such that each open boundary facet lies entirely in either $\Gamma^{+}$ or $\Gamma^{-}$. Denote by $\mathcal{F}_h$ the corresponding set of all facets. Let $\mathcal{F}_h^{\circ}=\mathcal{F}_h\backslash\partial\Omega$ be the set of interior facets and $\mathcal{F}_h^{\partial}=\mathcal{F}_h\backslash\mathcal{F}_h^{\circ}$ be the set of boundary facets. By construction, every boundary facet belongs either to the inflow or the outflow boundary, thus $\mathcal{F}_h^{\partial} = \mathcal{F}_h^+\cup \mathcal{F}_h^-$, where
$$
\begin{aligned}
&\mathcal{F}_{h}^{\pm} := \{ F\in\mathcal{F}_h^{\partial}: F\subset\Gamma^{\pm} \}.
\end{aligned}
$$
We denote by $h_T$ and $h_F$ the diameters of $T\in\mathcal{T}_h$ and $F\in\mathcal{F}_h$. Shape regularity ensures that $h_T\simeq h_F$. The \(L^2\)-inner product on a facet \(F\) is denoted by \(\langle \cdot, \cdot \rangle_F\).
\par

The formal dual operator of $L_{\bm{\beta}}$ is given by
\begin{equation}
\mathcal{L}_{\bm{\beta}}\bm{v}:=\nabla\times(\bm{\beta}\times\bm{v})-\bm{\beta}\nabla\cdot\bm{v}.
\end{equation}
For any domain $D$, a direct calculation yields
\begin{subequations}\label{identities_Lie_advection}
    \begin{align}
        L_{\bm{\beta}}\bm{u} + \mathcal{L}_{\bm{\beta}}\bm{u} &= -(\nabla\cdot\bm{\beta})\bm{u} + [\nabla\bm{\beta} + (\nabla\bm{\beta})^T]\bm{u}, \label{identities_Lie_advection_1}\\
        (L_{\bm{\beta}}\bm{u},\bm{v})_D&=(\bm{u},\mathcal{L}_{\bm{\beta}}\bm{v})_D + \langle\bm{\beta}\cdot\bm{n},\bm{u}\cdot\bm{v}\rangle_{\partial D}. \label{identities_Lie_advection_2}
    \end{align}
\end{subequations}
{
The weak formulation of \eqref{curl_advection_diffusion_equation} is derived as follows. First, we introduce a function $\tilde{\bm{u}}_b$ whose $\bm{H}(\mathrm{curl})$-trace matches the Dirichlet boundary data of the exact solution $\bm{u}$. Observe that the trace \(\gamma_t(\bm{u})|_{\Gamma}\) is prescribed by the data \(\bm{g}\). The trial and test spaces are then defined as
\begin{equation}
\begin{aligned}
    \bm{\mathcal S}&= \{ \bm v : L_{\bm\beta}\bm v \in \left(L^2(\Omega)\right)^3 \}\cap \bm H_0(\mathrm{curl};\Omega) \\
    &= \{ \bm v : \mathcal L_{\bm\beta}\bm v \in \left(L^2(\Omega)\right)^3 \}\cap \bm H_0(\mathrm{curl};\Omega). \\
\end{aligned}
\end{equation}
The equality of these two expressions for $\bm{\mathcal S}$ follows from identity \eqref{identities_Lie_advection_1} in light of the regularity assumption $\bm{\beta} \in \bigl(W^{1,\infty}(\Omega)\bigr)^3$. Multiplying \eqref{curl_advection_diffusion_equation} by a test function $\bm{v} \in \bm{\mathcal{S}}$ and integrating over $\Omega$ gives
$$
\left(\nabla\times(\varepsilon\nabla\times\bm{u}),\bm v\right) + (L_{\bm \beta}\bm u, \bm v) + (\gamma\bm{u},\bm v) = (\bm{f}, \bm v).
$$
Applying the integration by parts formula yields
$$
(\nabla\times(\varepsilon\nabla\times\bm{u}), \bm v) = (\varepsilon\nabla\times\bm{u},\nabla\times\bm v)+\langle\varepsilon\nabla\times\bm{u}, \bm v\times\bm n\rangle = (\varepsilon\nabla\times\bm{u},\nabla\times\bm v),
$$
where the $L^2$-inner product over $\partial\Omega$ is denoted by $\langle\cdot,\cdot\rangle$ for convenience. Thus, we arrive at the following form
$$
\left(\varepsilon\nabla\times\bm{u},\nabla\times\bm v\right) + (L_{\bm \beta}\bm u, \bm v) + (\gamma\bm{u},\bm v) = (\bm f, \bm v).
$$
To ensure the coercivity of the bilinear form, further modifications are required. A standard technique is to apply the integration by parts formula \eqref{identities_Lie_advection_2} to the term $\frac{1}{2}(L_{\bm{\beta}}\bm{u},\bm{u})$:
\begin{equation}\label{c_weak_form_coer_1}
\begin{aligned}
    (L_{\bm \beta}\bm u,\bm u) + (\gamma\bm u,\bm u) &= \frac{1}{2}(L_{\bm \beta}\bm u,\bm u) +  \frac{1}{2}(L_{\bm \beta}\bm u,\bm u) + (\gamma\bm u,\bm u) \\
    &= \frac{1}{2}(L_{\bm \beta}\bm u,\bm u) +  \frac{1}{2}(\bm u,\mathcal L_{\bm \beta}\bm u) + \frac{1}{2}\langle\bm\beta\cdot\bm n,\bm u\cdot\bm u\rangle + (\gamma\bm u,\bm u) \\
    &=\left( (\frac{1}{2}L_{\bm \beta}+\frac{1}{2}\mathcal L_{\bm \beta}+\gamma I)\bm u,\bm u \right) + \frac{1}{2}\langle\bm\beta\cdot\bm n,\bm u\cdot\bm u\rangle_{\Gamma^+} + \frac{1}{2}\langle\bm\beta\cdot\bm n,\bm u\cdot\bm u\rangle_{\Gamma^-}.
\end{aligned}
\end{equation}
The presence of the term $\left( \left(\frac{1}{2}L_{\bm{\beta}} + \frac{1}{2}\mathcal{L}_{\bm{\beta}} + \gamma I\right)\bm{u}, \bm{u} \right)$ leads to the following assumption:
\begin{assumption}[Friedrichs system\cite{friedrichs1958symmetric}]\label{well_posedness_assumption}   
There exists a positive constant $\rho_0$ such that
\begin{equation}\label{well_posedness_assumption_formulation}
    \rho(x):=\lambda_{\min}\left[(\gamma-\frac{\nabla\cdot\bm{\beta}}{2})I+\frac{\nabla\bm{\beta}+(\nabla\bm{\beta})^T}{2}\right]\ge\rho_0>0,~~\forall x\in\Omega,
\end{equation}
where $\lambda_{\min}$ is the smallest eigenvalue of the corresponding matrix.
\end{assumption}
Furthermore, to counteract the negative term $\frac{1}{2}\langle\bm{\beta}\cdot\bm{n},\bm{u}\cdot\bm{u}\rangle_{\Gamma^-}$, it is necessary to incorporate the term $-\langle \bm\beta\cdot\bm n,\bm u\cdot\bm v \rangle_{\Gamma^-}$ into the bilinear form. Under the condition that $\bm{v}|_{\partial \Omega} \times \bm{n} \equiv \bm{0}$, we derive
\begin{equation}
\begin{aligned}
    &\langle \bm\beta\cdot\bm n,\bm u\cdot\bm v \rangle_{\Gamma^-} = \langle \bm\beta\cdot\bm n,(\bm u\cdot \bm n)\cdot(\bm v\cdot\bm n) \rangle_{\Gamma^-} \\
    =& \langle \bm\beta\cdot\bm n,(\bm g\cdot \bm n)\cdot(\bm v\cdot\bm n) \rangle_{\Gamma^-} =  \langle \bm\beta\cdot\bm n,\bm g\cdot\bm v \rangle_{\Gamma^-}.
\end{aligned}
\end{equation}
This leads to the following weak formulation of \eqref{curl_advection_diffusion_equation}: Find $\bm u\in \tilde{\bm u}_b + \bm{\mathcal S}$, such that
\begin{equation}\label{curl_advection_diffusion_equation_weakform}
    \begin{aligned}
        a(\bm u,\bm v) = f_0(\bm v),\quad\forall \bm v\in \bm{\mathcal S},
    \end{aligned}
\end{equation}
where
\begin{subequations}
    \begin{align}
        a(\bm u,\bm v) &= \left(\varepsilon\nabla\times\bm{u},\nabla\times\bm v\right) + (L_{\bm \beta}\bm u, \bm v) + (\gamma\bm{u},\bm v)-\langle \bm\beta\cdot\bm n,\bm u\cdot\bm v \rangle_{_{\Gamma^-}},\label{curl_advection_diffusion_equation_weakform_bilinear}\\
        f_0(\bm v) &= (\bm f,\bm v) - \langle\bm\beta\cdot\bm n,\bm g\cdot\bm v\rangle_{\Gamma^-}.
    \end{align}
\end{subequations}
Combining \eqref{c_weak_form_coer_1} and \eqref{well_posedness_assumption_formulation}, we can deduce that
\begin{equation}\label{c_weak_form_coer_formulation}
a(\bm v,\bm v) \ge \varepsilon\Vert \nabla\times\bm v \Vert_0^2 + \rho_0\Vert \bm v \Vert_0^2 + \frac{1}{2}\langle\vert\bm\beta\cdot\bm n\vert,\vert\bm v\vert^2\rangle_{\Gamma} =: \norma{\bm v}^2.
\end{equation}

\begin{remark}
    We claim that $\norma{\cdot}$ does not constitute a complete norm on the space $\bm{\mathcal{S}}$. Moreover, due to the lack of control over the term $(L_{\bm{\beta}}\bm{u},\bm{v})$, the bilinear form is not bounded in the $\norma{\cdot}$-norm. Consequently, the Lax-Milgram theorem cannot be applied directly to establish the well-posedness of the weak formulation \eqref{curl_advection_diffusion_equation_weakform}. A detailed analysis of this issue lies beyond the scope of this work. We therefore assume the existence of a solution to \eqref{curl_advection_diffusion_equation_weakform} as an additional hypothesis.
\end{remark}

}

\par

For an interior facet \(F \in \mathcal{F}_h^{\circ}\) shared by two elements, we arbitrarily label one as \(T^{+}\) and the other as \(T^{-}\). Let \(\bm{n}^{\pm}\) denote the unit outward normal vectors of \(T^{\pm}\). The jump and average operators are defined as
$$
\llbracket\bm{v}\rrbracket = \bm{v}^+ - \bm{v}^-, \vavg{\bm{v}} = \frac{1}{2}(\bm{v}^+ + \bm{v}^-).
$$
Moreover, we define the weighted average as  
\[
\vavg{\bm{v}}_{\alpha} = \alpha^+\bm{v}^+ + \alpha^-\bm{v}^-,
\]  
where the weight function \(\alpha\) is a two-valued function of order \(\mathcal{O}(1)\) on all interior facets. Specifically, for each point \({x}\) on an interior facet, \(\alpha\) takes one value \(\alpha^+({x})\) associated with \(T^+\) and another \(\alpha^-({x})\) associated with \(T^-\), such that \(\alpha^+({x}) + \alpha^-({x}) = 1\) for all \({x} \in \bigcup_{F \in \mathcal{F}_h^{\circ}} F\). For a boundary facet \(F \in \mathcal{F}_h^{\partial}\), let $\bm{n}$ denote the unit outward normal vector of $\Omega$, $\bm{n}^+:=\bm{n}$ and define  
\[
\llbracket\bm{v}\rrbracket = \bm{v}, \quad \vavg{\bm{v}}_{\alpha} = \alpha \bm{v},
\]  
where the boundary values of \(\alpha\) are given by \(\alpha|_{\Gamma^+} = 0\) and \(\alpha|_{\Gamma^-} = 1\). Note that \(\alpha\) can also be interpreted as a function defined on the entire mesh skeleton \(\bigcup_{T \in \mathcal{T}_h} \partial T\). In such cases, we write \(\alpha|_{\partial T}\) to emphasize its restriction to the boundary of an individual element \(T\). 
\par
Several useful identities for the weighted average operators are readily verified:
\begin{lemma}[weighted average identities]\label{weighted_average_identities_lemma}
For any two vectors $\bm{v}_1, \bm{v}_2$, the following identities hold:
\begin{subequations}\label{identities_weighted_average}
    \begin{align}
        \vavg{\bm{v}_1}_{\alpha} &= \vavg{\bm{v}_1} + \frac{[\alpha]}{2}\llbracket\bm{v}_1\rrbracket, \label{identities_weighted_average_1} \\
        \llbracket\bm{v}_1\cdot\bm{v}_2\rrbracket &= \vavg{\bm{v}_1}_{\alpha}\cdot\llbracket\bm{v}_2\rrbracket + \llbracket\bm{v}_1\rrbracket\cdot\vavg{\bm{v}_2}_{1-\alpha}, \label{identities_weighted_average_2} \\
        \frac{1}{2}\llbracket\bm{v}_1\cdot\bm{v}_1\rrbracket &= \llbracket\bm{v}_1\rrbracket\cdot\vavg{\bm{v}_1}_{\alpha} - \frac{1}{2}[\alpha]\vert\llbracket\bm{v}_1\rrbracket\vert^2.\label{identities_weighted_average_3}
    \end{align}
\end{subequations}
Here $[\alpha]:=\alpha^+-\alpha^-$.
\end{lemma}

\begin{proof}

{
Given that $\alpha^++\alpha^-=1$, we obtain
\begin{equation}
\begin{aligned}
    \vavg{\bm v_1}_{\alpha} =& \frac{\alpha^+\bm v_1^+ + \alpha^-\bm v_1^-}{2} = \frac{\bm v_1^++\bm v_1^-}{2} + (\alpha^+-\frac{1}{2})\bm v_1^+ + (\alpha^--\frac{1}{2})\bm v_1^- \\
    =& \vavg{\bm v_1} + (\alpha^+-\frac{\alpha^++\alpha^-}{2})\bm v_1^+ + (\alpha^--\frac{\alpha^++\alpha^-}{2})\bm v_1^- =  \vavg{\bm v_1}  + \frac{\alpha^+-\alpha^-}{2}\bm v_1^+ - \frac{\alpha^+-\alpha^-}{2}\bm v_1^-\\
    =&\vavg{\bm v_1}  +\frac{[\alpha]}{2}\llbracket\bm v_1\rrbracket,
\end{aligned}
\end{equation}
which establishes \eqref{identities_weighted_average_1}.} Identity \eqref{identities_weighted_average_2} can be verified through direct computation:
\begin{equation}
\begin{aligned}
&\vavg{\bm{v}_1}_{\alpha}\cdot\llbracket\bm{v}_2\rrbracket + \llbracket\bm{v}_1\rrbracket\cdot\vavg{\bm{v}_2}_{1-\alpha} = (\alpha^+\bm{v}_1^++\alpha^-\bm{v}_1^-)\cdot(\bm{v}_2^+-\bm{v}_2^-) + (\bm{v}_1^+-\bm{v}_1^-)\cdot(\alpha^-\bm{v}_2^++\alpha^+\bm{v}_2^-) \\
=& \alpha^+\bm{v}_1^+\cdot\bm{v}_2^+-\alpha^+\bm{v}_1^+\cdot\bm{v}_2^-+\alpha^-\bm{v}_1^-\cdot\bm{v}_2^+-\alpha^-\bm{v}_1^-\cdot\bm{v}_2^-+\alpha^-\bm{v}_1^+\cdot\bm{v}_2^++\alpha^+\bm{v}_1^+\cdot\bm{v}_2^--\alpha^-\bm{v}_1^-\cdot\bm{v}_2^+-\alpha^+\bm{v}_1^-\cdot\bm{v}_2^-\\
=&(\alpha^++\alpha^-)\bm{v}_1^+\cdot\bm{v}_2^+ + (-\alpha^++\alpha^+)\bm{v}_1^+\cdot\bm{v}_2^- + (\alpha^--\alpha^-)\bm{v}_1^-\cdot\bm{v}_2^+-(\alpha^++\alpha^-)\bm{v}_1^-\cdot\bm{v}_2^-\\
=&\bm{v}_1^+\cdot\bm{v}_2^+-\bm{v}_1^-\cdot\bm{v}_2^-=\llbracket\bm{v}_1\cdot\bm{v}_2\rrbracket.
\end{aligned}
\end{equation}
Combining \eqref{identities_weighted_average_1} with \eqref{identities_weighted_average_2} yields
\begin{equation}
\begin{aligned}
&\llbracket\bm{v}_1\cdot\bm{v}_1\rrbracket = \llbracket\bm{v}_1\rrbracket\cdot( \vavg{\bm{v}_1}_\alpha + \vavg{\bm{v}_1}_{1-\alpha} )\\
=&\llbracket\bm{v}_1\rrbracket\cdot( \vavg{\bm{v}_1}_\alpha + \vavg{\bm{v}_1}_{\alpha} - [\alpha]\llbracket\bm{v}_1\rrbracket ) = 2\llbracket\bm{v}_1\rrbracket\cdot\vavg{\bm{v}_1}_{\alpha} - [\alpha]\vert\llbracket\bm{v}_1\rrbracket\vert^2,
\end{aligned}    
\end{equation}
This completes the proof of \eqref{identities_weighted_average_3}.

\end{proof}

\par
Let $\bm{V}_h$ be an $\bm{H}(\mathrm{curl})$-conforming finite element space on $\mathcal{T}_h$. We further define
$$
\bm{V}_{h,0} := \{ \bm{v}_h\in\bm{V}_h : \gamma_t(\bm{v}_h)|_{\Gamma}=0 \}.
$$
Given facet weight $\alpha$, we piecewise introduce the lifting operator $\bm{r}_{\alpha}$ as $\bm{r}_{\alpha}|_T: (L^2(\partial T))^3\rightarrow\bm{V}_h|_T$, which satisfies
\begin{equation}\label{lifting_operator_definition}
    \int_T \bm{r}_{\alpha}(\bm{v})\cdot\bm{w}_h \mathrm{d}x = \int_{\partial T} \alpha|_{\partial T}\bm{v}\cdot\bm{w}_h \mathrm{d}s,~~\forall\bm{w}_h\in\bm{V}_h|_T.
\end{equation}
In particular, $\bm{r}_\alpha$ is written as $\bm{r}$ when $\alpha^{\pm}=1/2$ on every interior facet. 
For all $ \bm{u}_h, \bm{v}_h\in\bm{V}_h$, the following identity holds:
\begin{equation}\label{identity_about_lifting_1}
\begin{split}
    &\sum_{F\in\mathcal{F}_h^{\circ}}\langle\bm{\beta}\cdot\bm{n}^+,\llbracket\bm{u}_h\rrbracket\cdot\vavg{\bm{v}_h}_{\alpha}\rangle_F+\sum_{F\in\mathcal{F}_h^-}\langle\bm{\beta}\cdot\bm{n},\bm{u}_h\cdot\bm{v}_h\rangle_F \\
    =& \sum_{T\in\mathcal{T}_h}\langle \bm{\beta}\cdot\bm{n}^+,\alpha|_{\partial T} 
\llbracket\bm{u}_h\rrbracket\cdot\bm{v}_h \rangle_{\partial T} = \sum_{T\in\mathcal{T}_h} ( \bm{r}_{\alpha}(\bm{\beta}\cdot\bm{n}^+\llbracket\bm{u}_h\rrbracket) , \bm{v}_h )_T.
\end{split}
\end{equation}

\par
Let ${P}_r(T)$ be the space of polynomials on $T$ with degree less than or equal to $r$. The following inverse and trace inequalities will be used in the analysis.
\begin{lemma}[inverse inequality]\label{standard_inverse_inequality_lemma}
Let $v_h\in P_r(T)$ for some $r\ge0$. Then there exists a positive constant $C$ independent of $h_T$, such that
\begin{equation}\label{standard_inverse_inequality}
     |v_h|_{1,T}\le Ch_T^{-1}\Vert v_h\Vert_{0,T}.
\end{equation}
\end{lemma}

\begin{lemma}[trace inequality]\label{standard_trace_inequality_lemma}
There exists a  positive constant $C$ depending only on the shape regularity constant, such that for any $v\in H^1(T)$
\begin{equation}\label{standard_trace_inequality}
    \Vert v\Vert_{0,\partial T}^2\le C\Vert v \Vert_{0,T} \left( h_T^{-2}\Vert v \Vert_{0,T}^2+\vert v \vert_{1,T}^2   \right)^{1/2}.
\end{equation}
\end{lemma}

Combining the above two inequalities, we obtain the following lemma.
\begin{lemma}[finite element trace inequality]\label{trace_inequality_lemma}
Let $v_h\in P_r(T)$ for some $r\ge0$. Then there exists a positive constant $C$ independent of $h_T$, such that
\begin{equation}\label{trace_inequality}
    \Vert v_h\Vert_{0,\partial T}\le Ch_T^{-1/2}\Vert v_h\Vert_{0,T}.
\end{equation}
\end{lemma}

\begin{remark}
Set $\bm{w}_h=\bm{r}_{\alpha}(\bm{v})$ in \eqref{lifting_operator_definition} and use the inequality \eqref{trace_inequality}, we can deduce that
\begin{equation}\label{lifting_operator_stability_estimate}
    \Vert\bm{r}_{\alpha}(\bm{v})\Vert_{0,T}\le Ch_T^{-1/2}\Vert\bm{v}\Vert_{0,\partial T},
\end{equation}
where $C$ is a constant independent of $h_T$. 
\end{remark}

\section{SUPG method}\label{sect:SUPGscheme}

This section presents the development of the streamline upwind/Petrov-Galerkin (SUPG) method for the magnetic advection-diffusion problem. Motivated by the scalar formulation, where the test functions are modified elementwise to $v + \delta_T \bm{\beta}\cdot\nabla v$, with $\delta_T \geq 0$ a piecewise constant stabilization parameter \citep{roos2008robust}. We extend this stabilization approach to the vector-valued problem.

\par

We define \(\tilde{\bm{u}}_{bh} \in \bm{V}_h\) as a function whose \(\bm{H}(\mathrm{curl})\)-trace approximates the Dirichlet boundary data of the exact solution \(\bm{u}\). Note that \(\gamma_t(\bm{u})|_{\Gamma}\) can be expressed in terms of the given data \(\bm{g}\). The standard Galerkin formulation for \eqref{curl_advection_diffusion_equation} seeks  $\bm{u}_h\in \tilde{\bm{u}}_{bh} + \bm{V}_{h,0}$ such that 
\begin{equation}\label{standard_Galerkin_method}
a_0(\bm{u}_h,\bm{v}_h) = f_0(\bm{v}_h), ~~\forall\bm{v}_h\in\bm{V}_{h,0},
\end{equation}
where
\begin{subequations}
    \begin{align}
        a_0(\bm{u}_h,\bm{v}_h) =& (\varepsilon\nabla\times\bm{u}_h,\nabla\times\bm{v}_h) + (L_{\bm{\beta},h}\bm{u}_h,\bm{v}_h) + (\gamma\bm{u}_h,\bm{v}_h)\notag \\
        &-\sum_{F\in\mathcal{F}_h^{\circ}}\langle\bm{\beta}\cdot\bm{n}^+,\llbracket\bm{u}_h\rrbracket\cdot\vavg{\bm{v}_h}\rangle_F-\sum_{F\in\mathcal{F}_h^-}\langle\bm{\beta}\cdot\bm{n},\bm{u}_h\cdot\bm{v}_h\rangle_F,\label{standard_Galerkin_method_bilinear_form}\\
        f_0(\bm{v}_h) =& (\bm{f},\bm{v}_h) - \sum_{F\in\mathcal{F}_h^-}\langle\bm{\beta}\cdot\bm{n},\bm{g}\cdot\bm{v}_h\rangle_F.
    \end{align}    
\end{subequations}
Here, $L_{\bm{\beta},h}$ denotes the discrete Lie advection operator, which is evaluated in a piecewise manner on each element of the mesh.
Using the identity about the lifting operator \eqref{identity_about_lifting_1}, we can rewrite $a_0(\bm{u}_h,\bm{v}_h)$ as
\begin{equation}
a_0(\bm{u}_h,\bm{v}_h) = (\varepsilon\nabla\times\bm{u}_h,\nabla\times\bm{v}_h) + (L_{\bm{\beta},h}\bm{u}_h,\bm{v}_h) - (\bm{r}(\bm{\beta}\cdot\bm{n}^+\llbracket\bm{u}_h\rrbracket),\bm{v}_h) + (\gamma\bm{u}_h,\bm{v}_h).
\end{equation}
{Following the same procedure as in the continuous setting, we obtain the Galerkin orthogonality
\begin{equation}
a_0(\bm{u},\bm{v}_h) = f_0(\bm{v}_h), ~~\forall\bm{v}_h\in\bm{V}_{h,0},
\end{equation}
as well as the discrete coercivity
\begin{equation}\label{discrete coercivity standard Galerkin}
a_0(\bm{v}_h,\bm{v}_h) \ge \varepsilon\Vert \nabla\times\bm v_h \Vert_0^2 + \rho_0\Vert \bm v_h \Vert_0^2 + \frac{1}{2}\sum_{F\in\mathcal F_h^\partial}\langle\vert\bm\beta\cdot\bm n\vert,\vert\bm v_h\vert^2\rangle_F,
\end{equation}
which ensures the well-posedness of \eqref{standard_Galerkin_method}. Note that \eqref{standard_Galerkin_method_bilinear_form} includes an additional term,
\(-\sum_{F \in \mathcal{F}_h^{\circ}} \langle \bm{\beta} \cdot \bm{n}^+, \llbracket \bm{u}_h \rrbracket \cdot \vavg{\bm{v}_h} \rangle_F\),
compared to its continuous counterpart \eqref{curl_advection_diffusion_equation_weakform_bilinear}. This term arises because $\bm{V}_h \not\subset \{ \bm v : L_{\bm\beta}\bm v \in \left(L^2(\Omega)\right)^3 \}$, which requires applying the integration-by-parts identity \eqref{identities_Lie_advection_2} elementwise to \(\frac{1}{2}(L_{\bm{\beta}} \bm{v}_h, \bm{v}_h)\) in proving \eqref{discrete coercivity standard Galerkin}. However, this formulation is unstable and exhibits numerical oscillations in the presence of boundary or interior layers, as demonstrated in Figure \ref{standard_Galerkin_bad_sect3}.
\begin{figure}[!htbp]
\centering
\includegraphics[width=0.45\textwidth]{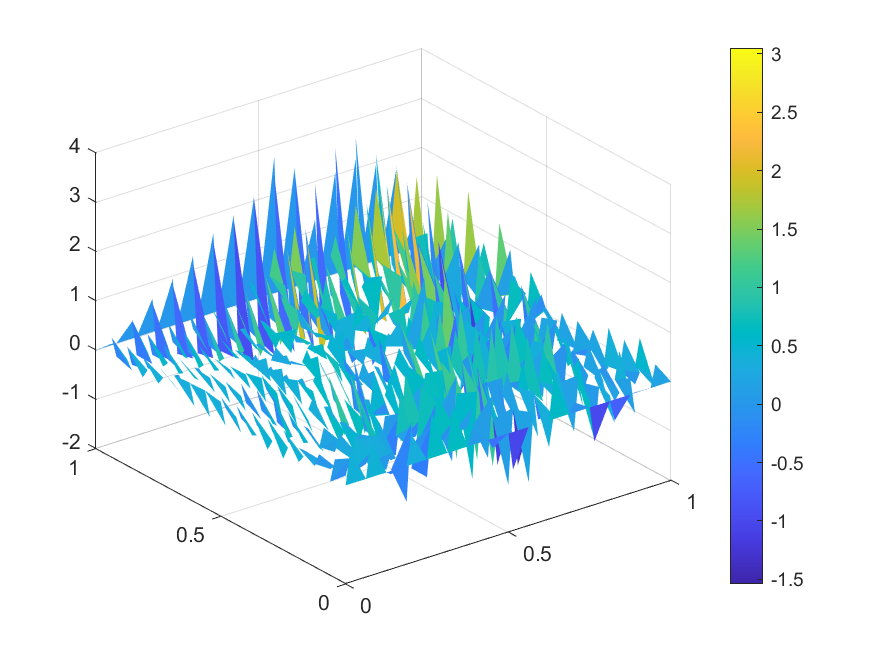} 
\caption{Severe spurious oscillations in the standard Galerkin method \eqref{standard_Galerkin_method}. The computational domain is $(0,1)^2$ with parameters $\varepsilon = 10^{-6}$, $\bm{\beta} = [1, 2]^T$, $\gamma = 0$, $\bm{f} = [1, 1]^T$, and $\bm{g} \equiv [0, 0]^T$. The solution is computed on a uniform mesh of \(2 \times 16^2\) triangles using first-order N\'{e}d\'{e}lec elements of the second kind.} \label{standard_Galerkin_bad_sect3}
\end{figure}
Moreover, even for smooth exact solutions, a reduction in the convergence order is observed in the advection-dominated regime. We provide detailed numerical studies in Section \ref{sect:numerical_experiment}.

Intuitively, the deficiency of the standard Galerkin method originates from the terms $(L_{\bm{\beta},h}\bm{u}_h,\bm{v}_h)$ and $\sum_{F\in\mathcal{F}_h^{\circ}}\langle\bm{\beta}\cdot\bm{n}^+,\llbracket\bm{u}_h\rrbracket\cdot\vavg{\bm{v}_h}\rangle_F$, which are not adequately controlled by the norm $\norma{\cdot}$. This motivates the introduction of corresponding stabilization terms. The derivation of our stabilized SUPG scheme consists of two key steps:

\begin{itemize}
\item \textbf{Step 1:} Introduce the stabilization term
\begin{equation}
    S_h^1(\bm u_h, \bm v_h) = \frac{1}{2}\sum_{F\in\mathcal{F}_h^\circ}\langle-[\alpha]\bm{\beta}\cdot\bm{n}^+,\llbracket\bm{u}_h\rrbracket\cdot\llbracket\bm{v}_h\rrbracket\rangle_F,
\end{equation}
where $\alpha$ is a weight function assumed to satisfy
\begin{equation}\label{alpha_need}
-\bm{\beta}(x)\cdot\bm{n}^+(x)(\alpha^+(x)-\alpha^-(x))\ge C_{up}\lvert\bm{\beta}(x)\cdot\bm{n}(x)\rvert,
\end{equation}
on interior facets for a positive constant $C_{up}$, where ensures that $S_h^1(\bm v_h,\bm v_h)$ is nonnegative. Recall that $\alpha|_{\Gamma^+}=0$ and $\alpha|_{\Gamma^-}=1$ are predefined, which implies that condition \eqref{alpha_need} holds on boundary facets as well. Adding $S_h^1(\cdot,\cdot)$ to the bilinear form $a_0(\cdot,\cdot)$ provides preliminary stabilization. Since the exact solution $\bm{u} \in \{ \bm v : L_{\bm\beta}\bm v \in \left(L^2(\Omega)\right)^3 \}$ implies $\bm\beta\cdot\bm n\llbracket\bm{u}\rrbracket \equiv 0$ on any interior facet due to \eqref{identities_Lie_advection_2}, this modification preserves the consistency of the scheme.

An application of identities \eqref{identities_weighted_average_1} and \eqref{identity_about_lifting_1} yields the following reformulation:
\begin{equation}
\begin{aligned}
    a_0(\bm u_h,\bm v_h) + S_h^1(\bm u_h,\bm v_h) =&(\varepsilon\nabla\times\bm{u}_h,\nabla\times\bm{v}_h) + (L_{\bm{\beta},h}\bm{u}_h,\bm{v}_h) + (\gamma\bm{u}_h,\bm{v}_h) \\
        &-\sum_{F\in\mathcal{F}_h^{\circ}}\left\langle\bm{\beta}\cdot\bm{n}^+,\llbracket\bm{u}_h\rrbracket\cdot(\vavg{\bm{v}_h}+\frac{[\alpha]}{2}\llbracket\bm v_h\rrbracket)\right\rangle_F-\sum_{F\in\mathcal{F}_h^-}\langle\bm{\beta}\cdot\bm{n},\bm{u}_h\cdot\bm{v}_h\rangle_F \\
        =&(\varepsilon\nabla\times\bm{u}_h,\nabla\times\bm{v}_h) + (L_{\bm{\beta},h}\bm{u}_h,\bm{v}_h) + (\gamma\bm{u}_h,\bm{v}_h) \\
        &-\sum_{F\in\mathcal{F}_h^{\circ}}\left\langle\bm{\beta}\cdot\bm{n}^+,\llbracket\bm{u}_h\rrbracket\cdot\vavg{\bm{v}_h}_\alpha\right\rangle_F-\sum_{F\in\mathcal{F}_h^-}\langle\bm{\beta}\cdot\bm{n},\bm{u}_h\cdot\bm{v}_h\rangle_F \\
        =& (\varepsilon\nabla\times\bm{u}_h,\nabla\times\bm{v}_h) + (L_{\bm{\beta},h}\bm{u}_h,\bm{v}_h) - (\bm{r}_\alpha(\bm{\beta}\cdot\bm{n}^+\llbracket\bm{u}_h\rrbracket),\bm{v}_h) + (\gamma\bm{u}_h,\bm{v}_h).
\end{aligned}
\end{equation}
Based on this, we define the discrete magnetic advection operator element-wise as $\tilde{L}_{\bm{\beta},h}\bm{u}_h|_T := {L}_{\bm{\beta},h}\bm{u}_h|_T - \bm{r}_{\alpha}|_T(\bm{\beta}\cdot\bm{n}\llbracket\bm{u}_h\rrbracket)$. The modified advection-diffusion-reaction operator is then given by $\tilde{\mathcal{A}}_h\bm{u}_h = \nabla\times(\varepsilon\nabla\times\bm{u}_h) + \tilde{L}_{\bm{\beta},h}\bm{u}_h + \gamma\bm{u}_h$.

\item \textbf{Step 2:} Inspired by the SUPG method for scalar problems, modifying the test functions $\bm v_h$ elementwise to $\bm v_h + \delta_T\tilde{L}_{\bm \beta,h}\bm v_h$ introduces a new stabilization term
\begin{equation}
S_h^2(\bm{u}_h,\bm{v}_h) = \sum_{T\in\mathcal{T}_h}( \tilde{\mathcal{A}}_h\bm{u}_h , \delta_T\tilde{L}_{\bm{\beta},h}\bm{v}_h )_T.
\end{equation}
\end{itemize}
Consequently, we develop the SUPG method for the magnetic advection-diffusion problem: Find $\bm{u}_h\in \tilde{\bm{u}}_{bh} + \bm{V}_{h,0}$ such that
\begin{equation}\label{magnetic_SUPG_scheme}
    a_h(\bm{u}_h,\bm{v}_h) = F_h(\bm{v}_h),~~\forall \bm{v}_h\in\bm{V}_{h,0},
\end{equation}
where
\begin{subequations}
\begin{align}
        a_h(\bm{u}_h,\bm{v}_h) =&  a_0(\bm{u}_h,\bm{v}_h) + S_h^1(\bm{u}_h,\bm{v}_h) + S_h^2(\bm{u}_h,\bm{v}_h)\notag\\
        =& (\varepsilon\nabla\times\bm{u}_h,\nabla\times\bm{v}_h) + (\tilde{L}_{\bm{\beta},h}\bm{u}_h + \gamma\bm{u}_h,\bm{v}_h)\notag\\
        &+\sum_{T\in\mathcal{T}_h}(\tilde{\mathcal{A}}_h\bm{u}_h,\delta_T\tilde{L}_{\bm{\beta},h}\bm{v}_h)_T,\label{magnetic_SUPG_scheme_bilinear_form}\\
        F_h(\bm{v}_h)=&(\bm{f},\bm{v}_h)+\sum_{T\in\mathcal{T}_h}(\bm{f},\delta_T\tilde{L}_{\bm{\beta},h}\bm{v}_h)_T\notag\\
        &-\sum_{F\in\mathcal{F}_h^-}\langle\bm{\beta}\cdot\bm{n},\bm{g}\cdot\bm{v}_h\rangle_F -\sum_{F\in\mathcal{F}_h^-}\langle\bm{\beta}\cdot\bm{n},\bm{g}\cdot\delta_T\tilde{L}_{\bm{\beta},h}\bm{v}_h\rangle_F.
\end{align}
\end{subequations}
It is straightforward to verify the Galerkin orthogonality
\begin{equation}\label{Galerkin_orthogonality}
    a_h(\bm{u}-\bm{u}_h,\bm{v}_h) = 0,~~\forall\bm{v}_h\in\bm{V}_{h,0}.
\end{equation}
}

Numerical experiments in Section \ref{sect:numerical_experiment} demonstrate that the proposed SUPG method \eqref{magnetic_SUPG_scheme} achieves superior stability, accelerated convergence, and enhanced oscillation damping in sharp layers compared to schemes without SUPG stabilization.

\begin{remark}
While $S_h^1(\cdot,\cdot)$ is widely used in existing methods \citep{heumann2013stabilized, Luo2025LocalPS}, we propose $S_h^2(\cdot,\cdot)$ as a novel stabilization term. Specifically, $S_h^1(\cdot,\cdot)$ primarily controls the jump terms across the element interfaces, while $S_h^2(\cdot,\cdot)$ is crucial for handling the magnetic advection terms. Section \ref{sect:numerical_experiment} provides numerical experiments to further clarify the role of the two stabilization terms.

\end{remark}

\section{Stability analysis}\label{sect:stability_analysis}
The energy norm $\norm{\cdot}$ is defined as
\begin{equation}
\begin{aligned}
\norm{\bm{u}_h}^2&:=\varepsilon\Vert\nabla\times\bm{u}_h\Vert_0^2+\rho_0\Vert\bm{u}_h\Vert_0^2+\sum_{T\in\mathcal{T}_h}\delta_T\Vert \tilde{L}_{\bm{\beta},h}\bm{u}_h\Vert_{0,T}^2\\
&+\frac{1}{2}\sum_{F\in\mathcal{F}_h^{\circ}}\langle \vert[\alpha]\bm{\beta}\cdot\bm{n}\vert,\vert\llbracket\bm{u}_h\rrbracket\vert^2\rangle_F + \frac{1}{2}\sum_{F\in\mathcal{F}_h^{\partial}}\langle\vert\bm{\beta}\cdot\bm{n}\vert,\vert\bm{u}_h\vert^2\rangle_F.
\end{aligned}
\end{equation}
In this section, we establish the coercivity of the bilinear form \( a_h(\cdot,\cdot) \) with respect to the energy norm under mild assumptions. Using the standard inverse inequality \eqref{standard_inverse_inequality}, we obtain the following estimate.
\begin{corollary}[inverse inequality for $\nabla\times$ operator]
There exists a positive constant $C_{inv}$ independent of $h_T$, such that
\begin{equation}\label{FE_inverse_ineq_for_curl}
    \Vert\nabla\times(\nabla\times\bm{v}_h)\Vert_{0,T} \le C_{inv}h_T^{-1}\Vert\nabla\times\bm{v}_h\Vert_{0,T}, ~~ \forall\bm{v}_h\in\bm{V}_h, \forall T\in\mathcal{T}_h.
\end{equation}
\end{corollary}
The following lemma provides a stability result for scheme \eqref{magnetic_SUPG_scheme}.
\begin{lemma}[discrete coercivity]\label{discrete_coercivity}
Suppose that the weight function $\alpha$ satisfies \eqref{alpha_need} and the SUPG parameter $\delta_T$ satisfies
\begin{equation}\label{delta_T_satisfy}
    \delta_T \le \min\{\frac{h_T^2}{2C_{inv}^2\varepsilon}, \frac{\rho_0}{2\Vert\gamma\Vert_{L^{\infty}(T)}^2}\},
\end{equation}
for each $T\in\mathcal{T}_h$. Then the discrete bilinear form \eqref{magnetic_SUPG_scheme_bilinear_form} is coercive on $\bm{V}_h$ with respect to the energy norm $\norm{\cdot}$, i.e.,
\begin{equation}\label{discrete_coercivity_concreteform}
    a_h(\bm{v}_h,\bm{v}_h)\ge\frac{1}{2}\norm{\bm{v}_h}^2,~~\forall\bm{v}_h\in\bm{V}_h.
\end{equation}
\end{lemma}

\begin{proof}
By definition, we have
\begin{equation}
\begin{aligned}
    a_h(\bm{v}_h,\bm{v}_h) =& {(\varepsilon\nabla\times\bm{v}_h,\nabla\times\bm{v}_h)} + \underbrace{(\tilde{L}_{\bm{\beta},h}\bm{v}_h+\gamma\bm{v}_h,\bm{v}_h)}_{I_1}\\
    &+{\sum_{T\in\mathcal{T}_h}\delta_T(\tilde{L}_{\bm{\beta},h}\bm{v}_h,\tilde{L}_{\bm{\beta},h}\bm{v}_h)_T}\\
    &+\underbrace{\sum_{T\in\mathcal{T}_h}\delta_T(\varepsilon\nabla\times(\nabla\times\bm{v}_h)+\gamma\bm{v}_h,\tilde{L}_{\bm{\beta},h}\bm{v}_h)_T}_{I_2}.
\end{aligned}
\end{equation}
\par
\underline{Estimate of $I_1$}: We claim that
\begin{equation}\label{pos_of_conv_term}
I_1=(\tilde{L}_{\bm{\beta},h}\bm{v}_h+\gamma\bm{v}_h,\bm{v}_h)\ge\rho_0\Vert\bm{v}_h\Vert_0^2+\frac{1}{2}\sum_{F\in\mathcal{F}_h^\circ}\langle \vert[\alpha]\bm{\beta}\cdot\bm{n}\vert,\vert\llbracket \bm{v}_h\rrbracket\vert^2\rangle_F + \frac{1}{2}\sum_{F\in\mathcal{F}_h^{\partial}}\langle\vert\bm{\beta}\cdot\bm{n}\vert,\vert\bm{v}_h\vert^2\rangle_F.
\end{equation}
Taking $\bm{u}_h=\bm{v}_h$ in \eqref{identities_Lie_advection_2}, we obtain
$$
\frac{1}{2}(L_{\bm{\beta}}\bm{v}_h,\bm{v}_h)_T = \frac{1}{2}(\bm{v}_h,\mathcal{L}_{\bm{\beta}}\bm{v}_h)_T + \frac{1}{2}\langle\bm{\beta}\cdot\bm{n},\bm{v}_h\cdot\bm{v}_h\rangle_{\partial T}.
$$
Using the definition of $\tilde{L}_{\bm{\beta},h}$, the left-hand side of \eqref{pos_of_conv_term} becomes
\begin{equation}
\begin{split}
    (\tilde{L}_{\bm{\beta},h}\bm{v}_h+\gamma\bm{v}_h,\bm{v}_h)=&({L}_{\bm{\beta},h}\bm{v}_h+\gamma\bm{v}_h,\bm{v}_h)\\
    &-\sum_{F\in\mathcal{F}_h^{\circ}}\langle\bm{\beta}\cdot\bm{n}^+,\llbracket\bm{v}_h\rrbracket\cdot\vavg{\bm{v}_h}_{\alpha}\rangle_{F} - \sum_{F\in\mathcal{F}_h^{-}}\langle\bm{\beta}\cdot\bm{n},\bm{v}_h\cdot\bm{v}_h\rangle_{F}\\
    =&\left( (\frac{1}{2}L_{\bm{\beta},h}+\frac{1}{2}\mathcal{L}_{\bm{\beta},h}+\gamma)  \bm{v}_h,\bm{v}_h\right)+\frac{1}{2}\sum_{T\in\mathcal{T}_h}\langle\bm{\beta}\cdot\bm{n},\bm{v}_h\cdot\bm{v}_h\rangle_{\partial T}\\
    &-\sum_{F\in\mathcal{F}_h^{\circ}}\langle\bm{\beta}\cdot\bm{n}^+,\llbracket\bm{v}_h\rrbracket\cdot\vavg{\bm{v}_h}_{\alpha}\rangle_{F} - \sum_{F\in\mathcal{F}_h^{-}}\langle\bm{\beta}\cdot\bm{n},\bm{v}_h\cdot\bm{v}_h\rangle_{F}.
\end{split}
\end{equation}
The identity \eqref{identities_Lie_advection_1} together with assumption \eqref{well_posedness_assumption_formulation} gives
\begin{equation}\label{pos_of_conv_term 1}
\left( (\frac{1}{2}L_{\bm{\beta},h}+\frac{1}{2}\mathcal{L}_{\bm{\beta},h}+\gamma I)  \bm{v}_h,\bm{v}_h\right) = \left( \left[(\gamma-\frac{\nabla\cdot\bm{\beta}}{2})I+\frac{\nabla\bm{\beta}+(\nabla\bm{\beta})^T}{2}\right]  \bm{v}_h,\bm{v}_h\right)\ge\rho_0\Vert\bm{v}_h\Vert_0^2.
\end{equation}
For the remaining terms, we have
\begin{equation}\label{pos_of_conv_term 2}
\begin{split}
    &\frac{1}{2}\sum_{T\in\mathcal{T}_h}\langle\bm{\beta}\cdot\bm{n},\bm{v}_h\cdot\bm{v}_h\rangle_{\partial T}-\sum_{F\in\mathcal{F}_h^{\circ}}\langle\bm{\beta}\cdot\bm{n}^+,\llbracket\bm{v}_h\rrbracket\cdot\vavg{\bm{v}_h}_{\alpha}\rangle_{F} - \sum_{F\in\mathcal{F}_h^{-}}\langle\bm{\beta}\cdot\bm{n},\bm{v}_h\cdot\bm{v}_h\rangle_{F}\\
    =&\sum_{F\in\mathcal{F}_h^{\circ}}\langle\bm{\beta}\cdot\bm{n}^+,\frac{1}{2}\llbracket
    \bm{v}_h\cdot\bm{v}_h\rrbracket-\llbracket\bm{v}_h\rrbracket\cdot\vavg{\bm{v}_h}_{\alpha}\rangle_F\\
    &+\frac{1}{2}\sum_{F\in\mathcal{F}_h^{+}}\langle\bm{\beta}\cdot\bm{n},\bm{v}_h\cdot\bm{v}_h\rangle_{F} + \frac{1}{2}\sum_{F\in\mathcal{F}_h^{-}}\langle-\bm{\beta}\cdot\bm{n},\bm{v}_h\cdot\bm{v}_h\rangle_{F}\\
    =&\frac{1}{2}\sum_{F\in\mathcal{F}_h^\circ}\langle-[\alpha]\bm{\beta}\cdot\bm{n}^+,\vert\llbracket\bm{v}_h\rrbracket\vert^2\rangle_F + \frac{1}{2}\sum_{F\in\mathcal{F}_h^{\partial}}\langle\vert\bm{\beta}\cdot\bm{n}\vert,\vert\bm{v}_h\vert^2\rangle_{F}\\
    =&\frac{1}{2}\sum_{F\in\mathcal{F}_h^\circ}\langle\vert[\alpha]\bm{\beta}\cdot\bm{n}\vert,\vert\llbracket\bm{v}_h\rrbracket\vert^2\rangle_F + \frac{1}{2}\sum_{F\in\mathcal{F}_h^{\partial}}\langle\vert\bm{\beta}\cdot\bm{n}\vert,\vert\bm{v}_h\vert^2\rangle_{F},
\end{split}
\end{equation}
where the second and third equalities follow from the weighted average identity \eqref{identities_weighted_average_3} and assumption \eqref{alpha_need}, respectively. Combining \eqref{pos_of_conv_term 1} and \eqref{pos_of_conv_term 2} yields the desired estimate \eqref{pos_of_conv_term}. 
\par 
\underline{Estimate of $I_2$}: By the Cauchy--Schwarz inequality and the inverse inequality \eqref{FE_inverse_ineq_for_curl}, we obtain
\begin{equation}\label{negpart4.10}
\begin{aligned}
\vert(\varepsilon\nabla\times(\nabla\times\bm{v}_h),\delta_T\tilde{L}_{\bm{\beta},h}\bm{v}_h)_T\vert&\le\delta_T\Vert\varepsilon\nabla\times(\nabla\times\bm{v}_h)\Vert_{0,T}^2 + \frac{1}{4}\delta_T\Vert\tilde{L}_{\bm{\beta},h}\bm{v}_h\Vert_{0,T}^2  \\
&\le C^2_{inv}\delta_T\varepsilon^2h_T^{-2}\Vert\nabla\times\bm{v}_h\Vert_{0,T}^2 + \frac{1}{4}\delta_T\Vert\tilde{L}_{\bm{\beta},h}\bm{v}_h\Vert_{0,T}^2.
\end{aligned}
\end{equation}
Similarly, we have
\begin{equation}\label{negpart4.11}
\begin{aligned}
    \vert(\gamma\bm{v}_h,\delta_T\tilde{L}_{\bm{\beta},h}\bm{v}_h)_T\vert&\le \delta_T\Vert\gamma\bm{v}_h\Vert_{0,T}^2+ \frac{1}{4}\delta_T\Vert\tilde{L}_{\bm{\beta},h}\bm{v}_h\Vert_{0,T}^2\\
    &\le  \delta_T\Vert\gamma\Vert_{L^{\infty}(T)}^2\Vert\bm{v}_h\Vert_{0,T}^2+ \frac{1}{4}\delta_T\Vert\tilde{L}_{\bm{\beta},h}\bm{v}_h\Vert_{0,T}^2.
\end{aligned}
\end{equation}
Combining the estimates \eqref{pos_of_conv_term}, \eqref{negpart4.10} and \eqref{negpart4.11}, we obtain
\begin{equation}
    \begin{split}
        a_h(\bm{v}_h,\bm{v}_h) &\ge \sum_{T\in\mathcal{T}_h}(1 - C^2_{inv}\varepsilon h_T^{-2}\delta_T)\varepsilon\Vert\nabla\times\bm{v}_h\Vert_{0,T}^2 \\
        &+ \sum_{T\in\mathcal{T}_h}(\rho_0-\Vert\gamma\Vert_{L^{\infty}(T)}^2\delta_T)\Vert\bm{v}_h\Vert_{0,T}^2 +\sum_{T\in\mathcal{T}_h}\frac{1}{2}\delta_T\Vert\tilde{L}_{\bm{\beta},h}\bm{v}_h\Vert_{0,T}^2\\
        &+ \frac{1}{2}\sum_{F\in\mathcal{F}_h^{\circ}}\langle \vert[\alpha]\bm{\beta}\cdot\bm{n}\vert,\vert\llbracket\bm{v}_h\rrbracket\vert^2\rangle_F + \frac{1}{2}\sum_{F\in\mathcal{F}_h^{\partial}}\langle\vert\bm{\beta}\cdot\bm{n}\vert,\vert\bm{v}_h\vert^2\rangle_F.
    \end{split}
\end{equation}
Finally, by assumption \eqref{delta_T_satisfy}, we conclude
\begin{equation}
    a_h(\bm{v}_h,\bm{v}_h) \ge \frac{1}{2}\norm{\bm{v}_h}^2.
\end{equation}
This completes the proof.   
\end{proof}

\section{A priori error estimate}\label{sect:priori}
In this section, let $C, C_T$ denote generic positive constants independent of the local mesh size $h_T$ and diffusion coefficient $\varepsilon$, but may depend on the shape regularity constant, $\bm{\beta}$, or $\gamma$. The subscript in $C_T$ indicates that it depends on the restrictions $\bm{\beta}|_T$ and/or $\gamma|_T$, rather than on the global functions $\bm{\beta}$ or $\gamma$. Before deriving error estimates, we introduce the following assumption:
\begin{assumption}[approximation property]
There exist a positive integer $k$ and an interpolation operator $\bm{i}_h$, such that $\bm{i}_h\bm{v}\in\bm{V}_h$ and
\begin{equation}\label{the_approximation_property}
    \Vert\bm{v}-\bm{i}_h\bm{v}\Vert_{m,T}\le Ch_T^{k+1-m}\vert\bm{v}\vert_{k+1,T}, \qquad \forall\bm{v}\in(H^{k+1}(T))^3,
\end{equation}
for all integers $m$ satisfying $0\le m\le k+1$. $C$ is a constant independent of $h_T$.
\end{assumption}
This holds for a broad class of $\bm{V}_h$, including the $k$-th order Nédélec space of the second kind, in which case $\bm{i}_h$ can be taken as the canonical interpolation operator. See \citep{boffi2013mixed} for details.
\par
\begin{lemma}[approximation inequalities]
There exist constants $C_T$ independent of $h_T$ such that 
\begin{subequations}
\begin{align}
    \left(\sum_{F\in\mathcal{F}_h} \langle \vert\bm{\beta}\cdot\bm{n}\vert , \vert 
(\bm{v}-\bm{i}_h\bm{v})^{\pm} \vert^2 \rangle_F \right)^{\frac{1}{2}}&\le\left( \sum_{T\in\mathcal{T}_h}C_T h_T^{2k+1}\vert\bm{v}\vert_{k+1,T}^2 \right)^{\frac{1}{2}},\label{approximation inequality 1}\\
    \left(\sum_{T\in\mathcal{T}_h}\delta_T\Vert\tilde{L}_{\bm{\beta},h}(\bm{v}-\bm{i}_h\bm{v})\Vert_{0,T}^2\right)^{\frac{1}{2}}&\le\left( \sum_{T\in\mathcal{T}_h}C_T\delta_Th_T^{2k}\vert\bm{v}\vert_{k+1,T}^2 \right)^{\frac{1}{2}},\label{approximation inequality 2}
    \end{align}
\end{subequations}
for any $\bm{v} \in (H^{k+1}(\Omega))^3$.
\end{lemma}
\begin{proof}
Define $\bm{w}:=\bm{v}-\bm{i}_h\bm{v}$. A direct calculation yields
\begin{equation}
\begin{aligned}
    &\sum_{F\in\mathcal{F}_h} \langle \vert\bm{\beta}\cdot\bm{n}\vert , \vert 
\bm{w}^{\pm} \vert^2 \rangle_F \le \sum_{F\in\mathcal{F}_h} \langle \vert\bm{\beta}\cdot\bm{n}\vert , \vert 
\bm{w}^{+} \vert^2 + \vert 
\bm{w}^{-} \vert^2 \rangle_F
\le 2\sum_{T\in\mathcal{T}_h}\langle\vert\bm{\beta}\cdot\bm{n}\vert , 
\vert\bm{w}\vert^2 \rangle_{\partial T}.
\end{aligned}
\end{equation}
The trace inequality \eqref{standard_trace_inequality} and the approximation property \eqref{the_approximation_property} yield
\begin{equation}\label{approximation inequality middle result}
\Vert\bm{w}\Vert_{0,\partial T}^2 \le C\Vert\bm{w}\Vert_{0,T}(h_T^{-2}\Vert\bm{w}\Vert_{0,T}^2+\vert\bm{w}\vert_{1,T}^2)^{1/2} \le Ch_T^{2k+1}\vert\bm{v}\vert_{k+1,T}^2.
\end{equation}
Consequently,
\begin{equation}
    \sum_{F\in\mathcal{F}_h} \langle \vert\bm{\beta}\cdot\bm{n}\vert , \vert 
\bm{w}^{\pm} \vert^2 \rangle_F\le \sum_{T\in\mathcal{T}_h}C_T\Vert\bm{w}\Vert_{0,\partial T}^2\le \sum_{T\in\mathcal{T}_h}C_Th_T^{2k+1}\vert\bm{v}\vert_{k+1,T}^2,
\end{equation}
thus establishing \eqref{approximation inequality 1}.

\par
Define $\omega(T)=\{T^\prime\in\mathcal{T}_h | T^{\prime}\cap T\neq\varnothing\}$ as the patch of $T$, and let $\delta_{\omega(T)}=\max_{T^{\prime}\in\omega(T)}\delta_{T^{\prime}}$. We make a mild assumption that $\delta_{\omega(T)}\lesssim\delta_T$, which is consistent with the common choice of $\delta_T$. Inequalities \eqref{lifting_operator_stability_estimate} and \eqref{approximation inequality middle result} imply
\begin{equation}\label{approx_property_lifting}
\begin{aligned}
    &\left(\sum_{T\in\mathcal{T}_h}\delta_T\Vert\bm{r}_{\alpha}(\bm{\beta}\cdot\bm{n^+\llbracket\bm{w}\rrbracket})\Vert_{0,T}^2\right)^{\frac{1}{2}} \le \left(\sum_{T\in\mathcal{T}_h}C_T\delta_Th_T^{-1}\Vert \bm{\beta}\cdot\bm{n}^+\llbracket\bm{w}\rrbracket \Vert_{0,\partial T}^2\right)^{\frac{1}{2}} \\
    \le&\left(\sum_{T\in\mathcal{T}_h}C_T\delta_{\omega(T)}h_T^{-1}\Vert\bm{w}\Vert_{0,\partial T}^2\right)^{\frac{1}{2}}\le\left(\sum_{T\in\mathcal{T}_h}C_T\delta_Th_T^{2k}\vert\bm{v}\vert_{k+1,T}^2\right)^{\frac{1}{2}},
\end{aligned}
\end{equation}
where the second inequality follows from the shape-regularity of $\mathcal{T}_h$. From \eqref{the_approximation_property}, we immediately derive
\begin{equation}\label{approx_property_Lbeta}
    \left(\sum_{T\in\mathcal{T}_h}\delta_T\Vert{L}_{\bm{\beta},h}\bm{w}\Vert_{0,T}^2\right)^{\frac{1}{2}}\le\left( 
\sum_{T\in\mathcal{T}_h}C_T\delta_Th_T^{2k}\vert\bm{v}\vert_{k+1,T}^2 \right)^{\frac{1}{2}}.
\end{equation}
The proof of \eqref{approximation inequality 2} follows from combining \eqref{approx_property_lifting} with \eqref{approx_property_Lbeta}.
\end{proof}
We now present the following error estimate.
\begin{theorem}[error estimate]
Let $\bm{u}\in(H^{k+1}(\Omega))^3$ be the exact solution of \eqref{curl_advection_diffusion_equation}, and $\bm{u}_h$ the solution of the discrete problem \eqref{magnetic_SUPG_scheme}. Under the same assumptions as Lemma \ref{discrete_coercivity}, there exist constants $C_T>0$ independent of $h_T$ and $\varepsilon$ such that
\begin{equation}
    \norm{\bm{u}-\bm{u}_h}\le\left(\sum_{T\in\mathcal{T}_h} C_T(\varepsilon+\frac{h_T^2}{\delta_T}+h_T+\frac{\varepsilon^2}{h_T^2}\delta_T+\delta_T)h_T^{2k} \vert\bm{u}\vert_{k+1,T}^2\right)^{\frac{1}{2}}.
\end{equation}
\end{theorem}
\begin{proof}  
By the triangle inequality, the total energy error decomposes into an interpolation error and a projection error:
\begin{equation}
    \norm{\bm{u}-\bm{u}_h}\le\norm{\bm{u}-\bm{i}_h\bm{u}}+\norm{\bm{i}_h\bm{u}-\bm{u}_h}.
\end{equation}
For the interpolation error, the approximation properties \eqref{the_approximation_property}, \eqref{approximation inequality 1}, and \eqref{approximation inequality 2} give
 \begin{equation}\label{estimate for the energy norm of the interpolation error}
 \begin{aligned}
    &\norm{\bm{u}-\bm{i}_h\bm{u}}^2 =  \varepsilon\Vert\nabla\times(\bm{u}-\bm{i}_h\bm{u})\Vert_0^2 + \rho_0\Vert\bm{u}-\bm{i}_h\bm{u}\Vert_0^2 + \sum_{T\in\mathcal{T}_h}\delta_T\Vert\tilde{L}_{\bm{\beta},h}(\bm{u}-\bm{i}_h\bm{u})\Vert_{0,T}^2  \\
    &+\frac{1}{2}\sum_{F\in\mathcal{F}_h^\circ}\langle\vert[\alpha]\bm{\beta}\cdot\bm{n}\vert,\vert\llbracket\bm{u}-\bm{i}_h\bm{u}\rrbracket\vert^2\rangle_F +\frac{1}{2}\sum_{F\in\mathcal{F}_h^\partial}\langle\vert\bm{\beta}\cdot\bm{n}\vert,\vert\bm{u}-\bm{i}_h\bm{u}\vert^2\rangle_F \\
    \le& \sum_{T\in\mathcal{T}_h} C_T(\varepsilon+h_T^2+\delta_T+h_T)h_T^{2k}\vert\bm{u}\vert_{k+1,T}^2 \le\sum_{T\in\mathcal{T}_h} C_T(\varepsilon+h_T+\delta_T)h_T^{2k}\vert\bm{u}\vert_{k+1,T}^2.
\end{aligned}
\end{equation}
For the projection error, applying the discrete coercivity \eqref{discrete_coercivity_concreteform} and Galerkin orthogonality \eqref{Galerkin_orthogonality}, we obtain
\begin{equation}
\frac{1}{2}\norm{\bm{i}_h\bm{u}-\bm{u}_h}^2 \le a_h(\bm{i}_h\bm{u}-\bm{u}_h,\bm{i}_h\bm{u}-\bm{u}_h) = a_h(\bm{i}_h\bm{u}-\bm{u},\bm{i}_h\bm{u}-\bm{u}_h).
\end{equation}
The right-hand side can be written as six parts:
$$
a_h(\bm{i}_h\bm{u}-\bm{u},\bm{i}_h\bm{u}-\bm{u}_h) := \sum_{i=1}^6M_i ,
$$
where
\begin{equation}
    \begin{split}
        & M_1 = (\varepsilon\nabla\times(\bm{i}_h\bm{u}-\bm{u}),\nabla\times(\bm{i}_h\bm{u}-\bm{u}_h)), \quad M_2 = (\tilde{L}_{\bm{\beta},h}(\bm{i}_h\bm{u}-\bm{u}),\bm{i}_h\bm{u}-\bm{u}_h),\\
        &M_3 = (\gamma(\bm{i}_h\bm{u}-\bm{u}),\bm{i}_h\bm{u}-\bm{u}_h),\quad M_4 = \sum_{T\in\mathcal{T}_h}(\varepsilon\nabla\times(\nabla\times(\bm{i}_h\bm{u}-\bm{u})),
        \delta_T\tilde{L}_{\bm{\beta},h}(\bm{i}_h\bm{u}-\bm{u}_h))_T, \\
        &M_5 = \sum_{T\in\mathcal{T}_h}(\tilde{L}_{\bm{\beta},h}(\bm{i}_h\bm{u}-\bm{u}),
        \delta_T\tilde{L}_{\bm{\beta},h}(\bm{i}_h\bm{u}-\bm{u}_h))_T, M_6 = \sum_{T\in\mathcal{T}_h}(\gamma(\bm{i}_h\bm{u}-\bm{u}),
        \delta_T\tilde{L}_{\bm{\beta},h}(\bm{i}_h\bm{u}-\bm{u}_h))_T.
    \end{split}
\end{equation}
These terms are estimated separately in the following.
\par
\underline{Estimate of $M_1$}: By the Cauchy-Schwarz inequality and the approximation property \eqref{the_approximation_property},
\begin{equation}\label{estimateofM1}
\begin{split}
    \vert M_1\vert =& \vert(\varepsilon\nabla\times(\bm{i}_h\bm{u}-\bm{u}),\nabla\times(\bm{i}_h\bm{u}-\bm{u}_h))\vert \\
    \le& \left(\sum_{T\in\mathcal{T}_h}\Vert\sqrt{\varepsilon}\nabla\times(\bm{i}_h\bm{u}-\bm{u})\Vert_{0,T}^2\right)^{\frac{1}{2}}\left(\sum_{T\in\mathcal{T}_h}\Vert\sqrt{\varepsilon}\nabla\times(\bm{i}_h\bm{u}-\bm{u}_h)\Vert_{0,T}^2\right)^{\frac{1}{2}}\\
    \le& \left(\sum_{T\in\mathcal{T}_h}C_T\varepsilon h_T^{2k}\right)^{\frac{1}{2}}\norm{\bm{i}_h\bm{u}-\bm{u}_h}.
\end{split}
\end{equation}\par
\underline{Estimate of $M_2$}: By the definition of $\tilde{L}_{\bm{\beta},h}$ and the integration by parts formula \eqref{identities_Lie_advection_2}, we derive
\begin{equation}
\begin{aligned}
M_2=&(\tilde{L}_{\bm{\beta},h}(\bm{i}_h\bm{u}-\bm{u}),\bm{i}_h\bm{u}-\bm{u}_h) = (L_{\bm{\beta},h}(\bm{i}_h\bm{u}-\bm{u}),\bm{i}_h\bm{u}-\bm{u}_h) \\
&-\sum_{F\in\mathcal{F}_h^{\circ}\cup\mathcal{F}_h^-}\langle \bm{\beta}\cdot\bm{n}^+,\llbracket(\bm{i}_h\bm{u}-\bm{u})\rrbracket\cdot\vavg{\bm{i}_h\bm{u}-\bm{u}_h}_{\alpha}\rangle_F\\
=&\underbrace{(\bm{i}_h\bm{u}-\bm{u},\mathcal{L}_{\bm{\beta},h}(\bm{i}_h\bm{u}-\bm{u}_h))}_{M_{2,1}}  \\
&+\underbrace{\sum_{T\in\mathcal{T}_h}\langle\bm{\beta}\cdot\bm{n},(\bm{i}_h\bm{u}-\bm{u})\cdot(\bm{i}_h\bm{u}-\bm{u}_h)\rangle_{\partial T} - \sum_{F\in\mathcal{F}_h^{\circ}\cup\mathcal{F}_h^-}\langle \bm{\beta}\cdot\bm{n}^+,\llbracket\bm{i}_h\bm{u}-\bm{u}\rrbracket\cdot\vavg{\bm{i}_h\bm{u}-\bm{u}_h}_{\alpha}\rangle_F}_{M_{2,2}}.
\end{aligned}
\end{equation}
From the identity \eqref{identities_Lie_advection_1}, we obtain
\begin{equation}
\begin{aligned}
M_{2,1} =& -(\bm{i}_h\bm{u}-\bm{u},L_{\bm{\beta},h}(\bm{i}_h\bm{u}-\bm{u}_h))+(\bm{i}_h\bm{u}-\bm{u},L_{\bm{\beta},h}(\bm{i}_h\bm{u}-\bm{u}_h)+\mathcal{L}_{\bm{\beta},h}(\bm{i}_h\bm{u}-\bm{u}_h))\\
=&\underbrace{-(\bm{i}_h\bm{u}-\bm{u},\tilde{L}_{\bm{\beta},h}(\bm{i}_h\bm{u}-\bm{u}_h))}_{m_1}\underbrace{-(\bm{i}_h\bm{u}-\bm{u},\bm{r}_{\alpha}(\bm{\beta}\cdot\bm{n}^+\llbracket\bm{i}_h\bm{u}-\bm{u}_h\rrbracket))}_{m_2} \\
&+\underbrace{(\bm{i}_h\bm{u}-\bm{u},(-(\nabla\cdot\bm{\beta})I+[\nabla\bm{\beta}+(\nabla\bm{\beta})^T])(\bm{i}_h\bm{u}-\bm{u}_h))}_{m_3}.
\end{aligned}
\end{equation}
The weighted average identity \eqref{identities_weighted_average_2} implies
\begin{equation}
\begin{aligned}
M_{2,2}=&\sum_{F\in\mathcal{F}_h^\circ}\langle\bm{\beta}\cdot\bm{n}^+, 
 \llbracket(\bm{i}_h\bm{u}-\bm{u})\cdot(\bm{i}_h\bm{u}-\bm{u}_h)\rrbracket - \llbracket\bm{i}_h\bm{u}-\bm{u}\rrbracket\cdot\vavg{\bm{i}_h\bm{u}-\bm{u}_h}_{\alpha} \rangle_F\\
 &+\sum_{F\in\mathcal{F}_h^+}\langle\bm{\beta}\cdot\bm{n},(\bm{i}_h\bm{u}-\bm{u})\cdot(\bm{i}_h\bm{u}-\bm{u}_h)\rangle_F\\
=&\underbrace{\sum_{F\in\mathcal{F}_h^\circ}\langle\bm{\beta}\cdot\bm{n}^+,\vavg{\bm{i}_h\bm{u}-\bm{u}}_{1-\alpha}\cdot\llbracket\bm{i}_h\bm{u}-\bm{u}_h\rrbracket\rangle_F}_{m_4}+\underbrace{\sum_{F\in\mathcal{F}_h^+}\langle\bm{\beta}\cdot\bm{n},(\bm{i}_h\bm{u}-\bm{u})\cdot(\bm{i}_h\bm{u}-\bm{u}_h)\rangle_F}_{m_5}.
\end{aligned}
\end{equation}
The estimate of $M_2$ requires estimating $m_i$ ($i=1,2,\cdots,6$). The approximation property \eqref{the_approximation_property} yields
\begin{equation}\label{estimateofm1}
\begin{split}
    m_1=&-(\bm{i}_h\bm{u}-\bm{u},\tilde{L}_{\bm{\beta},h}(\bm{i}_h\bm{u}-\bm{u}_h))\\
    \le&\left(\sum_{T\in\mathcal{T}_h}  \delta_T^{-1}\Vert\bm{i}_h\bm{u}-\bm{u}\Vert_{0,T}^2  \right)^{\frac{1}{2}}\left(\sum_{T\in\mathcal{T}_h} 
\delta_T\Vert\tilde{L}_{\bm{\beta},h}(\bm{i}_h\bm{u}-\bm{u}_h)\Vert_{0,T}^2   \right)^{\frac{1}{2}}\\
\le& \left(\sum_{T\in\mathcal{T}_h} C_T \delta_T^{-1}h_T^{2k+2}\vert\bm{u}\vert_{k+1,T}^2  \right)^{\frac{1}{2}}\norm{\bm{i}_h\bm{u}-\bm{u}_h}.
\end{split}
\end{equation}
Recalling the requirement \eqref{alpha_need} for weight $\alpha$, we deduce
\begin{equation}\label{energy_norm_greater_than_interior_jump}
\begin{split}
    \norm{\bm{i}_h\bm{u}-\bm{u}_h}^2\ge\frac{1}{2}\sum_{F\in\mathcal{F}_h^{\circ}}\langle\vert[\alpha]\bm{\beta}\cdot\bm{n}\vert,\vert\llbracket\bm{i}_h\bm{u}-\bm{u}_h\rrbracket\vert^2\rangle_F\ge\frac{1}{2}\sum_{F\in\mathcal{F}_h^{\circ}}\langle C_{up}\vert\bm{\beta}\cdot\bm{n}\vert,\vert\llbracket\bm{i}_h\bm{u}-\bm{u}_h\rrbracket\vert^2\rangle_F,
\end{split}
\end{equation}
where $C_{up}$ is a positive constant. Then the 
estimate for the lifting operator \eqref{lifting_operator_stability_estimate} and the estimate \eqref{energy_norm_greater_than_interior_jump} imply
\begin{equation}\label{estimateofm2}
\begin{split}
    m_2 =& -(\bm{i}_h\bm{u}-\bm{u},\bm{r}_{\alpha}(\bm{\beta}\cdot\bm{n}^+\llbracket\bm{i}_h\bm{u}-\bm{u}_h\rrbracket))\\
    \le&\left( \sum_{T\in\mathcal{T}_h} h_T^{-1}\Vert(\bm{i}_h\bm{u}-\bm{u})\Vert_{0,T}^2 \right)^{\frac{1}{2}} \left( \sum_{T\in\mathcal{T}_h  } h_T\Vert\bm{r}_{\alpha}(\bm{\beta}\cdot\bm{n}^+\llbracket\bm{i}_h\bm{u}-\bm{u}_h\rrbracket)\Vert_{0,T}^2 \right)^{\frac{1}{2}}\\
    \le&\left(\sum_{T\in\mathcal{T}_h}C_Th_T^{2k+1}  \right)^{\frac{1}{2}}\left(C_{up}^{-1}\sum_{F\in\mathcal{F}_h^{\circ}}\langle C_{up}\vert\bm{\beta}\cdot\bm{n}\vert,\vert\llbracket\bm{i}_h\bm{u}-\bm{u}_h\rrbracket\vert^2 \rangle_{F} + \sum_{F\in\mathcal{F}_h^{\partial}}\langle \vert\bm{\beta}\cdot\bm{n}\vert,\vert\bm{i}_h\bm{u}-\bm{u}_h\vert^2 \rangle_{F} \right)^{\frac{1}{2}}\\
    \le&\left(\sum_{T\in\mathcal{T}_h}C_Th_T^{2k+1}  \right)^{\frac{1}{2}}\norm{\bm{i}_h\bm{u}-\bm{u}_h}.
\end{split}
\end{equation}
Given that $L_{\bm{\beta}}(\bm{i}_h\bm{u}-\bm{u}_h)+\mathcal{L}_{\bm{\beta}}(\bm{i}_h\bm{u}-\bm{u}_h)$ contains no derivative of $\bm{i}_h\bm{u}-\bm{u}_h$, it follows that
\begin{equation}\label{estimateofm3}
\begin{split}
    m_3 =&\left((-(\nabla\cdot\bm{\beta})I+[\nabla\bm{\beta}+(\nabla\bm{\beta})^T])(\bm{i}_h\bm{u}-\bm{u}),\bm{i}_h\bm{u}-\bm{u}_h\right)\\
    \le& \left(\sum_{T\in\mathcal{T}_h}C_T\rho_0^{-1}\vert\bm{\beta}\vert^2_{1,\infty,T}\Vert\bm{i}_h\bm{u}-\bm{u}\Vert_{0,T}^2  \right)^{\frac{1}{2}} \left(\sum_{T\in\mathcal{T}_h}\rho_0\Vert\bm{i}_h\bm{u}-\bm{u}_h\Vert_{0,T}^2  \right)^{\frac{1}{2}}\\
    \le&\left(\sum_{T\in\mathcal{T}_h}C_Th_T^{2k+2}  \right)^{\frac{1}{2}} \norm{\bm{i}_h\bm{u}-\bm{u}_h}.
\end{split}
\end{equation}
From the approximation inequality \eqref{approximation inequality 1} and the estimate \eqref{energy_norm_greater_than_interior_jump}, we establish
\begin{equation}\label{estimateofm4}
\begin{split}
    m_4 =& \sum_{F\in\mathcal{F}_h^{\circ}}\langle\bm{\beta}\cdot\bm{n}^+,\vavg{\bm{i}_h\bm{u}-\bm{u}}_{1-\alpha}\cdot\llbracket\bm{i}_h\bm{u}-\bm{u}_h\rrbracket   \rangle_F\\
    \le&\left(  \sum_{F\in\mathcal{F}_h^{\circ}}\langle\vert C_{up}^{-1}\bm{\beta}\cdot\bm{n}\vert,\vert\vavg{\bm{i}_h\bm{u}-\bm{u}}_{1-\alpha}\vert^2\rangle_F  \right)^{\frac{1}{2}}\left(  \sum_{F\in\mathcal{F}_h^{\circ}} \langle \vert C_{up}\bm{\beta}\cdot\bm{n}\vert,\vert\llbracket\bm{i}_h\bm{u}-\bm{u}_h\rrbracket\vert^2\rangle_F \right)^{\frac{1}{2}}    \\
    \le&\left( \sum_{T\in\mathcal{T}_h}C_T h_T^{2k+1}\vert\bm{u}\vert_{k+1,T}^2 \right)^{\frac{1}{2}}\norm{\bm{i}_h\bm{u}-\bm{u}_h}.
\end{split}
\end{equation}
$m_5$ is estimated analogously as follows:
\begin{equation}\label{estimateofm5}
\begin{split}
    m_5 =&  \sum_{F\in\mathcal{F}_h^+}\langle\bm{\beta}\cdot\bm{n},(\bm{i}_h\bm{u}-\bm{u})\cdot(\bm{i}_h\bm{u}-\bm{u}_h)\rangle_F   \\
        \le& \left( \sum_{F\in\mathcal{F}_h^{+}} \langle\vert\bm{\beta}\cdot\bm{n}\vert,\vert\bm{i}_h\bm{u}-\bm{u}\vert^2\rangle_F  \right)^{\frac{1}{2}}\left( \sum_{F\in\mathcal{F}_h^{+}} \langle\vert\bm{\beta}\cdot\bm{n}\vert,\vert\bm{i}_h\bm{u}-\bm{u}_h\vert^2\rangle_F  \right)^{\frac{1}{2}}\\
        \le&\left( \sum_{T\in\mathcal{T}_h} C_T h_T^{2k+1}\vert\bm{u}\vert_{k+1,T}^2 \right)^{\frac{1}{2}}\norm{\bm{i}_h\bm{u}-\bm{u}_h}.
\end{split}
\end{equation}
Combining \eqref{estimateofm1}, \eqref{estimateofm2}, \eqref{estimateofm3}, \eqref{estimateofm4}, and \eqref{estimateofm5}, we obtain the estimate for $M_2$:
\begin{equation}\label{estimateofM2}
\begin{split}
    \vert M_2\vert\le\left(\sum_{T\in\mathcal{T}_h}C_T( \delta_T^{-1}h_T^{2}+h_T^2+h_T
 )h_T^{2k}\lvert\bm{u}\rvert_{k+1,T}^2\right)^{\frac{1}{2}}  \norm{\bm{i}_h\bm{u}-\bm{u}_h}.
\end{split}
\end{equation}
\par
\underline{Estimate of $M_3$}: The approximation property \eqref{the_approximation_property} combined with the Cauchy-Schwarz inequality implies
\begin{equation}\label{estimateofM3}
\begin{split}
     M_3 =& (\gamma(\bm{i}_h\bm{u}-\bm{u}),\bm{i}_h\bm{u}-\bm{u}_h) \le\left(\sum_{T\in\mathcal{T}_h}\rho_0^{-1}\Vert\gamma\Vert^2_{0,\infty,T}\Vert\bm{i}_h\bm{u}-\bm{u}\Vert_{0,T}^2  \right)^{\frac{1}{2}} \left(\sum_{T\in\mathcal{T}_h}\rho_0\Vert\bm{i}_h\bm{u}-\bm{u}_h\Vert_{0,T}^2  \right)^{\frac{1}{2}}\\
    \le&\left(\sum_{T\in\mathcal{T}_h}C_Th_T^{2k+2}  \right)^{\frac{1}{2}} \norm{\bm{i}_h\bm{u}-\bm{u}_h}.
\end{split}
\end{equation}\par
\underline{Estimate of $M_4+M_5+M_6$}: The approximation property \eqref{the_approximation_property} directly yields:
\begin{equation}\label{estimateofM456}
\begin{split}
    &M_4+M_5+M_6 \\
    \le&\left(\sum_{T\in\mathcal{T}_h}\delta_T (\varepsilon^2\Vert\nabla\times(\nabla\times(\bm{i}_h\bm{u}-\bm{u}))\Vert_{0,T}^2 + \Vert\tilde{L}_{\bm{\beta},h}(\bm{i}_h\bm{u}-\bm{u})\Vert_{0,T}^2 + \Vert\gamma(\bm{i}_h\bm{u}-\bm{u})\Vert_{0,T}^2)            \right)^{\frac{1}{2}} \\
    &\cdot\left( \sum_{T\in\mathcal{T}_h} \delta_T\Vert\tilde{L}_{\bm{\beta},h}(\bm{i}_h\bm{u}-\bm{u}_h)\Vert_{0,T}^2 \right)^{\frac{1}{2}}\\
    \le&\left(\sum_{T\in\mathcal{T}_h} C_T\delta_T(\varepsilon^2h_T^{2k-2}+h_T^{2k} + h_T^{2k+2})\vert\bm{u}\vert_{k+1,T}^2 \right)^{\frac{1}{2}}\norm{\bm{i}_h\bm{u}-\bm{u}_h} \\
    \le&\left(\sum_{T\in\mathcal{T}_h} C_T(\varepsilon+\delta_T)h_T^{2k}\vert\bm{u}\vert_{k+1,T}^2 \right)^{\frac{1}{2}}\norm{\bm{i}_h\bm{u}-\bm{u}_h}.
\end{split}
\end{equation}
The last step employs the assumed bound $\varepsilon\delta_T \le Ch_T^2$ from Lemma \ref{discrete_coercivity}.
\par
Collecting \eqref{estimateofM1}, \eqref{estimateofM2}, \eqref{estimateofM3} and \eqref{estimateofM456}, we obtain
\begin{equation}\label{before pihu-uh_estimate}
\begin{split}
    &\frac{1}{2}\norm{\bm{i}_h\bm{u}-\bm{u}_h}^2 \le a_h((\bm{i}_h\bm{u}-\bm{u}),(\bm{i}_h\bm{u}-\bm{u}_h))=\sum_{i=1}^6M_i\\
    \le&\left(\sum_{T\in\mathcal{T}_h} C_T(\varepsilon+\frac{h_T^2}{\delta_T}+h_T+\delta_T)h_T^{2k} \vert\bm{u}\vert_{k+1,T}^2\right)^{\frac{1}{2}}\norm{\bm{i}_h\bm{u}-\bm{u}_h}.
\end{split}
\end{equation}
Cancellation of $\norm{\bm{i}_h\bm{u}-\bm{u}_h}$ from \eqref{before pihu-uh_estimate} yields:
\begin{equation}\label{pihu-uh_estimate}
    \norm{\bm{i}_h\bm{u}-\bm{u}_h}\le\left(\sum_{T\in\mathcal{T}_h} C_T(\varepsilon+\frac{h_T^2}{\delta_T}+h_T+\delta_T)h_T^{2k}\vert\bm{u}\vert_{k+1,T}^2 \right)^{\frac{1}{2}}.
\end{equation}
Combining the projection error estimate \eqref{pihu-uh_estimate} and the interpolation error estimate \eqref{estimate for the energy norm of the interpolation error}, we arrive at
\begin{equation}
    \norm{\bm{u}-\bm{u}_h}\le\left(\sum_{T\in\mathcal{T}_h} C_T(\varepsilon+\frac{h_T^2}{\delta_T}+h_T+\delta_T)h_T^{2k}\vert\bm{u}\vert_{k+1,T}^2 \right)^{\frac{1}{2}}.
\end{equation}
This completes the proof.
\end{proof}
\begin{remark}
The above estimate indicates that a balance among $\varepsilon, \delta_T$, and $\delta_T^{-1}h_T^2$ is required, which suggests the following choice for $\delta_T$:
\begin{equation}
\delta_T = \left\{\begin{aligned}
O(h_T),~~&\text{if } \frac{\Vert\bm{\beta}\Vert_{0,\infty,T}h_T}{2\varepsilon}>1 ~( \text{advection-dominated case} ),\\
O(h_T^2/\varepsilon),~~&\text{if } \frac{\Vert\bm{\beta}\Vert_{0,\infty,T}h_T}{2\varepsilon}\le 1 ~( \text{diffusion-dominated case} ).
\end{aligned}\right.
\end{equation}
In this case, the convergence rate of the energy-norm error is
\begin{equation}\label{energy error convergence rate}
\norm{\bm{u}-\bm{u}_h} \lesssim \left\{\begin{aligned}
O(h^{k+\frac{1}{2}})\hspace{4em}~~& ( \text{advection-dominated case} ),\\
O\left(h^k(\varepsilon^{\frac{1}{2}}+h^{\frac{1}{2}}+h\varepsilon^{-\frac{1}{2}})\right)~~& ( \text{diffusion-dominated case} ).
\end{aligned}\right.
\end{equation}
However, the specific optimal choice of $\delta_T$ needs more detailed discussion.  
\end{remark}

\section{Numerical experiment}\label{sect:numerical_experiment}
In this section, we present several numerical experiments in both two and three dimensions to verify our theoretical results, as well as to display the performance of the proposed SUPG method in the presence of layers. {Unless otherwise specified,} the computational domain is $\Omega = (0,1)^d$ for $d=2, 3$, where uniform simplex meshes with different mesh sizes are applied. Each coordinate axis is partitioned into $N$ equal segments, {as shown in Figure \ref{2D and 3D square mesh}}.  
\begin{figure}[!htbp]    
  \centering            
  \subfloat[2D mesh \label{uniform 2D mesh fig}]
  {
      \includegraphics[width=0.4\textwidth]{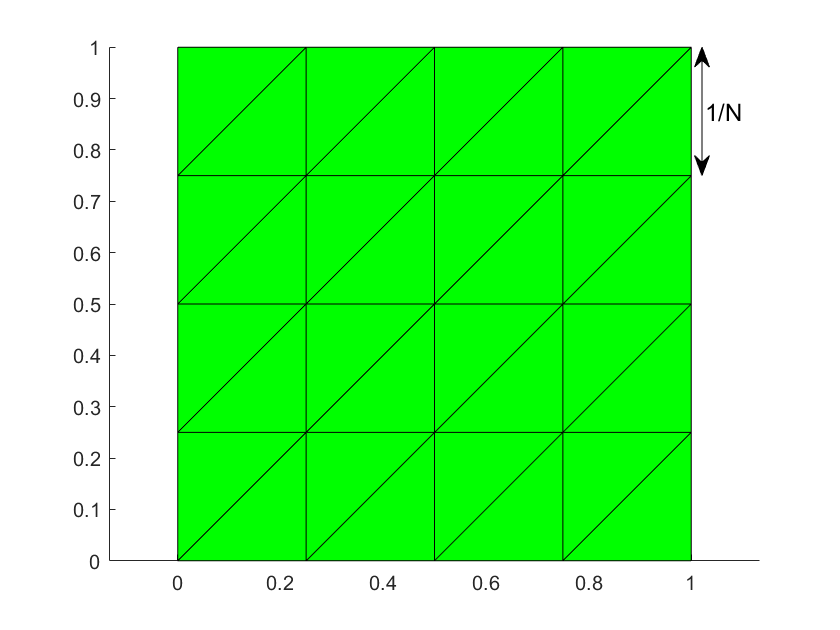} 
  }
  \subfloat[3D mesh]
  {
      \includegraphics[width=0.45\textwidth]{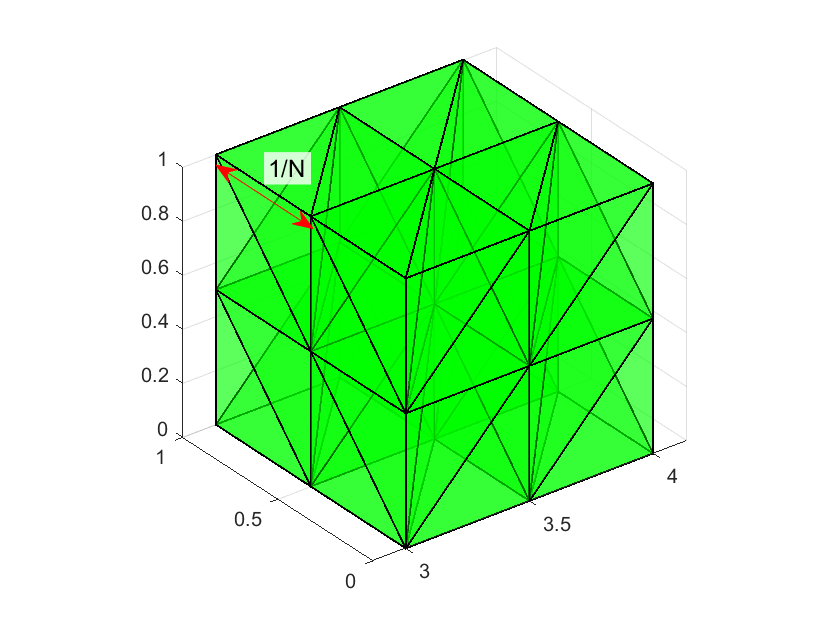} 
  }
  \caption{2D and 3D meshes}\label{2D and 3D square mesh}
\end{figure}
$\bm{V}_h$ is taken to be  the Nédélec finite element space
 of second kind. In cases with translational symmetry, \eqref{curl_advection_diffusion_equation} reduces to the two-dimensional boundary value problem posed on a domain $\Omega\subset\mathbb{R}^2$:
\begin{equation}\label{curl_advection_diffusion_equation_reduceto2D}
    \left\{
    \begin{aligned}
        \bm{R}\nabla(\varepsilon\nabla\cdot(\bm{R}\bm{u}))-\bm{R}\bm{\beta}\nabla\cdot(\bm{R}\bm{u})+\nabla(\bm{\beta}\cdot\bm{u})+\gamma\bm{u} &= \bm{f} ~~~~\text{in}~~\Omega \\
        (\bm{R}\bm{n}\cdot\bm{u})\bm{R}\bm{n} + \chi_{\Gamma^-}(\bm{u}\cdot\bm{n})\bm{n} &= \bm{g}~~~~\text{on}~~\Gamma
    \end{aligned}
    \right.
\end{equation}
where $\bm{R}=\begin{bmatrix}
    0 & 1\\
    -1 & 0
\end{bmatrix}$ is the $\frac{\pi}{2}$-rotation matrix.\par
In each experiment, unless otherwise specified, we set
$$
\alpha|_{\partial T}(x) = \left\{\begin{aligned}
    0,~~& \bm{\beta}(x)\cdot\bm{n}(x)>0    \\
    1,~~& \bm{\beta}(x)\cdot\bm{n}(x)\le0
\end{aligned}\right.
$$
for facet weight. For simplicity, the term $\rho_0\Vert\bm{u}_h\Vert_0^2$ in energy norm is replaced by $\Vert\bm{u}_h\Vert_0^2$, i.e., we take
\begin{equation}
\begin{aligned}
\norm{\bm{u}_h}^2&:=\varepsilon\Vert\nabla\times\bm{u}_h\Vert_0^2+\Vert\bm{u}_h\Vert_0^2+\sum_{T\in\mathcal{T}_h}\delta_T\Vert \tilde{L}_{\bm{\beta},h}\bm{u}_h\Vert_{0,T}^2\\
&+\frac{1}{2}\sum_{F\in\mathcal{F}_h^{\circ}}\langle \vert\llbracket{\alpha}\rrbracket\bm{\beta}\cdot\bm{n}\vert,\vert\llbracket\bm{u}_h\rrbracket\vert^2\rangle_F + \frac{1}{2}\sum_{F\in\mathcal{F}_h^{\partial}}\langle\vert\bm{\beta}\cdot\bm{n}\vert,\vert\bm{u}_h\vert^2\rangle_F.
\end{aligned}
\end{equation}
Obviously, this does not affect the convergence order.

{All numerical simulations were performed on a workstation running Ubuntu 20.04 LTS, equipped with dual \textbf{Intel Xeon Gold 6248R processors} (3.00 GHz, total 48 cores) and \textbf{192 GB of RAM}. The algorithms were implemented in the \textbf{MATLAB 2021b} environment utilizing the \textbf{iFEM} software package~\cite{chen2008ifem}.}

\subsection{Example 1: 3D convergence tests}\label{Numerical_experiments_example1}
Let $\Omega=(0,1)^3$. We set $\gamma=8$ and $\bm{\beta}=[1-z/2,2+x,3-y]^T$, a pair that satisfies condition \eqref{well_posedness_assumption_formulation}, and the source term $\bm{f}$ and boundary data $\bm{g}$ are chosen such that the exact solution is
$$
\bm{u}(x,y,z) = \begin{bmatrix}
    ye^{xz}  \\
    -x^2y  \\
    \sin(xyz)
\end{bmatrix},
$$
We take $\delta_T=0.4/N$ and test the convergence behavior for various $\varepsilon$. For comparison, the numerical results obtained by the standard Galerkin method \eqref{standard_Galerkin_method} are also listed. The results are shown in Table \ref{3D_numerical_results_table}.\par
\begin{table}[!htbp]
\centering
\begin{tabular}{ |c|cc|c|cccc|cc| } 
 \hline
 \multirow{2}{1em}{$\varepsilon$}&\multirow{2}{1em}{$k$}&\multirow{2}{1em}{$N$}&\multirow{2}{5.5em}{{\#free DoFs}}&\multicolumn{4}{|c|}{SUPG}&\multicolumn{2}{|c|}{No Stabilization}\\
 \cline{5-8}\cline{9-10}
  &  &&  & $\Vert\bm{e}_h\Vert_{0}$ & order & $\norm{\bm{e}_h}$ & order & $\Vert\bm{e}_h\Vert_{0}$ & order \\ 
 \hline
 \multirow{8}{2em}{1e-4} & \multirow{4}{0.5em}{\vspace{0em}$1$} & 2 &{52}& 7.6666e-2 & - & 6.6495e-1 & - & 8.2355e-2 &-\\ 
 && 4 &{632}& 1.7471e-2 & 2.13 & 2.1008e-1 &1.66 & 2.8146e-2 &1.55\\ 
 && 8 &{6,064}& 3.9708e-3 & 2.14 & 7.0171e-2 &1.58 & 1.0235e-2 &1.46\\ 
 && 16 &{52,832}& 9.2980e-4 & 2.09 & 2.4141e-2 &1.54 & 4.0153e-3 &1.35\\ 
 \cline{2-8}\cline{9-10}
 &\multirow{4}{0.5em}{\vspace{0em}$2$} & 2 &{294}& 5.9608e-3 & - & 7.7253e-2 & - & 8.5071e-3 &-\\ 
 && 4 &{2964}& 7.3609e-4 &3.02  & 1.3485e-2 &2.52  & 1.5806e-3 &2.43\\ 
 && 8 &{26,376}& 9.3468e-5 &2.98  & 2.3432e-3 &2.52  & 3.0638e-4 &2.37\\ 
 && 16 &{222,096}& 1.2012e-5 &2.96  & 4.0922e-4 &2.52  & 5.8485e-5 &2.39\\ 
 \hline
 \multirow{8}{2em}{1e-6} & \multirow{4}{0.5em}{\vspace{0em}$1$} & 2 &{52}& 7.6665e-2 & - & 6.6493e-1 & - & 8.2367e-2 &-\\ 
 && 4 &{632}& 1.7470e-2 &2.13  & 2.1006e-1 &1.66 & 2.8168e-2 &1.55\\ 
 && 8 &{6,064}& 3.9705e-3 &2.14  & 7.0156e-2 &1.58 & 1.0263e-2 &1.46\\ 
 && 16 &{52,832}& 9.2975e-4 &2.09  & 2.4130e-2 &1.54  & 4.0581e-3 &1.34\\ 
 \cline{2-8}\cline{9-10}
 &\multirow{4}{0.5em}{\vspace{0em}$2$} & 2 &{294}& 5.9604e-3 & - & 7.7249e-2 & - & 8.5171e-3 &-\\ 
 && 4 &{2,964}& 7.3607e-4 &3.02  & 1.3483e-2 &2.52  & 1.5912e-3 &2.42\\ 
 && 8 &{26,376}& 9.3517e-5 &2.98  & 2.3428e-3 &2.52  & 3.1752e-4 &2.33\\ 
 && 16 &{222,096}& 1.2037e-5 &2.96  & 4.0912e-4 &2.52  & 6.8196e-5 &2.22\\ 
 \hline
\end{tabular}
\caption{Example1: Convergence results for the 3D $\bm{H}(\mathrm{curl})$ problem with a smooth exact solution.}\label{3D_numerical_results_table}
\end{table}

The numerical experiments confirm that the convergence order of the energy norm is $k+\frac{1}{2}$ in the advection-dominated case, which is consistent with the theoretical analysis. Additionally, the results indicate $(k + 1)$-th order convergence in the $L^2$ norm. In contrast, unstabilized results exhibit reduced convergence orders in the advection-dominated cases. The contrast between the stabilized and unstabilized results clearly demonstrates the effectiveness of the proposed stabilization strategy in achieving the expected convergence rates and improving overall accuracy.

\subsection{Example 2: 2D convergence tests}\label{Numerical_experiments_example2}
Let $\Omega=(0,1)^2$ and set $\gamma=1$. The advective field is given by $\bm{\beta} = [y-1/2,-x+1/2]^T$, which satisfies condition \eqref{well_posedness_assumption_formulation}. The source term $\bm{f}$ and boundary data $\bm{g}$ are chosen such that the exact solution is given by
$$
\bm{u}(x,y) = \begin{bmatrix}
    16x(1-x)y(1-y)\\
    e^x\sin(\pi x)\sin(\pi y)
\end{bmatrix}.
$$
We set $\delta_T = 0.4/N$ for the SUPG method and examine its convergence across varying $\varepsilon$, while concurrently evaluating the standard Galerkin method \eqref{standard_Galerkin_method} for comparison. The results are shown in Table \ref{2D_numerical_results_table}.\par
\begin{table}[!htbp]
\centering 
\begin{tabular}{ |c|cc|c|cccc|cc| } 
 \hline
 \multirow{2}{1em}{$\varepsilon$}&\multirow{2}{1em}{$k$}&\multirow{2}{1em}{$N$}&\multirow{2}{5.5em}{{\#free DoFs}}&\multicolumn{4}{|c|}{SUPG}&\multicolumn{2}{|c|}{No Stabilization}\\
 \cline{5-8}\cline{9-10}
  &  & &  & $\Vert\bm{e}_h\Vert_{0}$ & order & $\norm{\bm{e}_h}$ & order & $\Vert\bm{e}_h\Vert_{0}$ & order \\ 
 \hline
 \multirow{10}{2em}{1e-4} &\multirow{5}{0.5em}{\vspace{0em}$1$} & 8&{352} & 1.8901e-2 &  -& 6.6223e-2 & - & 4.4344e-2 &-  \\ 
 && 16&{1,472} & 4.5262e-3 & 2.06 & 2.3686e-2 & 1.48 & 1.4469e-2 &1.62  \\ 
 && 32&{6,016} & 1.1114e-3 & 2.03 & 8.4580e-3 & 1.49 & 5.2422e-3 &1.46  \\ 
 && 64&{24,320} & 2.7620e-4 & 2.01 & 3.0461e-3 & 1.47 & 2.3347e-3 &1.17  \\ 
 && 128&{97,792} & 6.8884e-5 &2.00 & 1.1147e-3 & 1.45 & 1.1239e-3 &1.05  \\ 
 \cline{2-8}\cline{9-10}
 &\multirow{5}{0.5em}{\vspace{0em}$2$} & 8&{912} & 9.5291e-4 & - & 4.7395e-3 & -& 2.2541e-3 &- \\ 
 && 16 &{3,744}& 1.1932e-4 & 3.00 & 8.5531e-4 & 2.47 & 3.6589e-4 &2.62\\ 
 && 32 &{15,168}& 1.4385e-5 & 3.05 & 1.5258e-4 & 2.49 & 4.3877e-5 &3.06\\ 
 && 64 &{61,056}& 1.7059e-6 & 3.08 & 2.7205e-5 & 2.49 & 5.0226e-6 &3.13\\ 
 && 128 &{244,992}& 2.0428e-7 & 3.06 & 4.9074e-6 & 2.47 & 5.3127e-7 &3.24\\ 
 \hline
 \multirow{10}{2em}{1e-6}& \multirow{5}{0.5em}{\vspace{0em}$1$} & 8 &{352}& 1.8923e-2 &-  & 6.5878e-2 & - & 4.5735e-2 & - \\ 
 && 16 &{1,472}& 4.5381e-3 & 2.06 & 2.3435e-2 & 1.49 & 1.5914e-2 &1.52  \\ 
 && 32 &{6,016}& 1.1168e-3 & 2.02 & 8.2808e-3 & 1.50 & 6.1235e-3 &1.38  \\ 
 && 64 &{24,320}& 2.7822e-4 & 2.01 & 2.9241e-3 & 1.50 & 2.9370e-3 &1.06  \\ 
 && 128 &{97,792}& 6.9577e-5 & 2.00 & 1.0332e-3 & 1.50 & 1.5037e-3 &0.97  \\ 
 \cline{2-8}\cline{9-10}
 &\multirow{5}{0.5em}{\vspace{0em}$2$} & 8&{912} & 9.8305e-4 & - & 4.7249e-3 & - & 2.3732e-3 &-\\ 
 && 16&{3,744} & 1.3192e-4 & 2.90 & 8.5139e-4 & 2.47 & 4.4021e-4 &2.43\\ 
 && 32&{15,168} & 1.7674e-5 & 2.90 & 1.5164e-4 & 2.49 & 7.6255e-5 &2.53\\ 
 && 64&{61,056} & 2.3639e-6 & 2.90 & 2.6911e-5 & 2.49 & 1.5492e-5 &2.30\\ 
 && 128&{244,992} & 3.1062e-7 & 2.93 & 4.7698e-6 & 2.50 & 2.9279e-6 &2.40\\ 
 \hline
\end{tabular}
\caption{Example 1: Convergence results for the 2D $\bm{H}(\mathrm{curl})$ problem with a smooth exact solution.}\label{2D_numerical_results_table}
\end{table}
\par
Similarly to Example 1, in the advection-dominated case, the solution obtained by the SUPG method exhibits convergence orders of $k+\frac{1}{2}$ in the energy norm and $k+1$ in the $L^2$-norm. However, for the standard Galerkin method, convergence rates deteriorate significantly as $\varepsilon \to 0$. These results demonstrate the effectiveness of the SUPG method in achieving the theoretically predicted convergence rates.

\subsection{Example 3: Clarification of stabilization terms}
To explore the roles of the stabilization terms $S_h^1$ and $S_h^2$, we repeat the experiments from Examples 1 and 2 with $\alpha^{\pm}$ taken to be $\frac{1}{2}$ on all interior facets. This means that the bilinear form is $a_0(\bm{u}_h,\bm{v}_h) + S_h^2(\bm{u}_h,\bm{v}_h)$ and $\tilde{L}_{\bm{\beta},h}\cdot = {L}_{\bm{\beta},h}\cdot - \bm{r}(\bm{\beta}\cdot\bm{n}^+\llbracket\cdot\rrbracket)$. For $\varepsilon=10^{-6}$ and $k=2$, the numerical results are presented in Tables \ref{Clarification of stabilization terms 2D} and \ref{Clarification of stabilization terms 3D}. In this case, the energy norm is defined as
\begin{equation}
\begin{aligned}
    \nnorm{\bm{v}_h}_{S_h^2}^2:=\varepsilon\Vert\nabla\times\bm{v}_h\Vert_0^2 + \Vert\bm{v}_h\Vert_0^2+\sum_{T\in\mathcal{T}_h}\delta_T\Vert\tilde{L}_{\bm{\beta},h}\bm{v}_h\Vert_{0,T}^2+\frac{1}{2}\sum_{F\in\mathcal{F}_h^\partial}\langle\vert\bm{\beta}\cdot\bm{n}\vert,\vert\bm{v}_h\vert^2\rangle_F.
\end{aligned}
\end{equation}
For comparison, previous results obtained using the $S_h^1+S_h^2$ stabilization are listed. Additionally, results for the \textit{$S_h^1$ only} configuration are presented, which corresponds to setting $\delta_T = 0$.
\begin{table}[!htbp]
\centering 
\begin{tabular}{ |c|cccc|cccc|cc| } 
\hline
\multirow{2}{1em}{$N$} & \multicolumn{4}{|c|}{$S_h^1+S_h^2$} & \multicolumn{4}{|c|}{$S_h^2$ only} & \multicolumn{2}{|c|}{$S_h^1$ only}  \\
\cline{2-5}\cline{6-9}\cline{10-11}
& $\Vert\bm{e}_h\Vert_0$ & order & $\norm{\bm{e}_h}$ & order & $\Vert\bm{e}_h\Vert_0$ & order & $\nnorm{\bm{e}_h}_{S_h^2}$ & order& $\Vert\bm{e}_h\Vert_0$ & order \\
\hline
2 &5.9604e-3&-&7.7249e-2&-&6.4875e-3&-&6.9287e-2&-&7.5347e-3&-  \\
\cline{2-11}
4 &7.3607e-4&3.02 &1.3483e-2&2.52 &8.5512e-4&2.92 &1.1674e-2&2.57 &1.1722e-3&2.68   \\
\cline{2-11}
8 &9.3517e-5&2.98 &2.3428e-3&2.52 &1.1071e-4&2.95 &1.9682e-3&2.57 &2.1488e-4&2.45   \\
\cline{2-11}
16 &1.2037e-5&2.96 &4.0912e-4&2.52 &1.4107e-5&2.97 &3.3686e-4&2.55 &4.8061e-5&2.16   \\
\hline
\end{tabular}
\caption{Example 3: Convergence results for the 3D $\bm{H}(\mathrm{curl})$ problem with $\varepsilon=10^{-6}, k=2$ and various combinations of stabilization terms.}\label{Clarification of stabilization terms 2D}
\end{table}
\par
For both numerical examples, the \textit{$S_h^2$-only} configuration achieves convergence orders of $k+\frac{1}{2}$ in the energy norm and $k+1$ in the $L^2$-norm, with errors remaining robust for vanishing $\varepsilon$. In contrast, solutions computed using the \textit{$S_h^1$-only} configuration exhibit lower convergence rates in the $L^2$-norm. These results highlight the accuracy improvement achieved by the newly proposed term $S_h^2(\cdot,\cdot)$.

\begin{table}[!htbp]
\centering 
\begin{tabular}{ |c|cccc|cccc|cc| } 
\hline
\multirow{2}{1em}{$N$} & \multicolumn{4}{|c|}{$S_h^1+S_h^2$} & \multicolumn{4}{|c|}{$S_h^2$ only} & \multicolumn{2}{|c|}{$S_h^1$ only}  \\
\cline{2-5}\cline{6-9}\cline{10-11}
& $\Vert\bm{e}_h\Vert_0$ & order & $\norm{\bm{e}_h}$ & order & $\Vert\bm{e}_h\Vert_0$ & order & $\nnorm{\bm{e}_h}_{S_h^2}$ & order& $\Vert\bm{e}_h\Vert_0$ & order \\
\hline
8 &9.8305e-4&-&4.7249e-3&-& 8.2524e-4 & - & 3.4209e-3 & - &1.4479e-3&-  \\
16 &1.3192e-4&2.90&8.5139e-4&2.47& 1.0354e-4 & 2.99 & 5.6812e-4 & 2.59 &2.4599e-4& 2.56 \\
32 &1.7674e-5&2.90&1.5164e-4&2.49& 1.3480e-5 & 2.94 & 9.6190e-5 & 2.56 &4.8503e-5& 2.34 \\
64 &2.3639e-6&2.90&2.6911e-5&2.49& 1.7755e-6 & 2.92 & 1.6591e-5 & 2.54 &1.1041e-5& 2.14 \\
128 &3.1062e-7&2.93&4.7698e-6&2.50& 2.3259e-7 & 2.93 & 2.8969e-6 & 2.52 &2.3929e-6& 2.20 \\
\hline
\end{tabular}
\caption{Example 3: Convergence results for the 2D $\bm{H}(\mathrm{curl})$ problem with $\varepsilon=10^{-6}, k=2$ and various combinations of stabilization terms.}\label{Clarification of stabilization terms 3D}
\end{table}

{

\subsection{Example 4: Convergence with respect to degree of polynomial}\label{Numerical_experiments_example4}
In this numerical experiment, we examine the convergence with respect to the polynomial degrees. Set $\Omega = (0,1)^2, \varepsilon = 10^{-6}, \bm \beta = [y-1/2,-x+1/2]^T, \gamma = 1$, and the source term $\bm f$ and boundary data $\bm g$ are chosen such that the exact solution is
$$
\bm u(x,y) = \left[\begin{aligned}
    & 32x(1-x)y(1-y)\sin(x+y)\\
    &~~~~~~~ e^x\sin(\pi x)\sin(\pi y)
\end{aligned}\right].
$$
Set $\delta_T = 0.4/N$ for the SUPG method. For several fixed meshes with different refinement levels $N$, we plot the error in the energy norm $\norm{\bm{e}_h}$ against the polynomial degree $k$ of the finite element space, as shown in Figure \ref{high_order_polynomial_energy_error}. Numerically, it can be observed that on a fixed mesh, the proposed SUPG method converges as the polynomial degree $k$ increases. This behavior is consistent with observations for scalar problems.
\begin{figure}[!htbp]
\centering
\includegraphics[width=0.5\textwidth]{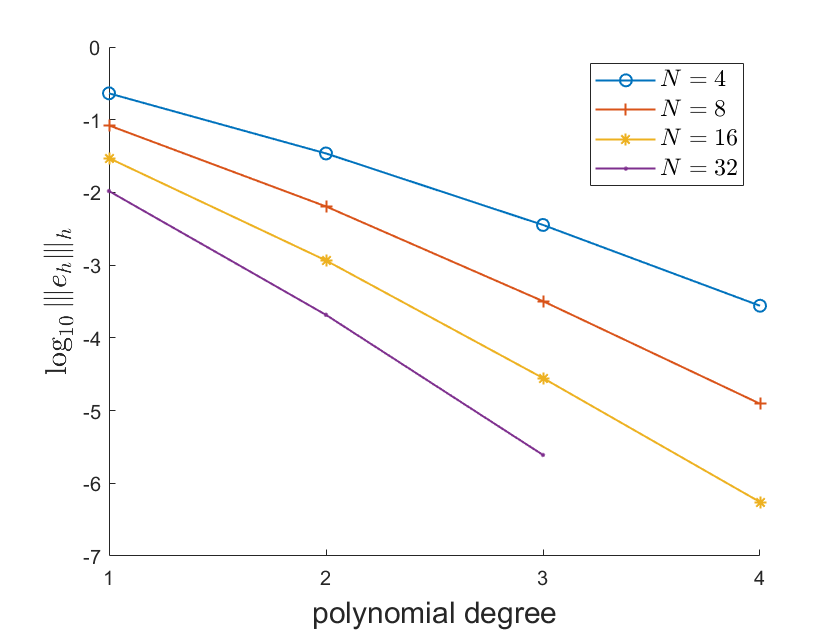} 
\caption{Example 4: Dependence of error of the SUPG scheme on $N$ and polynomial degree $k$ with $\delta_T = 0.4/N$.} \label{high_order_polynomial_energy_error}
\end{figure}

\subsection{Example 5: Numerical performance in complex regions}\label{Numerical_experiments_example5}
In this numerical experiment, we examine the performance of the SUPG method on complex domains. Set $\varepsilon = 10^{-6}, \bm \beta = [y^2/2+2;-x^2/2-1/2]^T, \gamma = 4$, and the source term $\bm f$ and boundary data $\bm g$ are chosen such that the exact solution is
$$
\bm u(x,y) = \left[\begin{aligned}
    & (1-x)(y+1)\cos\pi(x^2+y)\\
    & ~~~~~~-e^x\sin\pi(x+y)
\end{aligned}\right].
$$
We perform numerical tests with the SUPG parameter set to $\delta_T = 0.4/N$, and compare the results with the standard Galerkin method. Simulations are conducted on both a hexagonal domain and an L-shaped domain, with the corresponding meshes and the definition of $N$ illustrated in Figure \ref{complex regions: meshes and N}. The numerical results are presented in Table \ref{hexagon convergence} and \ref{Lshape convergence}.
\begin{figure}[!htbp]    
  \centering            
  \subfloat[hexagonal domain]
  {
      \includegraphics[width=0.45\textwidth]{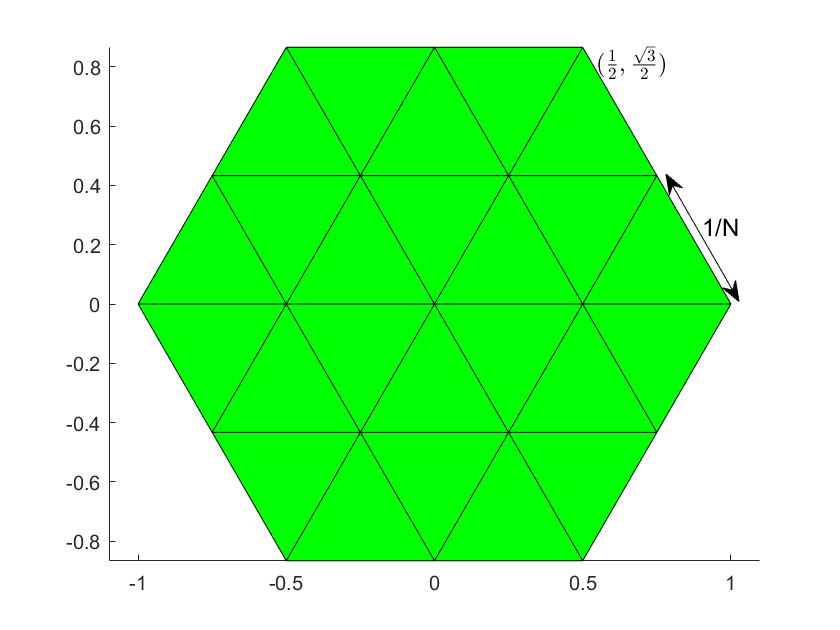} 
  }
  \subfloat[L-shaped domain]
  {
      \includegraphics[width=0.45\textwidth]{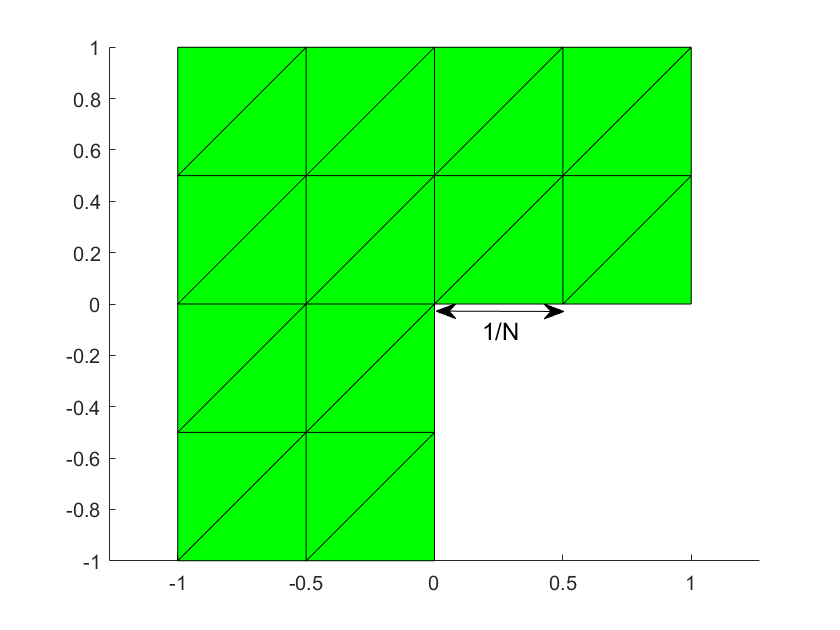} 
  }
  \caption{Example 5: Complex regions}\label{complex regions: meshes and N}
\end{figure}
\begin{table}[!htbp]
\centering
\begin{tabular}{ |c|cc|c|cccc|cc| } 
 \hline
 \multirow{2}{1em}{$\varepsilon$}&\multirow{2}{1em}{$k$}&\multirow{2}{1em}{$N$}&\multirow{2}{5.5em}{\#free DoFs}&\multicolumn{4}{|c|}{SUPG}&\multicolumn{2}{|c|}{No Stabilization}\\
 \cline{5-8}\cline{9-10}
  &  & &  & $\Vert\bm{e}_h\Vert_{0}$ & order & $\norm{\bm{e}_h}$ & order & $\Vert\bm{e}_h\Vert_{0}$ & order \\ 
 \hline
 \multirow{10}{2em}{1e-6} &\multirow{5}{0.5em}{\vspace{0em}$1$} & 4&264 & 9.7349e-2 & - & 1.0216e+0 & - & 1.9810e-1 & - \\ 
 && 8&1,104 & 2.4726e-2 & 1.98 & 3.8395e-1 & 1.41 & 9.8820e-2 & 1.00 \\ 
 && 16&4,512 & 6.4028e-3 & 1.95 & 1.3794e-1 & 1.48 & 4.9706e-2 & 0.99 \\ 
 && 32&18,240 & 1.6589e-3 & 1.95 & 4.8990e-2 & 1.49 & 2.4942e-2 & 0.99 \\ 
 && 64&73,344 & 4.2407e-4 &1.97 & 1.7347e-2 & 1.50 & 1.2493e-2 & 1.00 \\ 
 \cline{2-8}\cline{9-10}
 &\multirow{5}{0.5em}{\vspace{0em}$2$} & 4&684 & 1.2173e-2 & - & 1.9134e-1 &- & 2.3924e-2  & -\\ 
 && 8 &2,808& 1.5190e-3 & 3.00 & 3.2587e-2 & 2.55 & 5.4797e-3 &2.13\\ 
 && 16 &11,376& 1.8940e-4 & 3.00 & 5.6369e-3 & 2.53 & 1.3471e-3 &2.02\\ 
 && 32 &45,792& 2.4221e-5 & 2.97 & 9.8976e-4 & 2.51 & 3.3381e-4 &2.01\\ 
 && 64 &183,744& 3.2902e-6 & 2.88 & 1.7580e-4 & 2.49 & 8.2087e-5 &2.02\\ 
 \hline
\end{tabular}
\caption{Example 5: Convergence results on Hexagonal domain with a smooth exact solution.}\label{hexagon convergence}
\end{table}
\begin{table}[!htbp]
\centering
\begin{tabular}{ |c|cc|c|cccc|cc| } 
 \hline
. \multirow{2}{1em}{$\varepsilon$}&\multirow{2}{1em}{$k$}&\multirow{2}{1em}{$N$}&\multirow{2}{5.5em}{\#free DoFs}&\multicolumn{4}{|c|}{SUPG}&\multicolumn{2}{|c|}{No Stabilization}\\
 \cline{5-8}\cline{9-10}
  &  & &  & $\Vert\bm{e}_h\Vert_{0}$ & order & $\norm{\bm{e}_h}$ & order & $\Vert\bm{e}_h\Vert_{0}$ & order \\ 
 \hline
 \multirow{10}{2em}{1e-6} &\multirow{5}{0.5em}{\vspace{0em}$1$} & 4&256 & 2.2070e-1 & - & 1.9839e+0 & - & 3.5037e-1 & - \\ 
 && 8&1,088 & 4.9611e-2 & 2.15 & 7.4510e-1 & 1.41 & 1.6479e-1 &1.09  \\ 
 && 16&4,480 & 1.0957e-2 & 2.18 & 2.6654e-1 & 1.48 & 6.9347e-2 &1.25  \\ 
 && 32&18,176 & 2.5803e-3 & 2.09 & 9.4256e-2 & 1.50 & 3.1741e-2 &1.13  \\ 
 && 64&73,216 & 6.3132e-4 &2.03 & 3.3285e-2 & 1.50 & 1.5530e-2 &1.03  \\ 
 \cline{2-8}\cline{9-10}
 &\multirow{5}{0.5em}{\vspace{0em}$2$} & 4&672 & 2.3705e-2 & - & 3.2517e-1 & -&  5.2408e-2& -\\ 
 && 8 &2,784& 2.8395e-3 & 3.06 & 5.9225e-2 & 2.46 & 9.5513e-3 &2.46\\ 
 && 16 &11,328& 3.3927e-4 & 3.07 & 1.0564e-2 & 2.49 & 2.2000e-3 &2.12\\ 
 && 32 &45,696& 4.2979e-5 & 2.98 & 1.8711e-3 & 2.50 & 5.7367e-4 &1.94\\ 
 && 64 &183,552& 5.9320e-6 & 2.86 & 3.3134e-4 &2.50  & 1.2379e-4 &2.21\\ 
 \hline
\end{tabular}
\caption{Example 5: Convergence results on L-shaped domain with a smooth exact solution.}\label{Lshape convergence}
\end{table}

These results further validate that the energy error achieves the theoretical convergence rate of $k+\frac{1}{2}$. Moreover, comparison with the standard Galerkin results demonstrates the necessity of the stabilization approach.

}

\subsection{Example 6: Boundary layers}
In this numerical experiment, we evaluate the oscillation-suppression capability of the SUPG method by solving a two-dimensional boundary layer problem. The problem configuration is specified as follows: $\Omega = (0,1)^2$, $\varepsilon=10^{-6}$, $\bm{\beta} = [1,2]^T$, $\gamma=0$, with source term $\bm{f}=[1,1]^T$ and homogeneous boundary condition $\bm{g}=\bm{0}$. We plot the first component of the approximate solutions obtained by different numerical methods.
\par
{When computed with the standard Galerkin method, as illustrated in Figure~\ref{standard_Galerkin_bad_sect3}, the solution exhibits significant numerical oscillations, leading to a substantial degradation in accuracy.} As shown in Figure~\ref{upwind_method_figure_N16_N64}, the \textit{$S_h^1$ only} method provides inadequate stabilization, as spurious oscillations persist under mesh refinement. In contrast, Figure \ref{SUPG_bdlayer_figure_N16_N64} clearly shows that the SUPG method yields stable solutions with significantly reduced oscillations as $\varepsilon \to 0$. {Moreover, Figure \ref{SUPG_bdlayer_figure_N16_k=2} reveals that the stabilization effect of the SUPG method is further enhanced in higher-order finite element spaces}. This result confirms the superior stabilization capability of our proposed method.


\begin{figure}[!htbp]    
  \centering            
  \subfloat[$N=16$]
  {
      \includegraphics[width=0.45\textwidth]{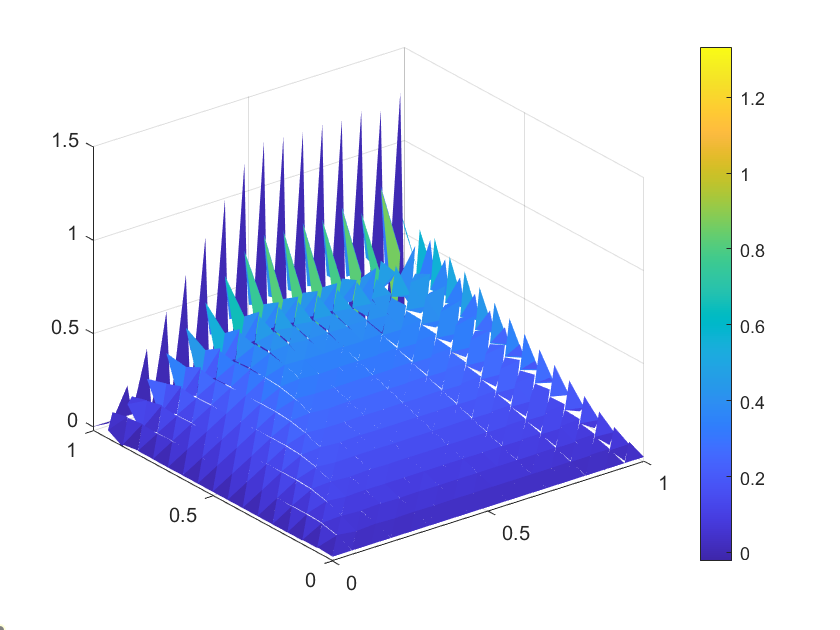} \label{upwind_method_figure_N16} 
  }
  \subfloat[$N=64$]
  {
      \includegraphics[width=0.45\textwidth]{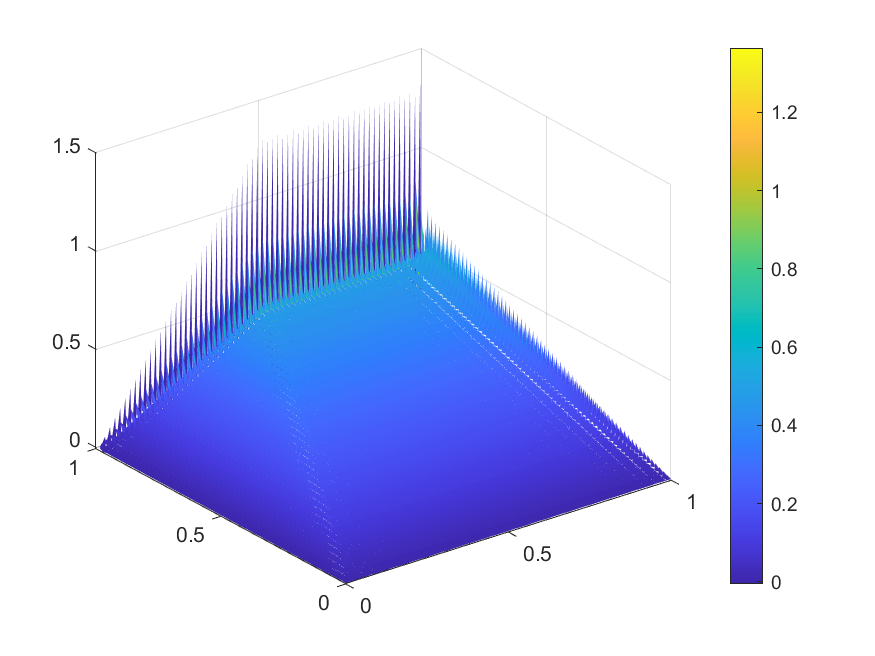} \label{upwind_method_figure_N64} 
  }
  \caption{Example 6. Numerical solution obtained by the \textit{$S_h^1$ only} method ($k=1$).}\label{upwind_method_figure_N16_N64}
\end{figure}

\begin{figure}[!htbp]    
  \centering            
  \subfloat[$N=16$]   
  {
      \includegraphics[width=0.45\textwidth]{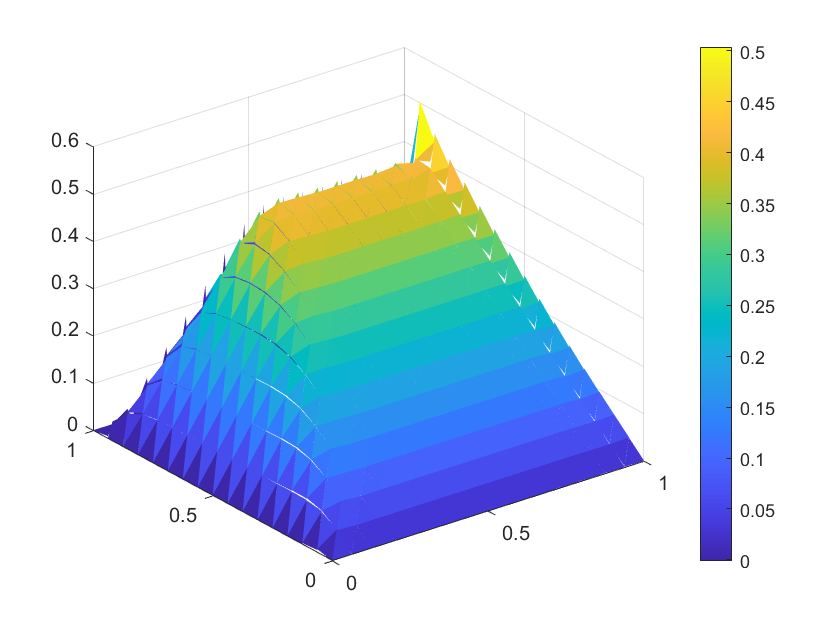}
  }
  \subfloat[$N=64$]
  {
      \includegraphics[width=0.45\textwidth]{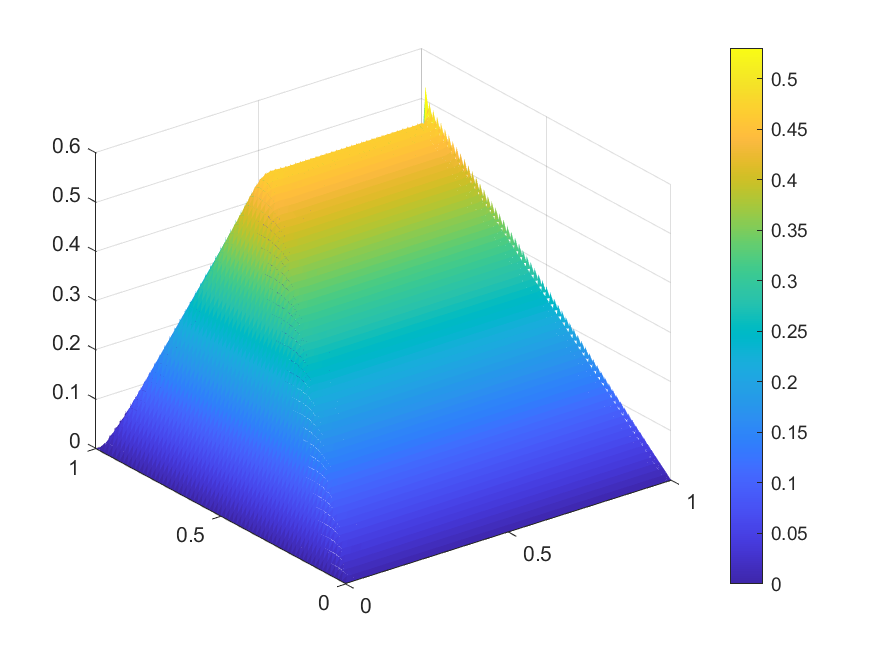}
  }
  \caption{Example 6. Numerical solution $u_h$ obtained by the SUPG method \eqref{magnetic_SUPG_scheme} ($k=1, \delta_T=0.4/N$).}\label{SUPG_bdlayer_figure_N16_N64}
\end{figure}

\begin{figure}[!htbp]
\centering
\includegraphics[width=0.45\textwidth]{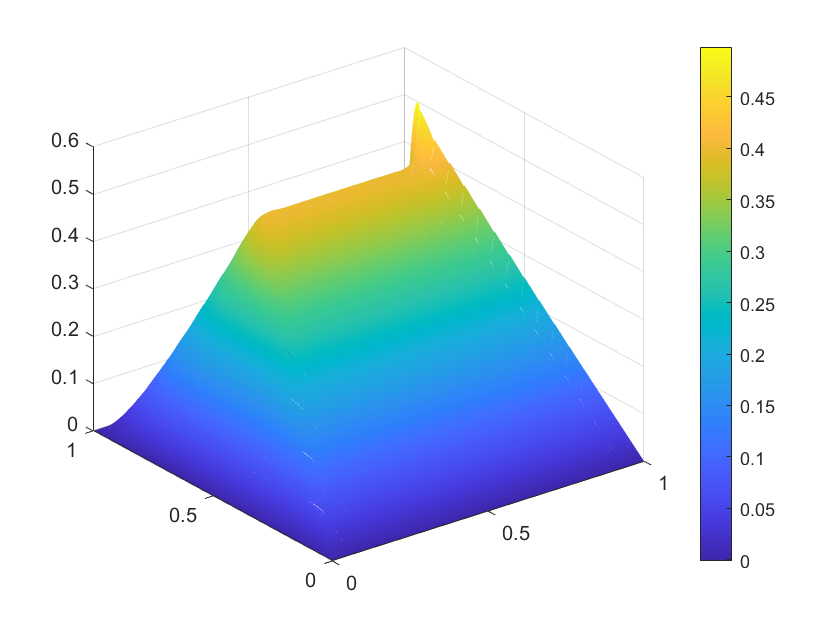} 
\caption{{Example 6. Numerical solution $u_h$ obtained by the SUPG method \eqref{magnetic_SUPG_scheme} ($k=2, N=16, \delta_T=0.4/N$).}} 
\label{SUPG_bdlayer_figure_N16_k=2}
\end{figure}

\subsection{Example 7: Internal layers}
In this numerical experiment, we implement the SUPG method for a two-dimensional internal layer problem. The computational domain is $\Omega = (0,1)^2$. The problem parameters are $\varepsilon = 10^{-3}$, $\bm{\beta} = [1,0]^T$, and $\gamma = 0$. The piecewise-defined source term is
\[
\bm{f}(y) = 
\begin{cases} 
[1,1]^T & \text{for } 0.25 < y < 0.75, \\
[0,0]^T & \text{otherwise}.
\end{cases}
\]
Homogeneous Dirichlet boundary conditions are applied. The stabilization parameters are set to $\delta_T = 0.4/N$. We present the first components of the numerical solutions, along with their cross-sections at $x=0.5$.

Figure \ref{SUPG_intlayer_figure} shows that the SUPG method effectively controls spurious oscillations, restricting their propagation into smooth solution regions and keeping their amplitudes at an moderate level. {Moreover, the results in Figure~\ref{SUPG_intlayer_figure_highorder} indicate that increasing the order of the polynomial space leads to a further reduction in spurious oscillations.}

\begin{figure}[!htbp]    
  \centering  
  \subfloat[Numerical solution $\bm{u}_h^{(1)}$, $N=16, k=1$]
  {
      \includegraphics[width=0.45\textwidth]{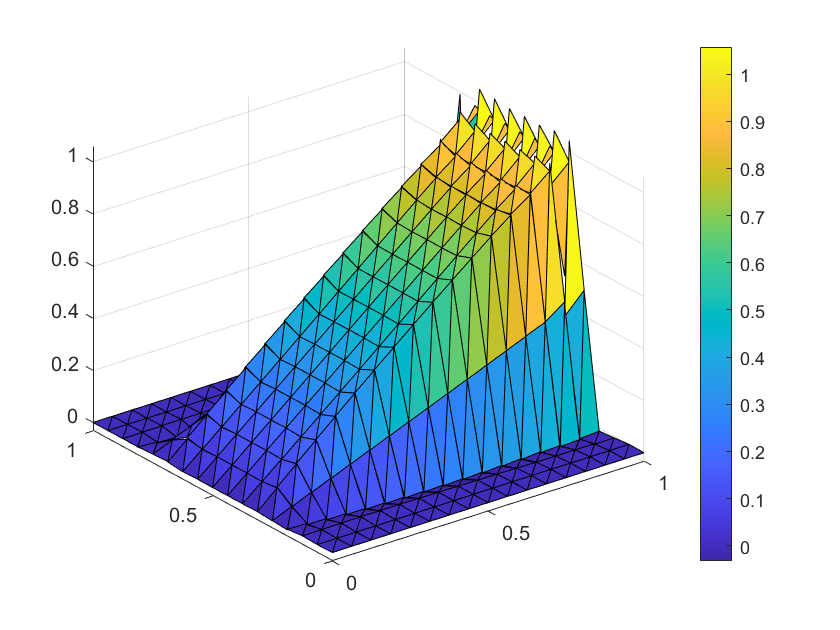}\label{SUPG_intlayer_figure_N16_numericalsol}
  }
  \subfloat[$\bm{u}_h^{(1)}(0.5,y)$]
  {
      \includegraphics[width=0.45\textwidth]{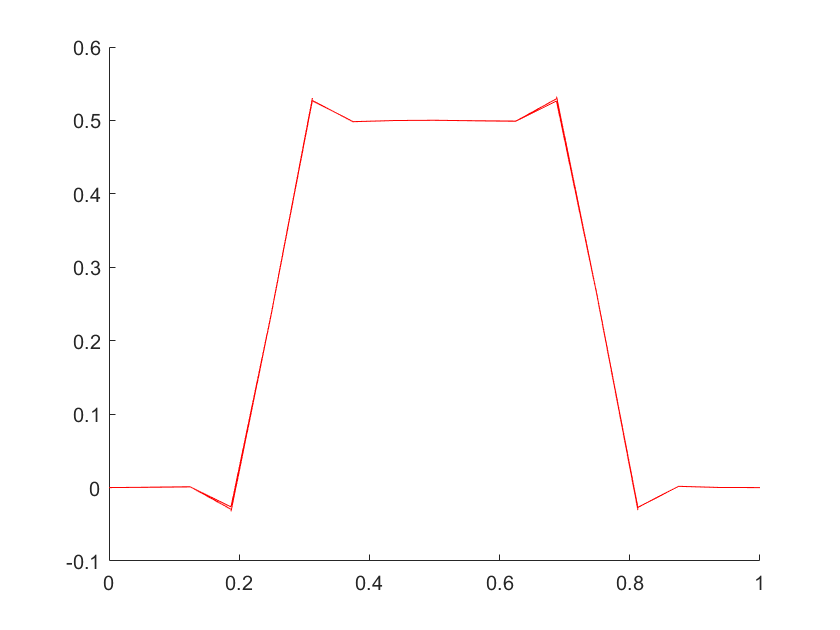}\label{SUPG_intlayer_figure_N16_x=0.5}
  }
  \\
  \subfloat[Numerical solution $\bm{u}_h^{(2)}$, $N=32, k=1$]
  {
      \includegraphics[width=0.45\textwidth]{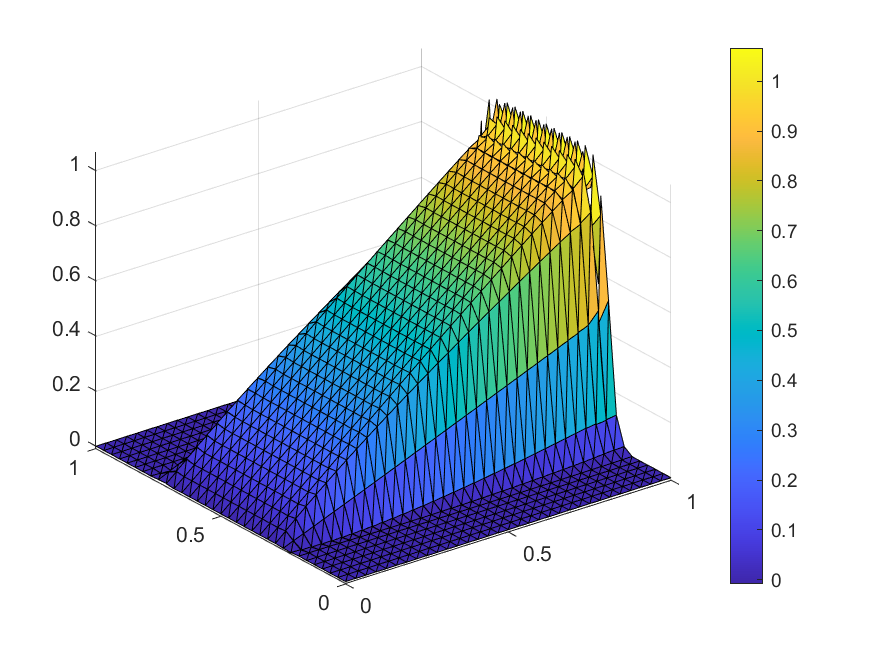}\label{SUPG_intlayer_figure_N32_numericalsol}
  }
  \subfloat[$\bm{u}_h^{(2)}(0.5,y)$]
  {
      \includegraphics[width=0.45\textwidth]{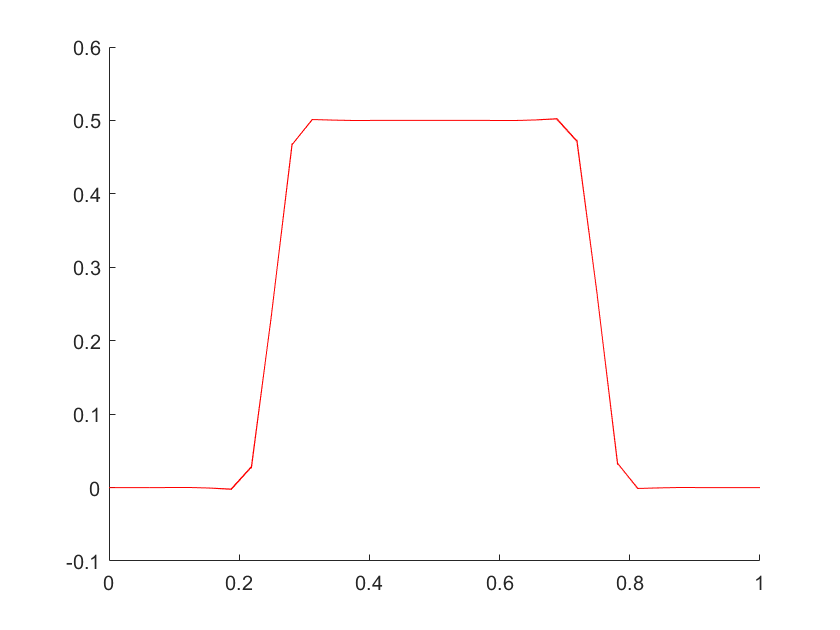}\label{SUPG_intlayer_figure_N32_x=0.5}
  }

  \caption{Example 7. Numerical solutions of the internal layer problem.}\label{SUPG_intlayer_figure}
\end{figure}

\begin{figure}[!htbp]    
  \centering  
  \subfloat[Numerical solution $\bm{u}_h^{(3)}$, $N=16, k=2$]
  {
      \includegraphics[width=0.45\textwidth]{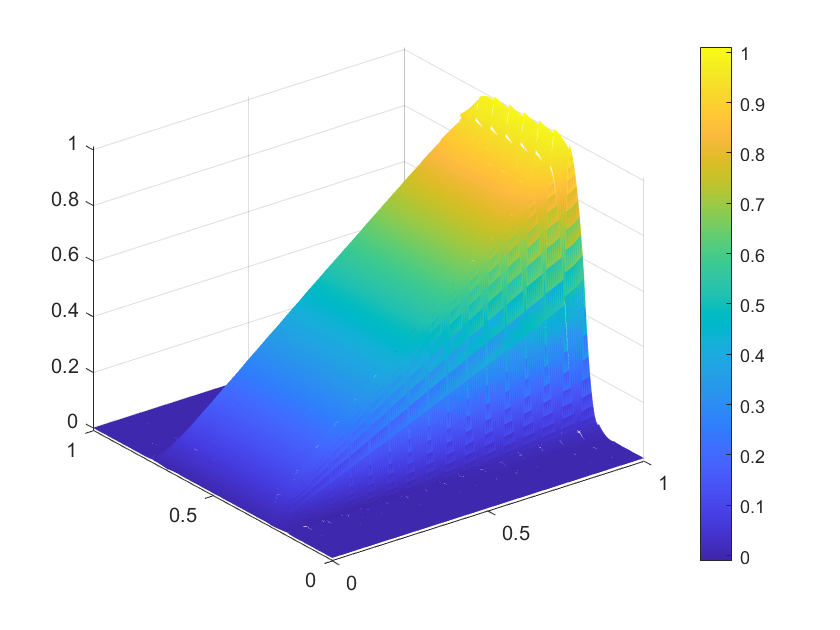}
  }
  \subfloat[$\bm{u}_h^{(3)}(0.5,y)$]
  {
      \includegraphics[width=0.45\textwidth]{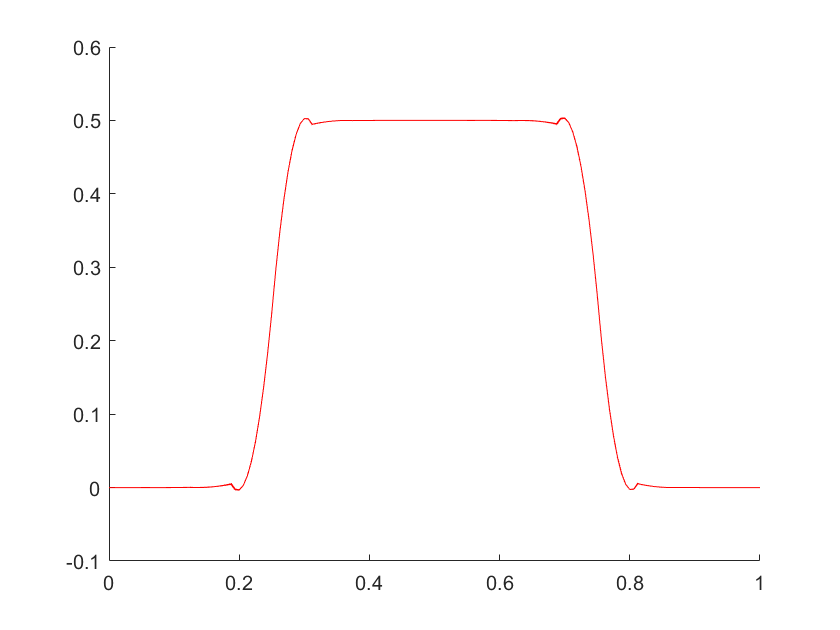}
  }
  \\
  \subfloat[Numerical solution $\bm{u}_h^{(4)}$, $N=16, k=3$]
  {
      \includegraphics[width=0.45\textwidth]{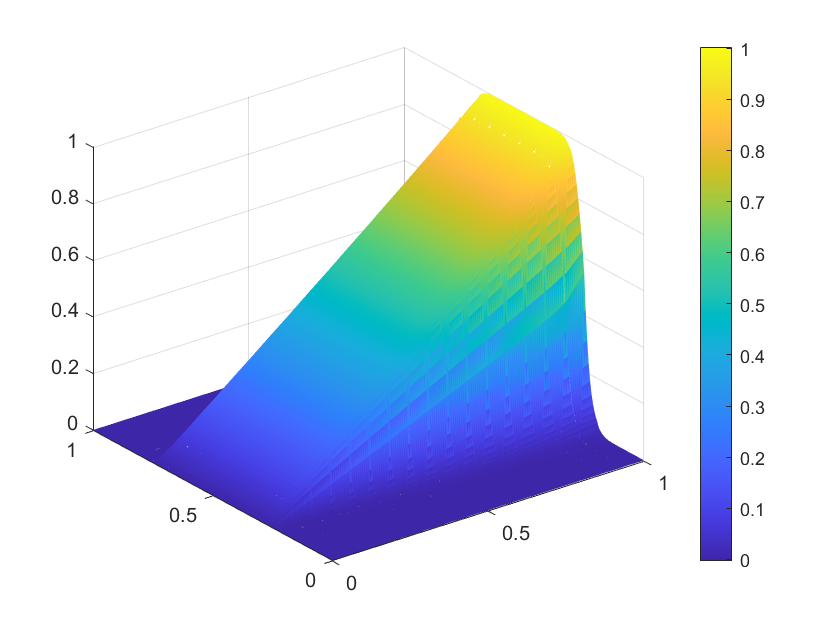}
  }
  \subfloat[$\bm{u}_h^{(4)}(0.5,y)$]
  {
      \includegraphics[width=0.45\textwidth]{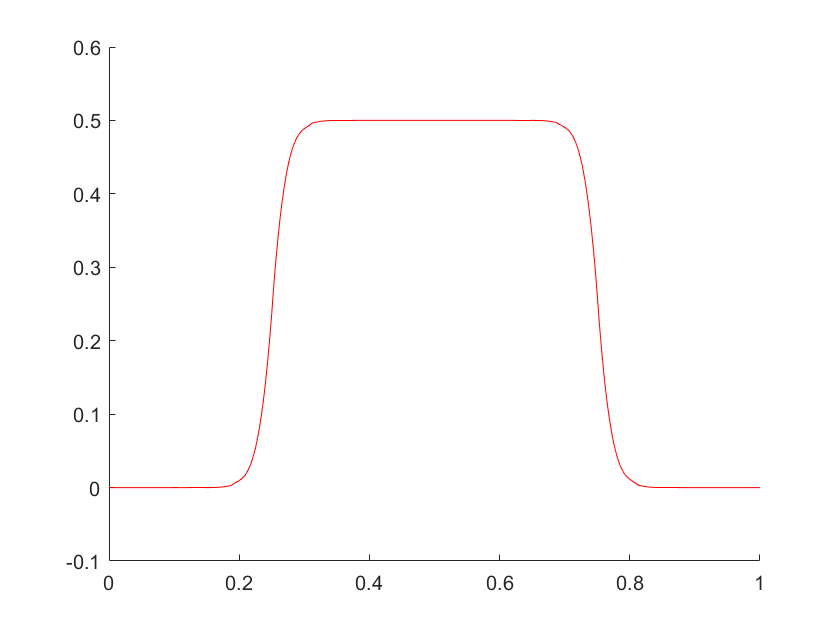}
  }

  \caption{{Example 7. Numerical solutions of the internal layer problem, obtained with higher-order finite elements.}}\label{SUPG_intlayer_figure_highorder}
\end{figure}

{
\subsection{Example 8: Manual refined meshes for boundary layers}
In this numerical experiment, we replicate the problem setting of Example 6 to evaluate the performance of the proposed SUPG method on manual refined meshes. In this case, we set the SUPG parameters as $\delta_T = 0.4\sqrt{2 \times \text{area of } T}$. Note that for unrefined elements, the value of the SUPG parameter coincides with $\delta_T = 0.4/N$ used in the previous example. We consider two types of manual refinement: the first type preserves shape-regularity during refinement (see Figure \ref{type-1 manual refined meshes}), while the second type loses shape-regularity (see Figure \ref{type-2 manual refined meshes}).

\begin{figure}[!htbp]    
  \centering            
  \subfloat[$\mathcal{T}^{\text{manual}}_{1,\text{bd}} , h_{\min} = 1/64$]   
  {
      \includegraphics[width=0.5\textwidth]{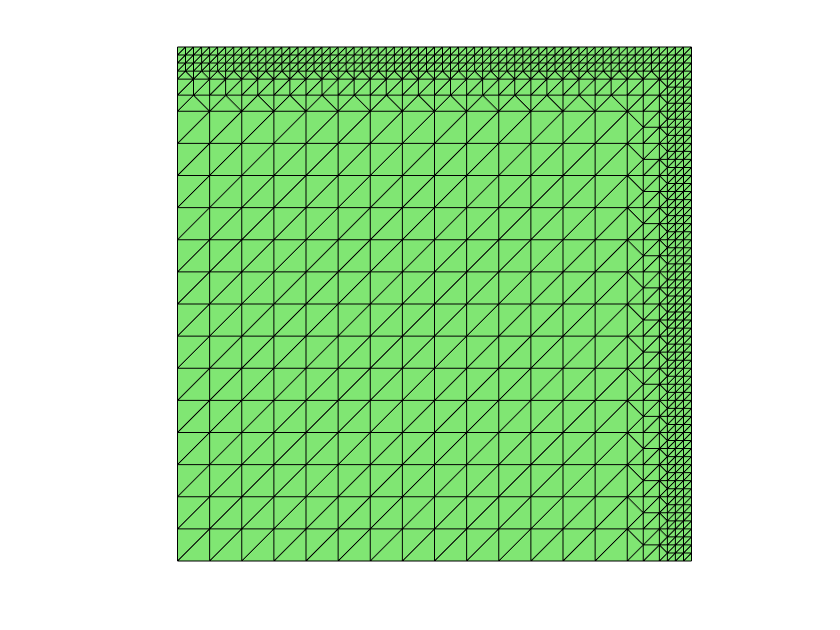}
  }
  \centering      
  \subfloat[$\mathcal{T}^{\text{manual}}_{2,\text{bd}} , h_{\min} = 1/256$\label{first type manual mesh level2}]
  {
      \includegraphics[width=0.5\textwidth]{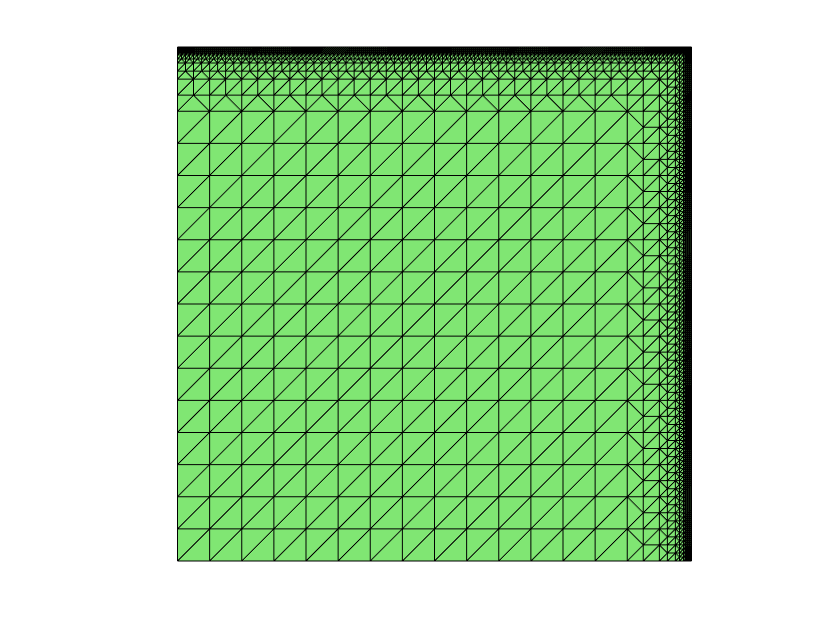}
  }
  \caption{Example 8: The first type of manual refined mesh. $h_{\min}$ denotes the length of the shortest leg among all right triangles constituting the mesh.}\label{type-1 manual refined meshes}
\end{figure}

\begin{figure}[!htbp]    
  \centering            
  \subfloat[$\mathcal{T}^{\text{manual}}_{3,\text{bd}} , h_{\min} = 1/64$]   
  {
      \includegraphics[width=0.5\textwidth]{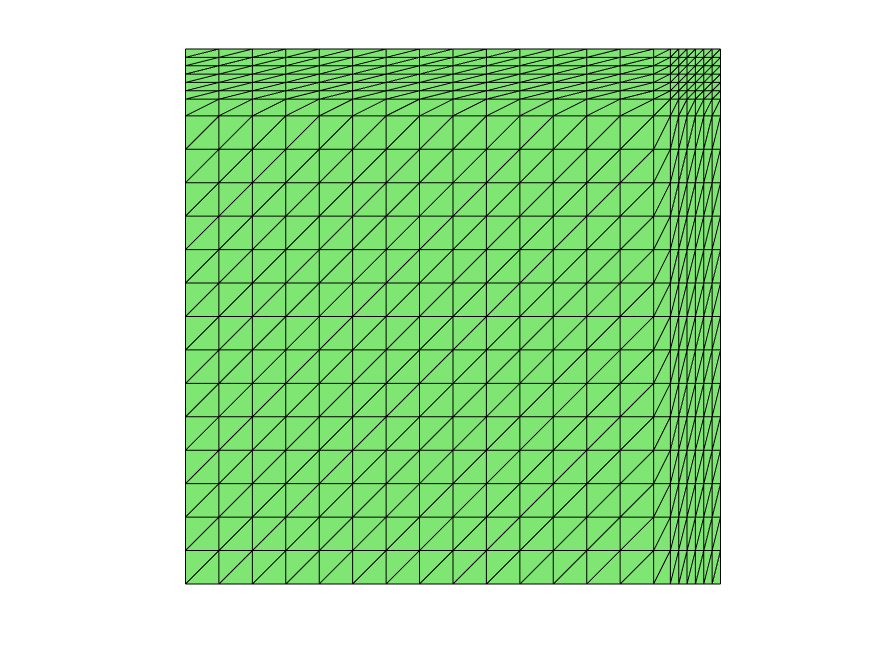}
  }
  \centering      
  \subfloat[$\mathcal{T}^{\text{manual}}_{4,\text{bd}} , h_{\min} = 1/256$\label{second type manual mesh level2}]
  {
      \includegraphics[width=0.5\textwidth]{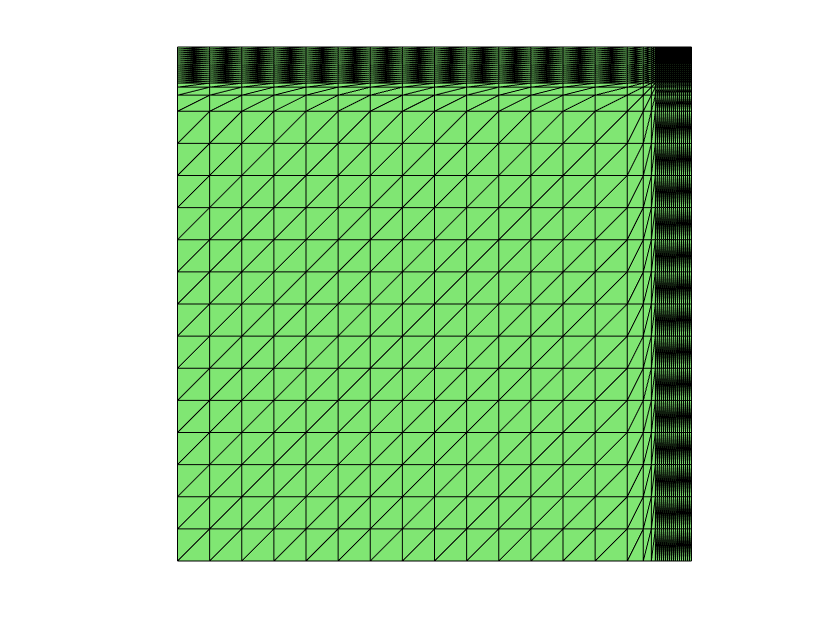}
  }
  \caption{Example 8: The second type of manual refined mesh.}\label{type-2 manual refined meshes}
\end{figure}

For the four manually refined meshes shown in Figures \ref{type-1 manual refined meshes} and \ref{type-2 manual refined meshes}, the first component of the numerical solution obtained by the proposed SUPG method is displayed in Figure \ref{SUPG solution first component in manual refined meshes}. For a more intuitive visualization, we plot the profile of the first component along the cross-section $x = x_0$ for some $x_0\in(0,1)$ in Figure \ref{numerical solutions bd manual x=x0}.

\begin{figure}[!htbp]    
  \centering            
  \subfloat[$\mathcal T_{1,\text{bd}}^{\text{manual}}$]   
  {
      \includegraphics[width=0.45\textwidth]{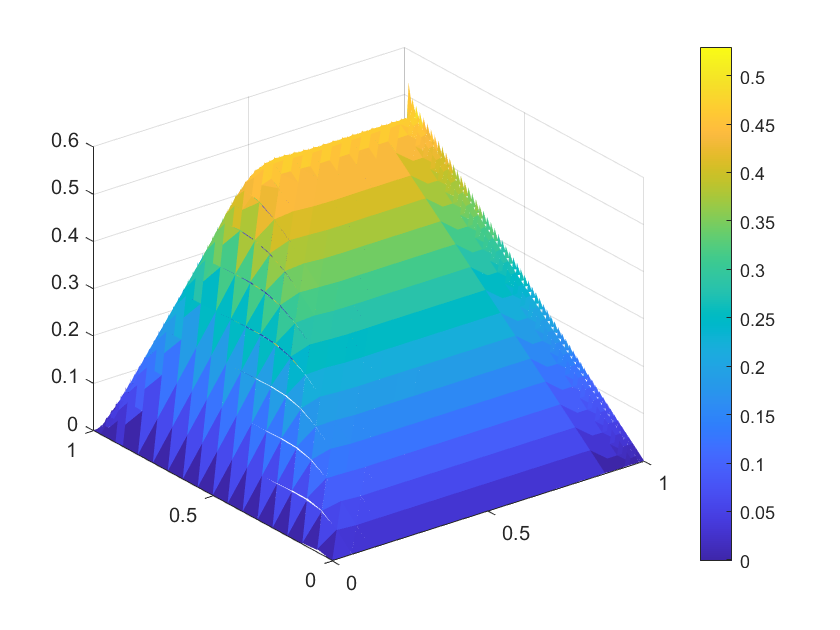}
  }
  \centering      
  \subfloat[$\mathcal T_{2,\text{bd}}^{\text{manual}}$]
  {
      \includegraphics[width=0.45\textwidth]{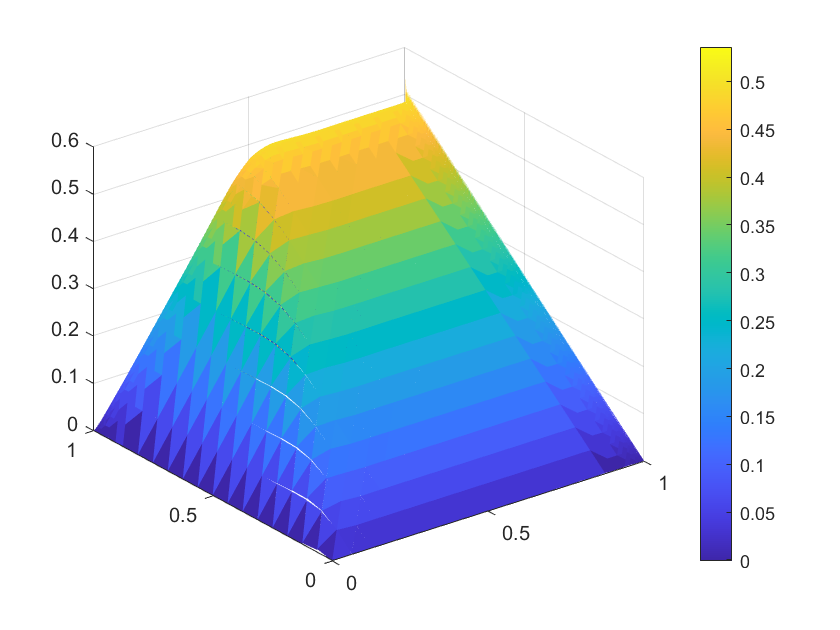}
  }
  \\
  \centering            
  \subfloat[$\mathcal T_{3,\text{bd}}^{\text{manual}}$]   
  {
      \includegraphics[width=0.45\textwidth]{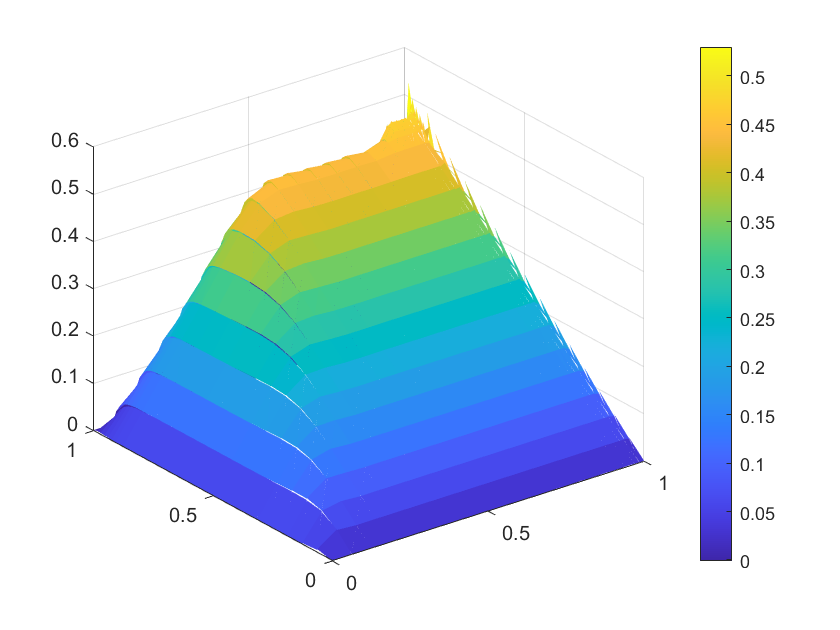}
  }
  \centering      
  \subfloat[$\mathcal T_{4,\text{bd}}^{\text{manual}}$]
  {
      \includegraphics[width=0.45\textwidth]{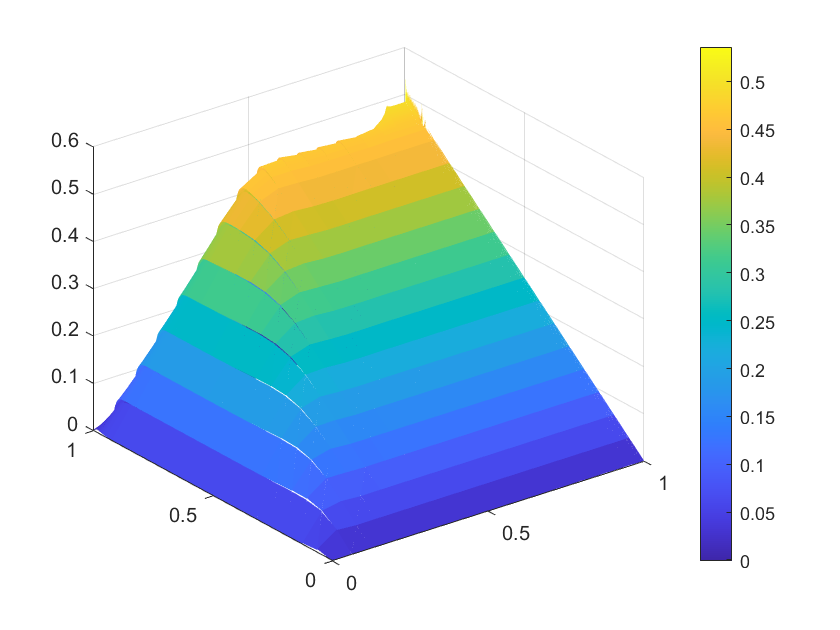}
  }
  \caption{Example 8: Numerical solution $u_h$ obtained by the SUPG method \eqref{magnetic_SUPG_scheme} ($k=1, \delta_T = 0.4\sqrt{2 \times \text{area of } T}$) on manual refined meshes}\label{SUPG solution first component in manual refined meshes}
\end{figure}

\begin{figure}[!htbp]    
  \centering            
  \subfloat[$x_0=0.75$, the first type of mesh]   
  {
      \includegraphics[width=0.45\textwidth]{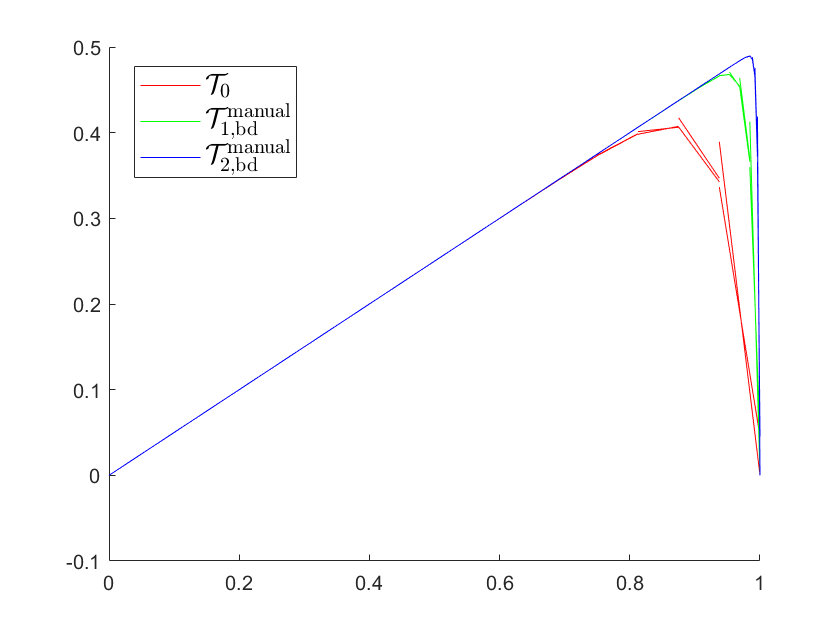}
  }
  \\
  \centering            
  \subfloat[$x_0=0.75$, the second type of mesh]   
  {
      \includegraphics[width=0.45\textwidth]{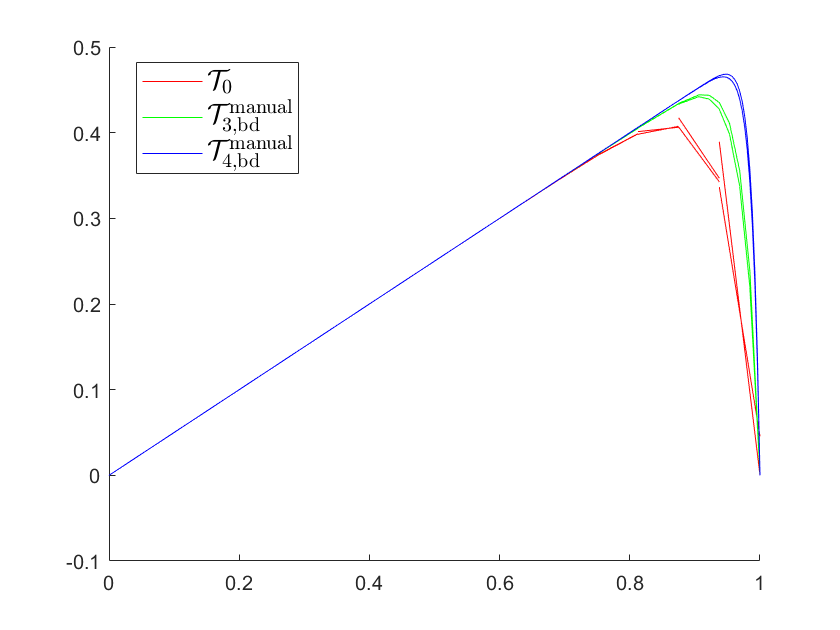}
  }
  \centering      
  \subfloat[$x_0=15/16$, the second type of mesh]
  {
      \includegraphics[width=0.45\textwidth]{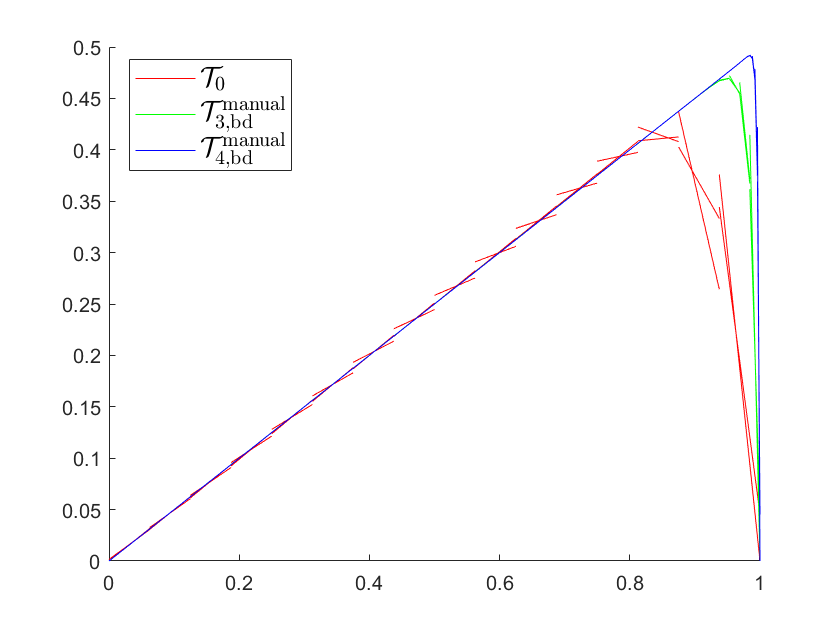}
  }
  \caption{Example 8: The first component of numerical solutions along the cross-section $x = x_0$. $\mathcal T_0$ is the uniform mesh given in Figure \ref{uniform 2D mesh fig} with $N=16$. }\label{numerical solutions bd manual x=x0}
\end{figure}

As observed in Figure \ref{numerical solutions bd manual x=x0}, both types of manually refined meshes yield improved performance compared to the initial uniform mesh. In particular, the second type of mesh achieves superior accuracy near $(1,1)$ compared to other regions of the boundary layer, due to its highest refinement intensity around this point.

Compared to the first type, the second type of mesh offers the advantage of significantly reducing the number of elements, thereby enhancing computational efficiency. However, it presents two drawbacks: First, the loss of shape-regularity may introduce additional analytical difficulties; second, it cannot be obtained by direct refinement of the initial mesh. A more in-depth investigation of these aspects is deferred to future work.

\subsection{Example 9: Manual refined meshes for internal layers}
In this numerical experiment, we replicate the problem settings of Example 7, modifying only the value of $\varepsilon$ to $10^{-5}$, and conduct numerical tests on manual refined meshes. Similar to Example 8, we consider two types of manual refinement: preserving shape-regularity (Figure~\ref{type-1 manual refined meshes internal layers}) or losing shape-regularity (Figure~\ref{type-2 manual refined meshes internal layers}).

\begin{figure}[!htbp]    
  \centering            
  \subfloat[$\mathcal{T}^{\text{manual}}_{1, \text{int}} , h_{\min} = 1/64$]   
  {
      \includegraphics[width=0.5\textwidth]{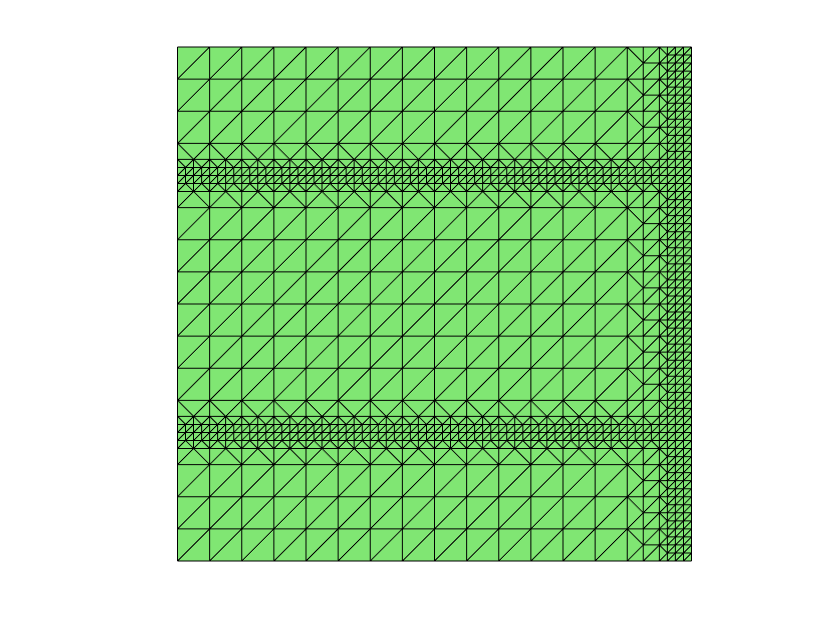}
  }
  \centering      
  \subfloat[$\mathcal{T}^{\text{manual}}_{2, \text{int}} , h_{\min} = 1/256$]
  {
      \includegraphics[width=0.5\textwidth]{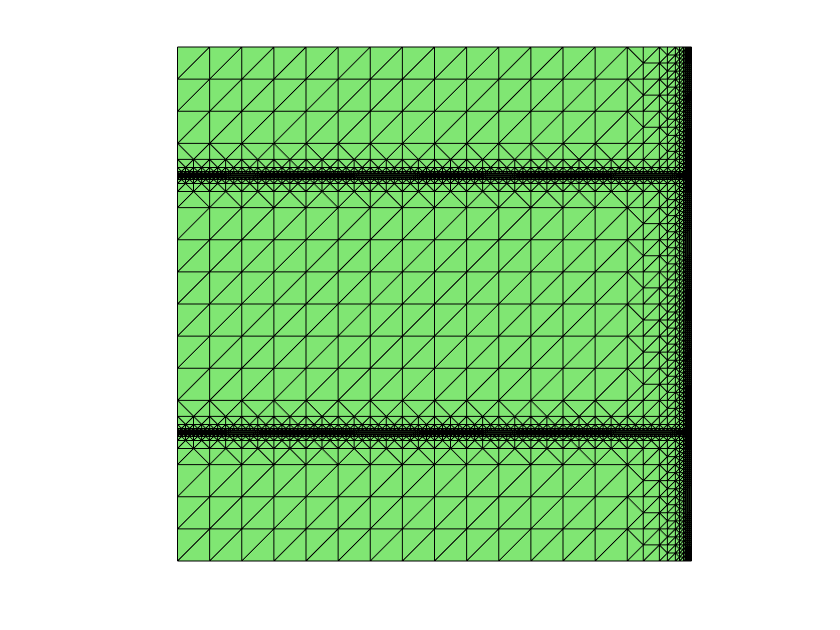}
  }
  \caption{Example 9: The first type of manual refined mesh.}\label{type-1 manual refined meshes internal layers}
\end{figure}

\begin{figure}[!htbp]    
  \centering            
  \subfloat[$\mathcal{T}^{\text{manual}}_{3, \text{int}} , h_{\min} = 1/64$]   
  {
      \includegraphics[width=0.5\textwidth]{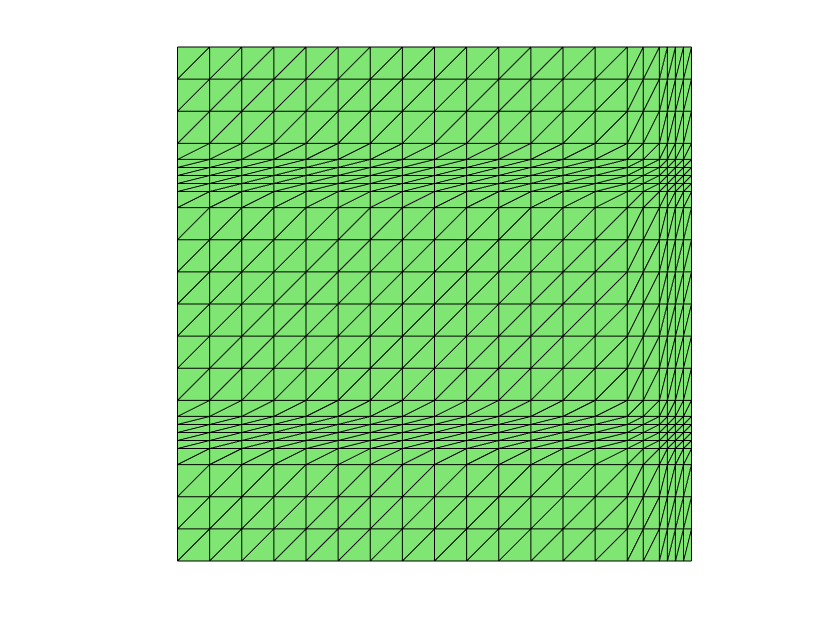}
  }
  \centering      
  \subfloat[$\mathcal{T}^{\text{manual}}_{4, \text{int}} , h_{\min} = 1/256$]
  {
      \includegraphics[width=0.5\textwidth]{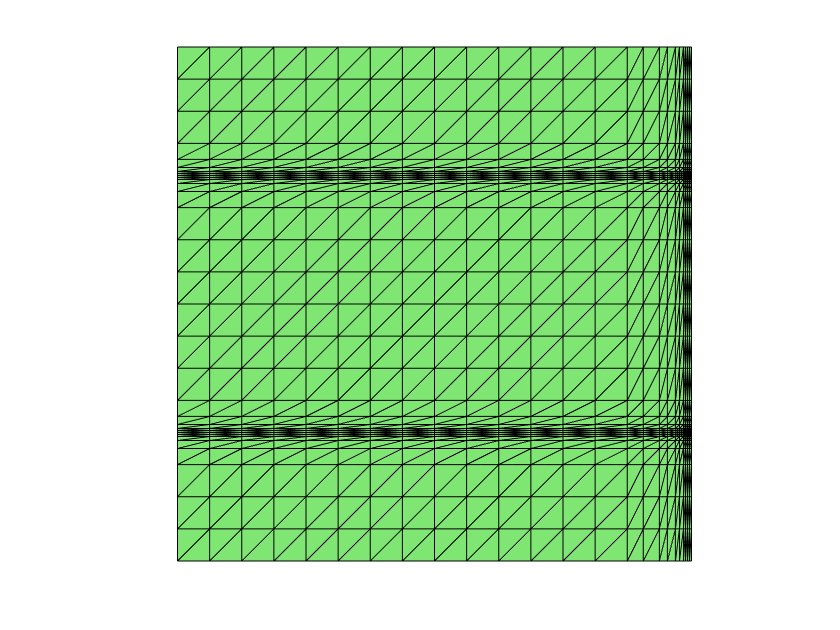}
  }
  \caption{Example 9: The second type of manual refined mesh.}\label{type-2 manual refined meshes internal layers}
\end{figure}

For the four manually refined meshes shown in Figures \ref{type-1 manual refined meshes internal layers} and \ref{type-2 manual refined meshes internal layers}, the first component of the numerical solution obtained by the proposed SUPG method is displayed in Figure \ref{SUPG solution first component in manual refined meshes int}. For a more intuitive visualization, we plot the profile of the first component along the cross-section $x = 0.5$ in Figure \ref{numerical solutions int manual x=0.5}.

\begin{figure}[!htbp]    
  \centering            
  \subfloat[$\mathcal T_{1, \text{int}}^{\text{manual}}$]   
  {
      \includegraphics[width=0.45\textwidth]{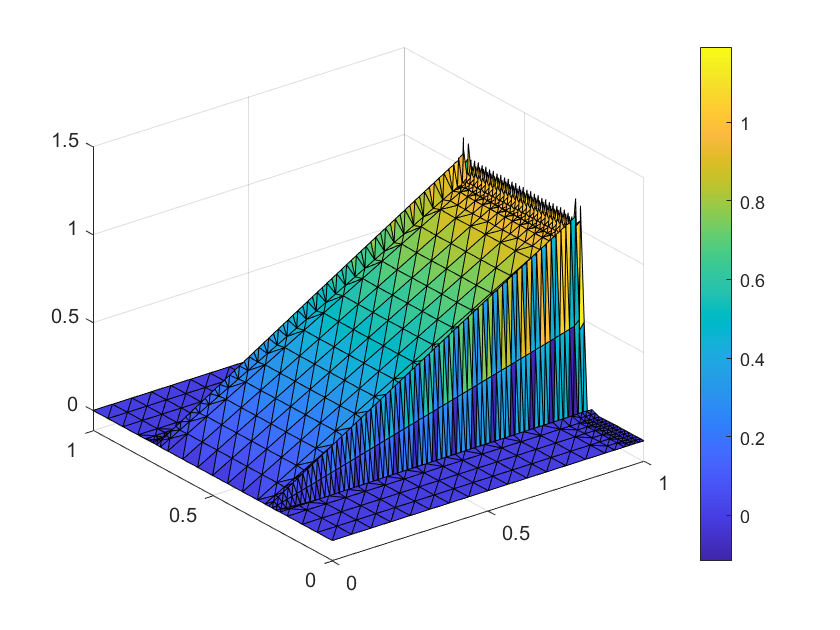}
  }
  \centering      
  \subfloat[$\mathcal T_{2, \text{int}}^{\text{manual}}$]
  {
      \includegraphics[width=0.45\textwidth]{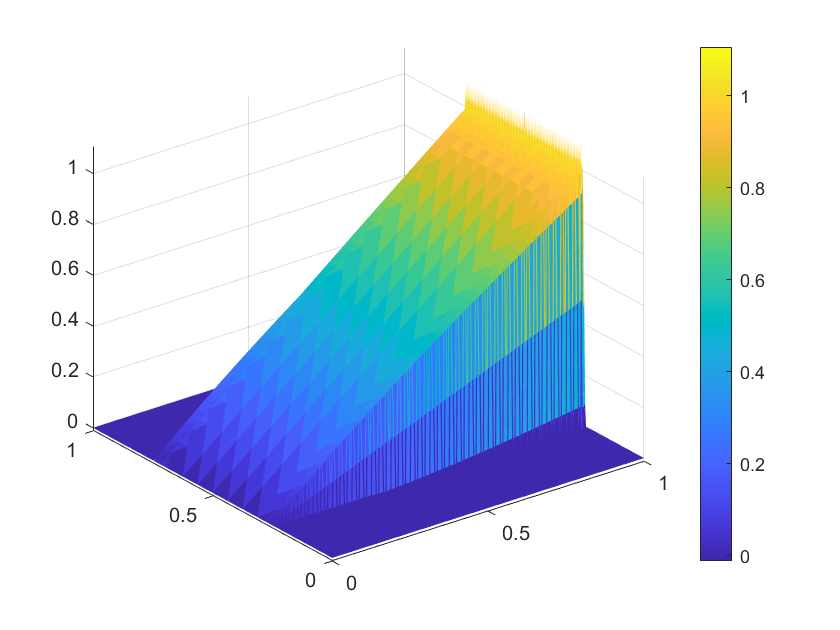}
  }
  \\
  \centering            
  \subfloat[$\mathcal T_{3, \text{int}}^{\text{manual}}$]   
  {
      \includegraphics[width=0.45\textwidth]{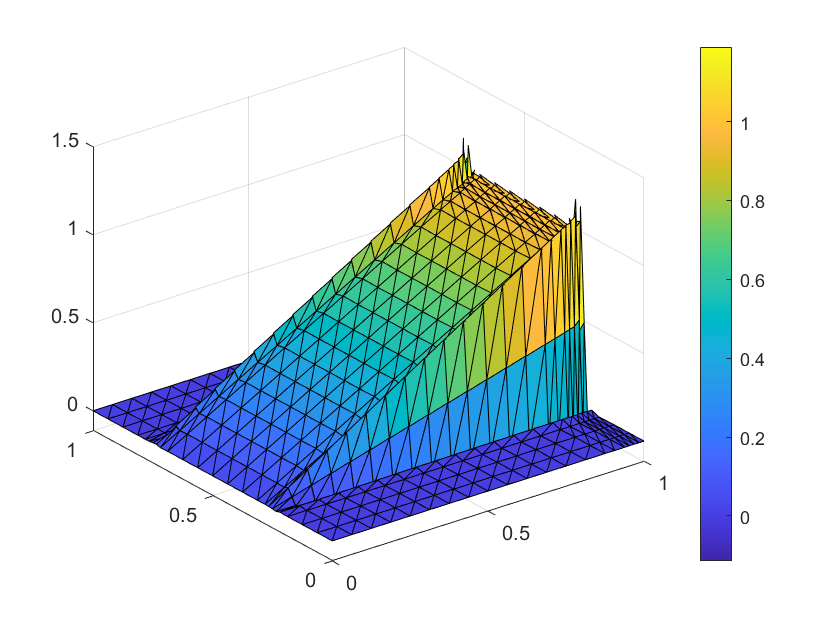}
  }
  \centering      
  \subfloat[$\mathcal T_{4, \text{int}}^{\text{manual}}$]
  {
      \includegraphics[width=0.45\textwidth]{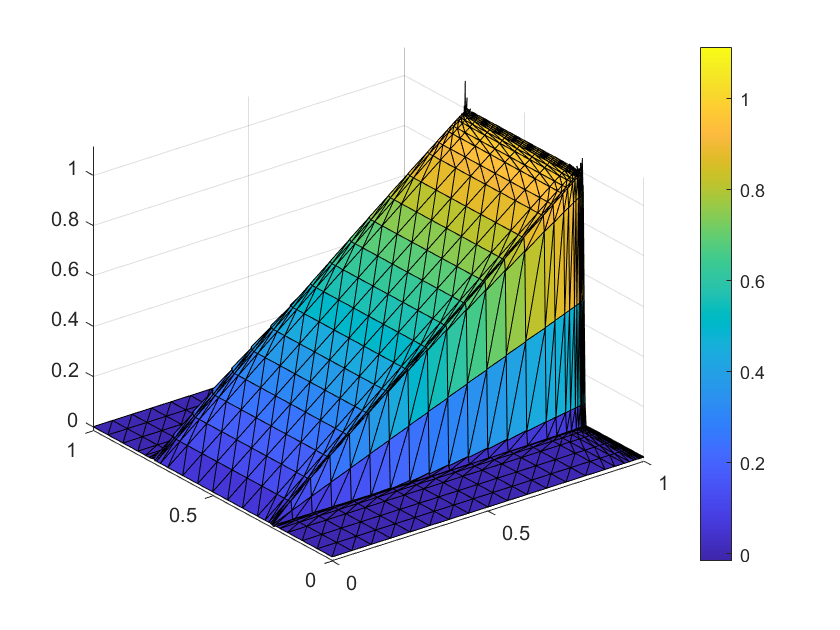}
  }
  \caption{Example 9: Numerical solution $u_h$ obtained by the SUPG method \eqref{magnetic_SUPG_scheme} ($k=1, \delta_T=0.4h_T$), in manual refined meshes}\label{SUPG solution first component in manual refined meshes int}
\end{figure}

\begin{figure}[!htbp]    
  \centering            
  \subfloat[the first type of mesh]   
  {
      \includegraphics[width=0.45\textwidth]{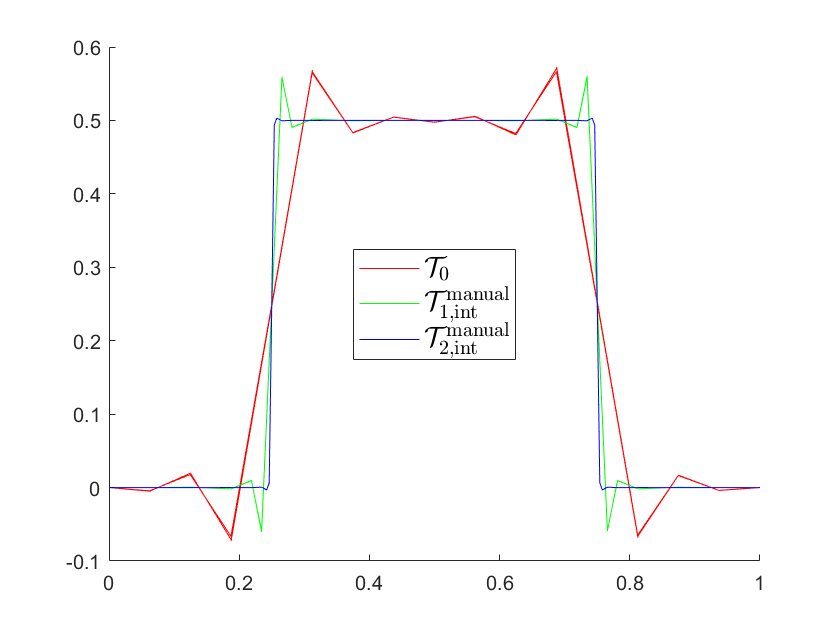}
  }
  \centering            
  \subfloat[the second type of mesh]   
  {
      \includegraphics[width=0.45\textwidth]{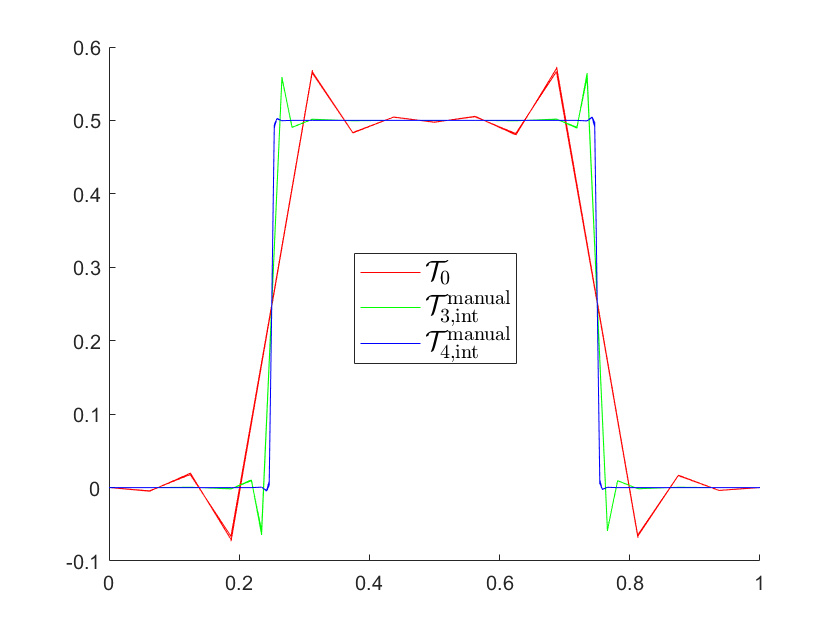}
  }
  \caption{Example 9: The first component of numerical solutions along the cross-section $x = 0.5$. $\mathcal T_0$ is the uniform mesh given in Figure \ref{uniform 2D mesh fig} with $N=16$. }\label{numerical solutions int manual x=0.5}
\end{figure}

As observed in Figure \ref{numerical solutions int manual x=0.5}, local mesh refinement effectively reduces the extent of numerical oscillations in regions where layers occur. However, such refinement generally does not diminish the magnitude of the oscillations unless the mesh is sufficiently fine to resolve the structure of the thin layers.

\subsection{Example 10: Instability of standard Galerkin method}
This numerical experiment further investigates the instability of the
standard Galerkin method by repeating the problem settings of
Example~6. As shown in Figure~\ref{standard_Galerkin_method}, the numerical solution obtained by the standard Galerkin method exhibits significant inaccuracies on coarse uniform meshes. We further examine its performance on manually refined meshes and with higher-order finite element spaces, as illustrated in Figures~\ref{instability of standard Galerkin high order} and~\ref{instability of standard Galerkin refined mesh}.

\begin{figure}[!htbp]    
  \centering            

      \includegraphics[width=0.45\textwidth]{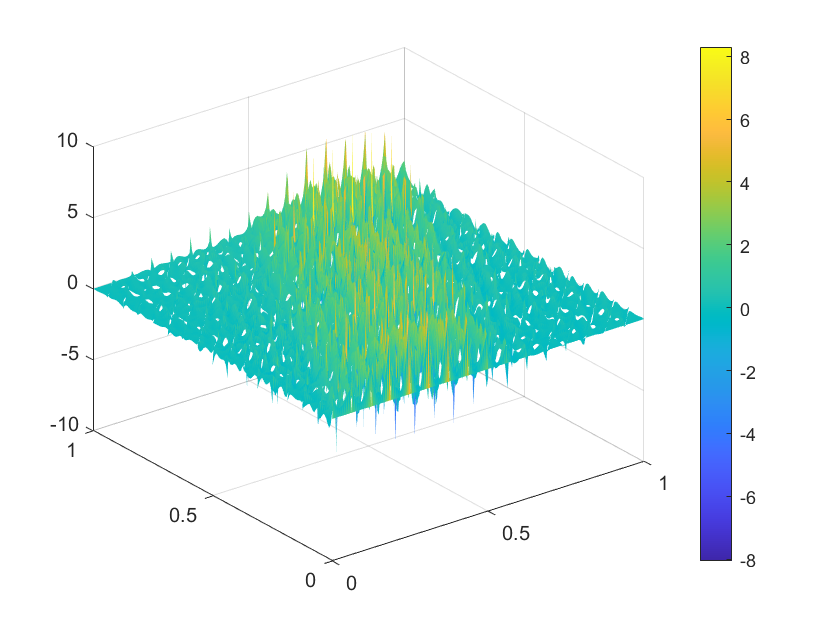} 
  \caption{Example 10: Numerical solution obtained by standard Galerkin method \eqref{standard_Galerkin_method} ($N=16, k = 4$).}\label{instability of standard Galerkin high order}
\end{figure}

\begin{figure}[!htbp]    
  \centering            
  \subfloat[Mesh in Figure~\ref{first type manual mesh level2}]
  {
      \includegraphics[width=0.45\textwidth]{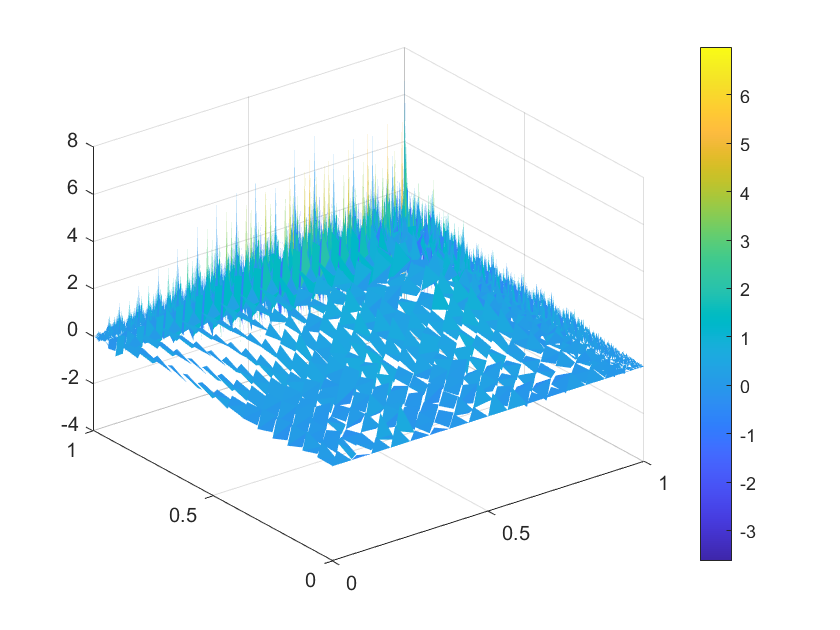} 
  }
  \subfloat[Mesh in Figure~\ref{second type manual mesh level2}]
  {
      \includegraphics[width=0.45\textwidth]{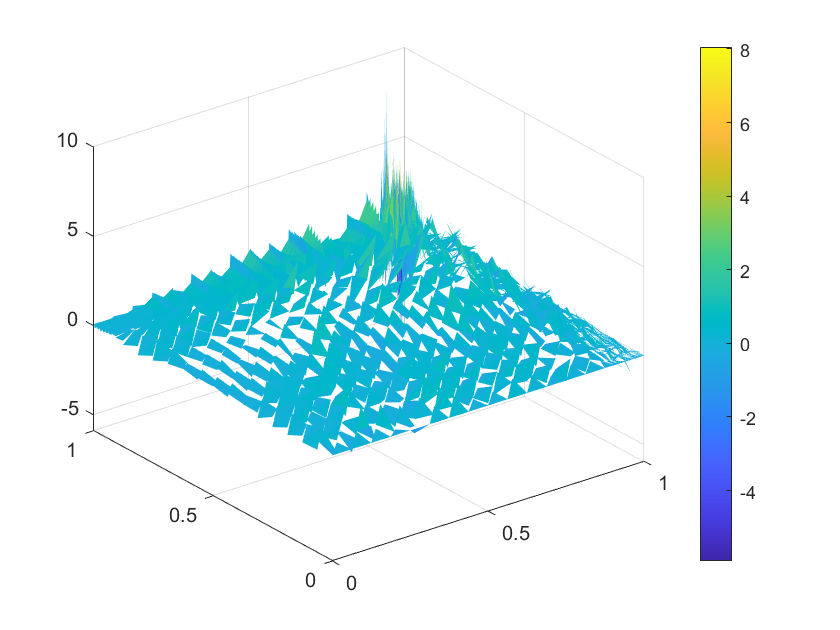} 
  }
  \caption{Example 10: Numerical solutions obtained by standard Galerkin method \eqref{standard_Galerkin_method} on manual refined meshes ($k = 1$).}\label{instability of standard Galerkin refined mesh}
\end{figure}

The results in Figures~\ref{instability of standard Galerkin high order} and~\ref{instability of standard Galerkin refined mesh} demonstrate that the standard Galerkin method continues to perform poorly. This indicates that for magnetic advection-diffusion equations in advection-dominated regimes, particularly when layers occur, neither local mesh refinement nor increasing the polynomial order significantly improves the stability of the numerical scheme, which is consistent with prior observations for scalar advection–diffusion equations \cite{sun2010numerical}. This underscores the necessity of designing stabilized formulations, such as the SUPG scheme proposed in this work.

}

\section{Conclusion}\label{sec:conclusion}
{
In this paper, we have developed and analyzed the streamline upwind/Petrov-Galerkin (SUPG) method for the magnetic advection-diffusion problem. The resulting schemes possess stability and conformity properties that make them well-suited for potential applications to advection-dominated discretizations arising in magnetohydrodynamics (MHD).

The key ingredient of the SUPG method involves introducing stabilization terms that incorporate residuals and weighted advection terms, thereby enhancing stability and accuracy. Unlike scalar cases, a critical distinction in the $\bm{H}(\mathrm{curl})$-conforming finite element discretization for magnetic advection-diffusion problems lies in establishing facet integral terms $\langle\bm{\beta}\cdot\bm{n},\llbracket \bm{u}_h\rrbracket\cdot\vavg{\bm{v}_h}\rangle_F$ to ensure numerical stability. Through the definition of a lifting operator, these terms are incorporated into the discrete magnetic advection operator, enabling the development of an SUPG formulation. Numerical experiments demonstrate that the proposed method attains the optimal convergence rates in the energy norm for smooth solutions, while effectively mitigating spurious oscillations in the vicinity of sharp layers.
These results demonstrate the efficiency of the proposed method.}


\section*{Data availability}
No data was used for the research described in the article.

\section*{Declaration of competing interest}
The authors declare that they have no known competing financial interests or personal relationships that could have appeared to
influence the work reported in this paper.

\section*{Acknowledgments}
This work was supported in part by the National Natural Science Foundation of China (Grants No.~12222101 and No.~12571383) and the Beijing Natural Science Foundation (Grant No.~1232007).

{
\appendix
\section{Supplementary Details for Numerical Experiments}\label{secA1}
To ensure the reproducibility of our numerical tests, the source term $\bm{f}$ and boundary data $\bm{g}$ used in Subsections \ref{Numerical_experiments_example1}, \ref{Numerical_experiments_example2}, \ref{Numerical_experiments_example4}, and \ref{Numerical_experiments_example5} are documented in this appendix. Images of the exact solutions are also provided.

\subsection{ Example 1 }
In this example, the inflow and outflow boundaries are identified as:
$$
\begin{aligned}
    \Gamma^- & = \{(x,y,z)\in\partial\Omega: x = 0~\text{or}~y=0~\text{or}~z=0\},  \\
    \Gamma^+ & = \{(x,y,z)\in\partial\Omega: x = 1~\text{or}~y=1~\text{or}~z=1\}.
\end{aligned}
$$
The source term $\bm{f}$ and boundary data $\bm{g}$ are specified as follows:
$$
\begin{aligned}
\bm f =&~ \varepsilon\begin{bmatrix}
y \cos(x y z) - 2x - x^{2} y e^{x z} - x y^{2} z \sin(x y z) \\
    2y + x \cos(x y z) + z e^{x z} - x^{2} y z \sin(x y z) \\
    y e^{x z} + x^{2} z^{2} \sin(x y z) + y^{2} z^{2} \sin(x y z) + x y z e^{x z}
\end{bmatrix} \\
&+\begin{bmatrix}
    8y e^{xz} - x^{2}y + (x + 2)e^{xz} - yz\left(\frac{z}{2} - 1\right) e^{xz} - xy(y - 3) e^{xz} \\
    2xy\left(\frac{z}{2} - 1\right) - x^{2}(x + 2) - 8x^{2}y - \sin(xyz) \\
    8\sin(xyz) - \frac{y e^{xz}}{2} + x(x + 2)z \cos(xyz) - xy (y - 3)\cos(xyz) - yz\left(\frac{z}{2} - 1\right) \cos(xyz)
\end{bmatrix}.
\end{aligned}
$$
$$
\begin{aligned}
\bm g =&\left\{
\begin{aligned}
    & [y, 0, 0]^T &\{(x,y,z)\in\partial\Omega : x = 0\}\\
    & [0, -\sin(yz), -y]^T &\{(x,y,z)\in\partial\Omega : x = 1\}\\
    & [0, 0, 0]^T &\{(x,y,z)\in\partial\Omega : y = 0\}\\
    & [\sin(xz), 0, -e^{xz}]^T &\{(x,y,z)\in\partial\Omega : y = 1\}\\
    & [y, -x^2y, 0]^T &\{(x,y,z)\in\partial\Omega : z = 0\}\\
    & [x^2y, ye^x, 0]^T &\{(x,y,z)\in\partial\Omega : z = 1\}\\
\end{aligned}
\right..
\end{aligned}
$$
A plot of the exact solution is provided in Figure \ref{Example1_exactsol}.
\begin{figure}[!htbp]
\centering
\includegraphics[width=0.5\textwidth]{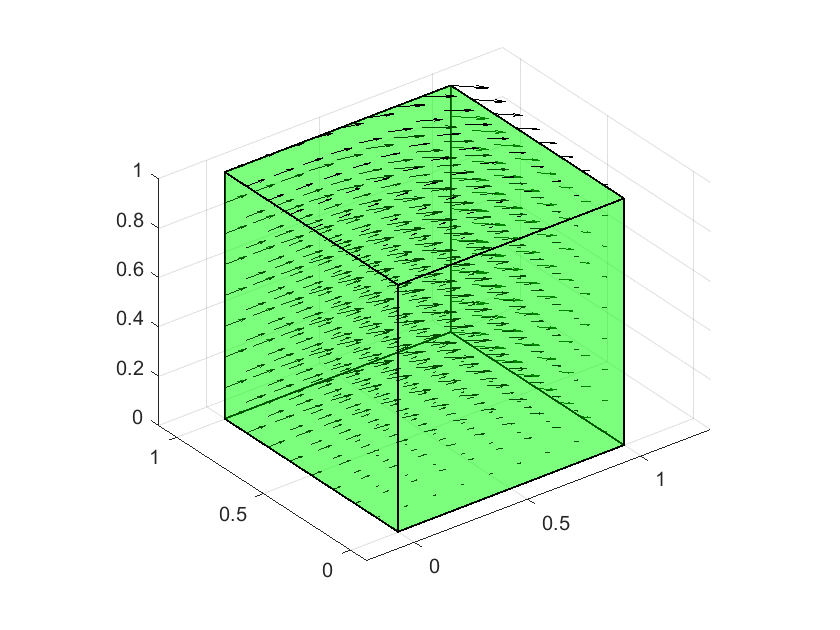} 
\caption{A plot of $\bm u = [ye^{xz}, -x^2y ,\sin(xyz) ]^T$} \label{Example1_exactsol}
\end{figure}

\subsection{ Example 2 }
In this example, the inflow and outflow boundaries are identified as:
$$
\begin{aligned}
    \Gamma^-=&\{(x,y)\in\partial\Omega: x=0, \frac{1}{2}<y<1\}\cup\{(x,y)\in\partial\Omega: x=1, 0<y<\frac{1}{2}\}\\
    &\cup\{(x,y)\in\partial\Omega: y=0, 0<x<\frac{1}{2}\}\cup\{(x,y)\in\partial\Omega: y=1, \frac{1}{2}<x<1\},  \\
    \Gamma^+=&\{(x,y)\in\partial\Omega: x=0, 0\le y\le\frac{1}{2} \}\cup\{(x,y)\in\partial\Omega: x=1, \frac{1}{2}\le y\le1 \}\\
    &\cup\{(x,y)\in\partial\Omega: y=0, \frac{1}{2}\le x\le1 \}\cup\{(x,y)\in\partial\Omega: y=1, 0\le x\le\frac{1}{2} \}.
\end{aligned}
$$
The source term $\bm{f}$ and boundary data $\bm{g}$ are specified as follows:
$$
\begin{aligned}
\bm f =&~ \varepsilon\begin{bmatrix}
\pi^{2} e^{x} \cos(\pi x) \cos(\pi y) - 32x(x - 1) + \pi e^{x} \cos(\pi y) \sin(\pi x) \\
    64xy - 32y - 32x + (\pi^{2}-1) e^{x} \sin(\pi x) \sin(\pi y) - 2\pi e^{x} \cos(\pi x) \sin(\pi y) + 16
\end{bmatrix} \\
&+\begin{bmatrix}
    - 16x\left(x - \frac{1}{2}\right)(x - 1)(2y - 1) + 16(2x - 1)y\left(y - \frac{1}{2}\right)(y - 1) \\
    e^{x}\left(y - \frac{1}{2}\right) \sin(\pi y)\left(\sin(\pi x) + \pi \cos(\pi x)\right) - \pi e^{x}\left(x - \frac{1}{2}\right) \cos(\pi y) \sin(\pi x)
\end{bmatrix}\\
&+\begin{bmatrix}
    16x(x - 1)y(y - 1) - e^{x} \sin(\pi x) \sin(\pi y)  \\
    e^{x} \sin(\pi x) \sin(\pi y) + 16x(x - 1)y(y - 1) 
\end{bmatrix}\\
\bm g \equiv&~0
.
\end{aligned}
$$
A plot of the exact solution is provided in Figure \ref{Example2_exactsol}.
\begin{figure}[!htbp]
\centering
\includegraphics[width=0.5\textwidth]{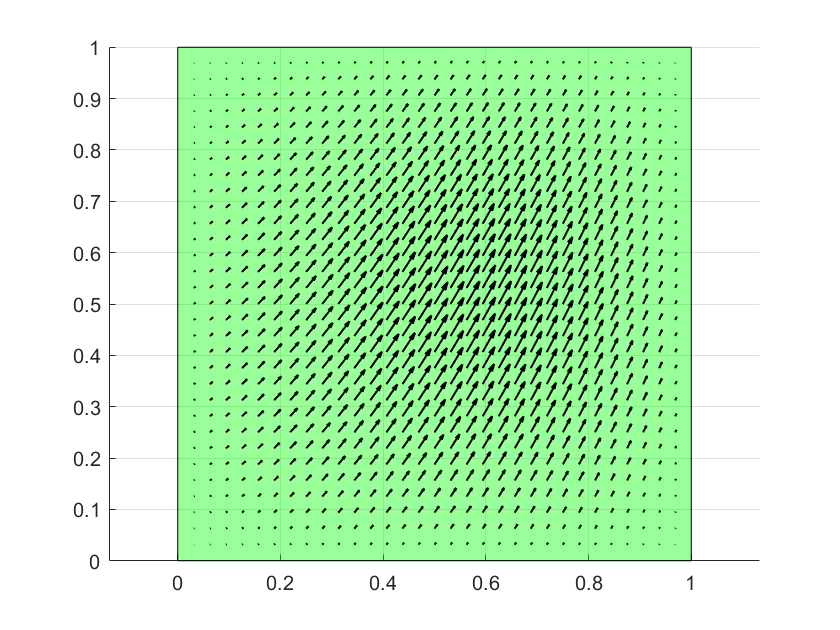} 
\caption{A plot of $\bm u = [16x(1-x)y(1-y),e^x\sin(\pi x)\sin(\pi y)]^T$}\label{Example2_exactsol}
\end{figure}

\subsection{ Example 4 }
In this example, the inflow and outflow boundaries are identified as:
$$
\begin{aligned}
    \Gamma^-=&\{(x,y)\in\partial\Omega: x=0, \frac{1}{2}<y<1\}\cup\{(x,y)\in\partial\Omega: x=1, 0<y<\frac{1}{2}\}\\
    &\cup\{(x,y)\in\partial\Omega: y=0, 0<x<\frac{1}{2}\}\cup\{(x,y)\in\partial\Omega: y=1, \frac{1}{2}<x<1\},  \\
    \Gamma^+=&\{(x,y)\in\partial\Omega: x=0, 0\le y\le\frac{1}{2} \}\cup\{(x,y)\in\partial\Omega: x=1, \frac{1}{2}\le y\le1 \}\\
    &\cup\{(x,y)\in\partial\Omega: y=0, \frac{1}{2}\le x\le1 \}\cup\{(x,y)\in\partial\Omega: y=1, 0\le x\le\frac{1}{2} \}.
\end{aligned}
$$
The source term $\bm{f}$ and boundary data $\bm{g}$ are specified as follows:
$$
\begin{aligned}
\bm f = &~\varepsilon
\begin{bmatrix}
    \pi^2 e^{x} \cos(\pi x) \cos(\pi y) + \pi e^{x} \cos(\pi y) \sin(\pi x) \\
    (\pi^2 - 1) e^{x} \sin(\pi x) \sin(\pi y) - 2\pi e^{x} \cos(\pi x) \sin(\pi y)
\end{bmatrix}
\\
&+ \varepsilon
\begin{bmatrix}
    32x(x-1)(y+1)(y-2) \sin(x+y) - 64x(x-1)(2y-1) \cos(x+y) \\
    -32(x^2 y^2 - x^2 y - x y^2 - 3x y + 2x + 2y - 1) \sin(x+y)
\end{bmatrix}
\\
&+ \varepsilon
\begin{bmatrix}
    0 \\
    32(2xy - x - y)(x + y - 1) \cos(x+y)
\end{bmatrix}
\\
&+
\begin{bmatrix}
    -16(4x^3 y - 2x^3 - 2x^2 y^2 - 4x^2 y + 3x^2) \sin(x+y) \\
    e^{x} \sin(\pi y) \left( (y+\frac{1}{2}) \sin(\pi x) + \pi (y-\frac{1}{2}) \cos(\pi x) \right)
\end{bmatrix}
\\
&+
\begin{bmatrix}
    16(4x y^3 - 8x y^2 + 2x y + x - 2y^3 + 3y^2 - y) \sin(x+y) \\
    -\pi e^{x} (x-\frac{1}{2}) \cos(\pi y) \sin(\pi x) + 32x(x-1)y(y-1) \sin(x+y)
\end{bmatrix}
\\
&+
\begin{bmatrix}
    32xy(y-x)(x-1)(y-1) \cos(x+y) - e^{x} \sin(\pi x) \sin(\pi y) \\
    0
\end{bmatrix}\\
\bm g \equiv&~0.
\end{aligned}
$$
A plot of the exact solution is provided in Figure \ref{Example4_exactsol}.
\begin{figure}[!htbp]
\centering
\includegraphics[width=0.5\textwidth]{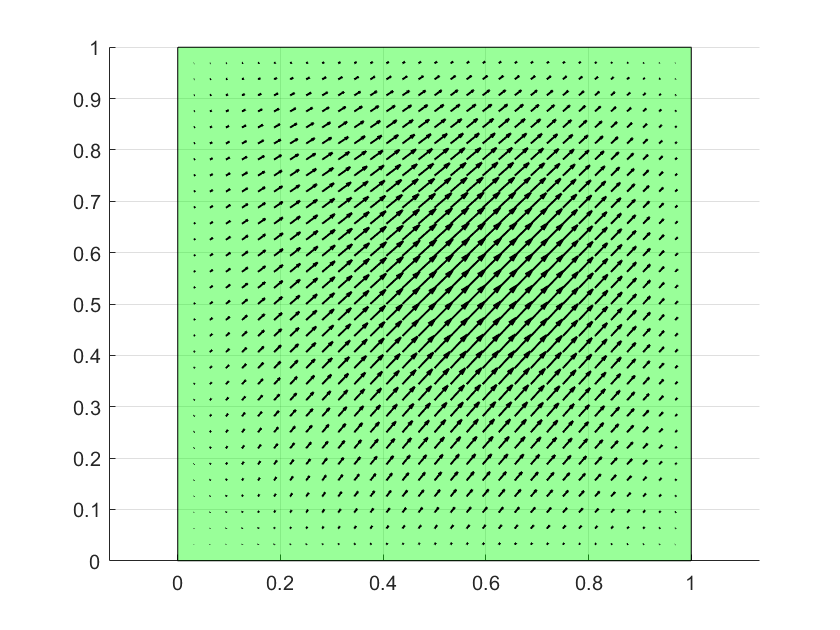} 
\caption{A plot of $\bm u = [32x(1-x)y(1-y)\sin(x+y),e^x\sin(\pi x)\sin(\pi y)]^T$} \label{Example4_exactsol}
\end{figure}

\subsection{ Example 5 }
In this example, the source term $\bm{f}$ is specified as follows:
$$
\begin{aligned}
\bm{f}  =&~\varepsilon
\begin{bmatrix}
    \pi^2 e^{x} \sin(\pi(x + y)) - 2\pi(x - 1) \sin(\pi(x^2 + y)) \\
    \pi(y + 1) \sin(\pi(x^2 + y)) + e^{x} \sin(\pi(x + y )) - \cos(\pi(x^2 + y))
\end{bmatrix}\\
&+
\varepsilon
\begin{bmatrix}
     -\pi e^{x} \cos(\pi(x + y))  \\
     -\pi^2 e^{x} \sin(\pi(x + y)) + 2\pi e^{x} \cos(\pi(x + y)) 
\end{bmatrix}\\
&+
\varepsilon
\begin{bmatrix}
    - \pi^2(x - 1)(y + 1) \cos(\pi(x^2 + y)) \\
     2\pi x(x - 1)\sin(\pi(x^2 + y)) + 2\pi^2 x(x - 1)(y + 1) \cos(\pi(x^2 + y))
\end{bmatrix}\\
&+
\begin{bmatrix}
    (x^2/2 + 1/2)((x - 1)\cos(\pi(x^2 + y)) - \pi(x - 1)(y + 1) \sin(\pi(x^2 + y))) \\
    -4e^{x} \sin(\pi(x + y)) - y \cos(\pi(x^2 + y))(x - 1)(y + 1)
\end{bmatrix}\\
&+
\begin{bmatrix}
    - (y^2/2 + 2)((y + 1)\cos(\pi(x^2 + y)) - 2\pi x(x - 1)(y + 1)\sin(\pi(x^2 + y))) \\
    -e^{x}(y^2/2 + 2)(\sin(\pi(x + y)) + \pi \cos(\pi(x + y)))
\end{bmatrix}\\
&+
\begin{bmatrix}
     x e^{x} \sin(\pi(x + y)) - 4(x - 1)(y + 1)\cos(\pi(x^2 + y)) \\
     \pi e^{x}(x^2/2 + 1/2)\cos(\pi(x + y)) 
\end{bmatrix}&
\end{aligned}
$$
For the hexagonal computational domain, the vertices are defined by $P_1 = (1,0)$, $P_2 = \left(\frac{1}{2},\frac{\sqrt{3}}{2}\right)$, $P_3 = \left(-\frac{1}{2},\frac{\sqrt{3}}{2}\right)$, $P_4 = (-1,0)$, $P_5 = \left(-\frac{1}{2},-\frac{\sqrt{3}}{2}\right)$, and $P_6 = \left(\frac{1}{2},-\frac{\sqrt{3}}{2}\right)$. The inflow and outflow boundaries are identified as:
$$
\begin{aligned}
    \Gamma^- &= P_2P_3\cup P_3P_4\cup P_4P_5,  \\
    \Gamma^+ &= P_1P_2\cup P_5P_6\cup P_6P_1.
\end{aligned}
$$
The boundary data $\bm g$ is given by:
$$
\begin{aligned}
\bm g =&\left\{
\begin{aligned}
    & \begin{bmatrix} \frac{1}{4}\left(\sqrt{3}e^{x}\sin(\pi(x + y)) - (x - 1)(y + 1)\cos(\pi(x^2 + y))\right) \\ \frac{\sqrt{3}}{4}\left(-\sqrt{3}e^{x}\sin(\pi(x + y)) + (x - 1)(y + 1)\cos(\pi(x^2 + y))\right) \end{bmatrix} &(x,y)\in P_1P_2\\
    & \left[-(x - 1)(y + 1)\cos(\pi(x^2 + y)), -e^{x}\sin(\pi(x + y))\right]^T &(x,y)\in P_2P_3\\
    & \left[-(x - 1)(y + 1)\cos(\pi(x^2 + y)), -e^{x}\sin(\pi(x + y))\right]^T &(x,y)\in P_3P_4\\
    & \left[-(x - 1)(y + 1)\cos(\pi(x^2 + y)), -e^{x}\sin(\pi(x + y))\right]^T &(x,y)\in P_4P_5\\
    & \left[-(x - 1)(y + 1)\cos(\pi(x^2 + y)), 0\right]^T &(x,y)\in P_5P_6\\
    & \begin{bmatrix} -\frac{1}{4}\left(\sqrt{3}e^{x}\sin(\pi(x + y)) + (x - 1)(y + 1)\cos(\pi(x^2 + y))\right) \\ -\frac{\sqrt{3}}{4}\left(\sqrt{3}e^{x}\sin(\pi(x + y)) + (x - 1)(y + 1)\cos(\pi(x^2 + y))\right) \end{bmatrix} &(x,y)\in P_6P_1\\
\end{aligned}
\right..
\end{aligned}
$$
For the hexagonal computational domain, the vertices are defined by $P_1 = (0,0), P_2 = (1,0), P_3 = (1,1), P_4 = (-1,1), P_5 = (-1,-1),$ and $ P_6 = (0,-1)$. The inflow and outflow boundaries are identified as:
$$
\begin{aligned}
    \Gamma^- &= P_3P_4\cup P_4P_5,  \\
    \Gamma^+ &= P_1P_2\cup P_2P_3\cup P_5P_6\cup P_6P_1.
\end{aligned}
$$
The boundary data $\bm g$ is given by:
$$
\begin{aligned}
\bm g =&\left\{
\begin{aligned}
    & \left[-(x - 1)\cos(\pi x^2), 0\right]^T &(x,y)\in P_1P_2\\
    & \left[0, e\sin(\pi y)\right]^T &(x,y)\in P_2P_3\\
    & \left[2(x - 1)\cos(\pi x^2), e^{x}\sin(\pi x)\right]^T &(x,y)\in P_3P_4\\
    & \left[-2(y + 1)\cos(\pi y), e^{-1}\sin(\pi y)\right]^T &(x,y)\in P_4P_5\\
    & \left[0, 0\right]^T &(x,y)\in P_5P_6\\
    & \left[0, -\sin(\pi y )\right]^T &(x,y)\in P_6P_1\\
\end{aligned}
\right..
\end{aligned}
$$
Plots of the exact solutions on the hexagonal and L-shaped meshes are provided in Figure \ref{Example5_exactsol}.
\begin{figure}[!htbp]    
  \centering            
  \subfloat[hexagonal domain]
  {
      \includegraphics[width=0.45\textwidth]{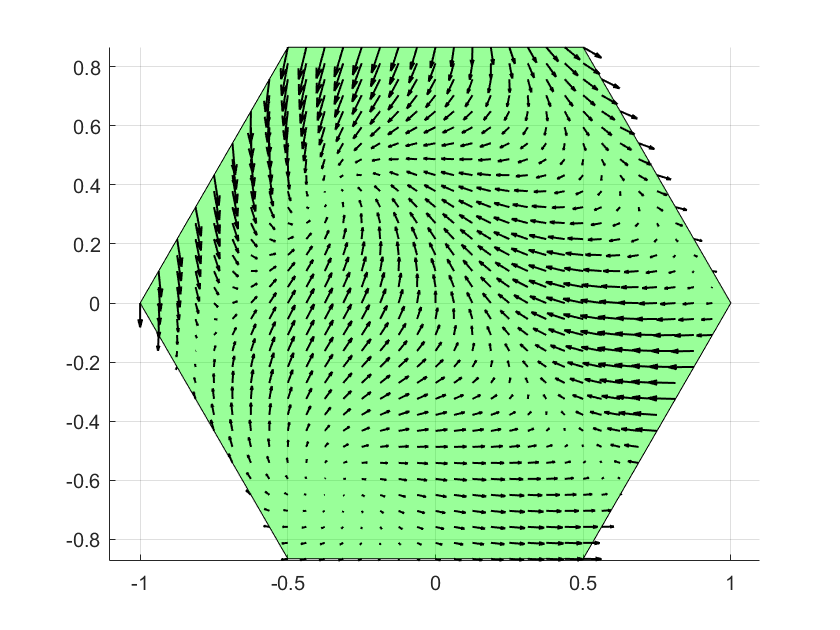} 
  }
  \subfloat[L-shaped domain]
  {
      \includegraphics[width=0.45\textwidth]{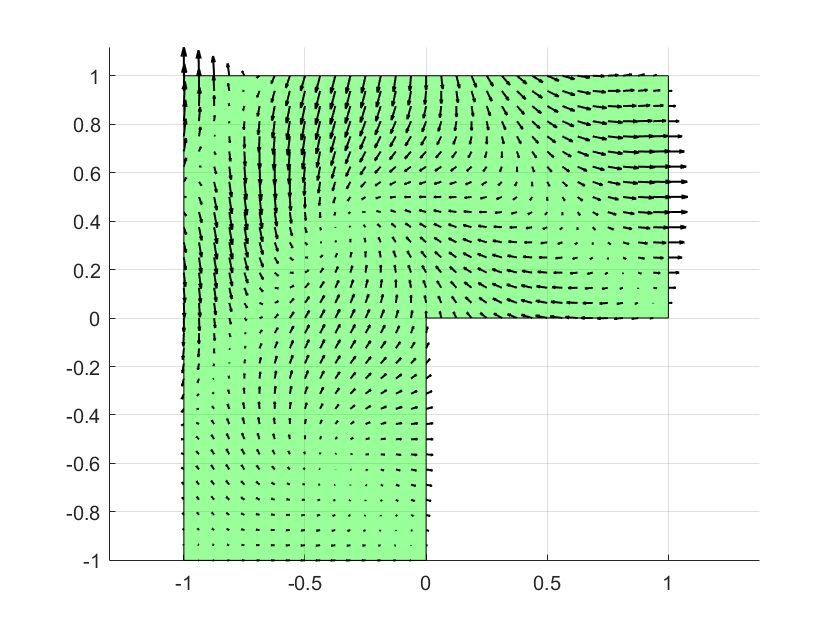} 
  }
  \caption{Plots of $\bm u = [(1-x)(y+1)\cos\pi(x^2+y),e^x\sin\pi(x+y+1)]^T$ on different meshes}\label{Example5_exactsol}
\end{figure}


}

\bibliographystyle{elsarticle-num-names} 
\bibliography{supg}

\end{document}